\documentclass[10pt]{amsart}
\usepackage{amssymb}
\usepackage{mathtools}
\usepackage[foot]{amsaddr}
\usepackage{cite}
\usepackage{enumitem}
\usepackage{a4}

\usepackage{color}
\usepackage[T1]{fontenc}
\usepackage{lmodern}

\newtheorem{theorem}{Theorem}[section]
\newtheorem{lemma}[theorem]{Lemma}

\newtheorem{conj}[theorem]{Conjecture}
\newtheorem{proposition}[theorem]{Proposition}
\newtheorem{property}{\textbf{Property}}
\newtheorem{ass}{\textbf{Assumption}}
\theoremstyle{definition}

\theoremstyle{remark}
\newtheorem{remark}[theorem]{Remark}

\numberwithin{equation}{section}

\def\O{{\Omega}}
\def\Oc{{\O_{\textup{c}}}}

\def\c{{\mathcal{C}}}

\def\I{{\mathcal{I}}}

\def\r{{\mathcal{R}}}

\def\T{{\mathcal{T}}}

\def\L {{\mathcal{L}}}

\def\M{{\mathcal{M}}}
\def\B{{\mathcal{B}}}
\def\MB{\M_{\B}}

\def\R{{\mathbb{R}}}
\def\C{{\mathbb{C}}}

\def\N{{\mathbb{N}}}

\def\upd{{\textup{d}}}
\def\upe{{\textup{e}}}

\def\dom{\operatorname{dom}}
\def\BUC{\mathcal{BUC}}

\newcommand{\revision}[1]{#1}

\DeclareMathOperator{\diag}{diag}
\usepackage{bm}
\newcommand{\vect}[1]{\bm{\mathbf{#1}}}

\def\tilde{\widetilde}
 \title[An abstract framework for structured population models]{An abstract framework for a class of nonlocal structured population models: existence, uniqueness and stability of steady states}
\date\today

\author{J\'{e}rome Coville}
\address[J. C.]{INRAE PACA, Equipe BIOSP, Centre de Recherche d'Avignon, Domaine Saint Paul, Site Agroparc, 84914 Avignon cedex 9, France}
\email{jerome.coville@inrae.fr}

\author{L\'{e}o Girardin}
\address[L. G.]{CNRS, Institut Camille Jordan, Universit\'{e} Claude Bernard Lyon-1, 43 boulevard du 11 novembre 1918, 69622 Villeurbanne Cedex, France}
\email{leo.girardin@math.cnrs.fr}

\begin{document}

\begin{abstract}
This paper is concerned with the study of a class of nonlinear nonlocal functional evolution problems defined in an abstract Banach algebra. We introduce an abstract functional setting that encompasses a wide range of structured population models appearing in biomathematical literature. Within this framework, we analyze the well-posedness of the Cauchy problem and the existence of stationary solutions in the positive cone of the Banach algebra. By reviewing a large number of approaches, we also derive conditions for the local and global stability of these stationary solutions. Additionally, we explore the limits of these conditions by exhibiting explicit counterexamples. In particular, for mutation--selection models with symmetric mutation operators, we uncover both sufficient conditions for existence, uniqueness and stability, and counterexamples to existence or stability.
\end{abstract}

\dedicatory{We dedicate this work to Professor Hiroshi Matano with the deepest respect and admiration.}
\maketitle
\tableofcontents{}

\section{Introduction}\label{s:intro}

The purpose of this article is to study a class of nonlinear functional evolution equations defined on an abstract Banach algebra $E$ and to provide generic results under non-perturbative assumptions. 
This aim is primarily motivated by the observation that discrete and continuous logistic structured population models found in the literature are usually analyzed separately, 
although their analysis shows some similarity due to a common algebraic structure. 
The development of a common framework as well as a generic analysis seems appropriate in this context. 
Abstract functional frameworks are standard for dynamical systems that are monotone, \textit{i.e.} order-preserving
(\textit{cf.}, for instance and among many others, \cite{Amann_1976,Ogiwara_Hilhorst_Matano_2020,Weinberger_1982,Zhao_2017}). Here, on the contrary, 
we will deal with a certain class of non-monotone dynamical systems.

\subsection{An abstract framework}

\subsubsection{The spatial domain}

Let $N,N'\in\N^\star$. Let $\Oc\subset \R^N$ be a nonempty smooth bounded open connected set and $\O=\Oc\times[N']$, where $[N']$ denotes the set $[1,N']\cap\N$. 
Elements of $\O$ are generically denoted $x=(x_{\textup{c}},x_{\textup{d}})$ with the continuous part $x_{\textup{c}}\in\Oc$ and 
the discrete part $x_{\textup{d}}\in[N']$. 
The boundary $\partial\O=\partial\Oc\times[N']$ is denoted $\Gamma$.

When $N'=1$, $\O$ and $\Oc$ are isomorphically identified. Recall that $\BUC([N'],\R)$ can be isomorphically identified with $\R^{N'}$.
Similarly, in what follows, $\BUC(\overline{\O},\R)$ is isomorphically identified with $(\BUC(\overline{\Oc},\R))^{N'}$.

The set $\O$ is equipped with the measure $\mu$ which is the product of the Lebesgue measure on $\Oc$ 
and the counting measure on $[N']$. In particular $\mu(\O)=N'|\Oc|$.

\subsubsection{The \revision{irreducible} Feller semigroup}

We define two (possibly unbounded) linear operators:
\begin{equation*}\setlength\arraycolsep{2pt}
    \begin{matrix}
        P: & \dom(P) & \to & \BUC(\O) \\
        & \varphi & \mapsto & \displaystyle\left[ x\in\O\mapsto \nabla\cdot\left(A\nabla\varphi\right)\left(x\right)+ q\left(x\right)\cdot\nabla\varphi\left(x\right) \right],
    \end{matrix}
\end{equation*}
\begin{equation*}\setlength\arraycolsep{2pt}
    \begin{matrix}
        S: & \dom(S) & \to & \BUC(\O) \\
        & \varphi & \mapsto & \displaystyle\left[ x\in\O\mapsto \int_{\O}[\varphi(y)-\varphi(x)]J(x,y)\upd\mu(y) \right].
    \end{matrix}
\end{equation*}
Here and thereafter,
\begin{equation*}
    \nabla=
    \begin{pmatrix}
        \partial_{x_1} \\
        \vdots \\
        \partial_{x_N}
    \end{pmatrix}
    ,\quad
    A\in \left(\c^1\left(\overline{\Oc},\R^{N\times N}\right)\right)^{N'},\quad
    q\in \left(\revision{\c^1}\left(\overline\Oc,\R^{N}\right)\right)^{N'},
\end{equation*}
and $J$ \revision{and $J^\star:(x,y)\mapsto J(y,x)$ are both} Caratheodory functions:
\begin{equation*}
    \begin{cases}
       \forall x \in \O & J(x, \cdot)\text{ and }J^\star(x,\cdot) \text{ are measurable,} \\
        \text{for a.e. }y\in\O &  J(\cdot, y)\text{ and }J^\star(\cdot,y) \text{ are uniformly continuous}.
   \end{cases}
\end{equation*}
Moreover, $J$ is nonnegative and, 
for any $x_{\textup{d}}\in[N']\cap\N$, $A(\cdot,x_{\textup{d}})$ is symmetric and at least degenerate elliptic:
\begin{equation*}
    \forall x_{\textup{c}}\in\Oc \quad \forall\xi\in\R^{N}\quad \xi\cdot A(x_{\textup{c}},x_{\textup{d}})\xi\geq 0.
\end{equation*}

We will be interested in the (possibly unbounded) linear interior operator $\M=P+S$, 
supplemented with a linear boundary operator $\B = q_{\nu}\frac{\partial}{\partial\nu}+P_{\partial}+S_{\partial}^\Gamma+S_{\partial}^\Oc$,
where $q_{\nu}$ is a nonnegative function on $\Gamma$, $\frac{\partial}{\partial\nu}$ is the outward-pointing normal derivative on $\partial\Oc$,
\begin{equation*}\setlength\arraycolsep{2pt}
    \begin{matrix}
        P_{\partial}: & \dom(P_{\partial}) & \to & \BUC(\Gamma) \\
        & \varphi & \mapsto & \displaystyle\left[ x\in\Gamma\mapsto \nabla_{\partial}\cdot\left(A_{\partial}\nabla_{\partial}\varphi\right)\left(x\right)+ q_{\partial}\left(x\right)\cdot\nabla_{\partial}\varphi\left(x\right) \right],
    \end{matrix}
\end{equation*}
\begin{equation*}\setlength\arraycolsep{2pt}
    \begin{matrix}
        S_{\partial}^\Gamma: & \dom(S_{\partial}^\Gamma) & \to & \BUC(\Gamma) \\
        & \varphi & \mapsto & \displaystyle\left[ x\in\Gamma\mapsto \int_{\Gamma}[\varphi(y)-\varphi(x)]J_{\partial}^\Gamma(x,y)\upd\mu_{\partial}(y) \right],
    \end{matrix}
\end{equation*}
\begin{equation*}\setlength\arraycolsep{2pt}
    \begin{matrix}
        S_{\partial}^\Oc: & \dom(S_{\partial}^\Oc) & \to & \BUC(\Gamma) \\
        & \varphi & \mapsto & \displaystyle\left[ x\in\Gamma\mapsto \int_{\O}[\varphi(y)-\varphi(x)]J_{\partial}^\Oc(x,y)\upd\mu(y) \right].
    \end{matrix}
\end{equation*}
The parameters $q_{\nu}$, $A_{\partial}$, $q_{\partial}$, $J_{\partial}^\Gamma$, $J_{\partial}^\Oc$ satisfy natural assumptions as well: $A_{\partial}$ is of class $\c^1$, symmetric and at least degenerate elliptic,
$q_{\nu}$ and $q_{\partial}$ are \revision{of class $\c^1$}, $J_{\partial}^\Gamma$ and $J_{\partial}^\Oc$ \revision{and the associated $(J_{\partial}^\Gamma)^\star$ and $(J_{\partial}^\Oc)^\star$} are nonnegative Caratheodory functions. The supports of $q_\nu$, $A_{\partial}$, $q_{\partial}$, $J_{\partial}^\Gamma(\cdot,y)$, 
$J_{\partial}^\Oc(\cdot,y)$ might be strict subsets of $\Gamma$.
The notations $\nabla_{\partial}$ and $\mu_{\partial}$ denote as expected the nabla operator on the submanifold $\partial\Oc$ and the product between the $N-1$-dimensional Hausdorff measure and
the counting measure on $[N']$ respectively.
Supplementary conditions on $\B$ are required for well-posedness; these will be implicitly part of
Assumption \ref{ass:compatibility} below. 

This framework includes Neumann boundary conditions, oblique derivative boundary conditions, nonlocal boundary
conditions (\textit{e.g.} for renewal equations) and Ventcel--Vi\v{s}ik second-order boundary conditions. It
is also possible that $\B=0$, in which case there are simply no boundary conditions.
Although this is already quite general, we note that the exact same analysis could be performed with 
$\B$ a periodic boundary operator (namely, by replacing $\Oc$ with the torus $\mathbb{T}^{N-1}$).
Nevertheless, the homogeneity of $\B$, chosen to ensure $\B 1=0$, forbids Dirichlet or Robin boundary 
conditions. Some of our methods and results could be adapted for such boundary conditions
but this is not immediate and these cases are therefore deliberately kept apart.


To summarize, for any $u\in\dom(\M)$, we consider the operator $\MB=(\M,\B)$ defined by:
\begin{equation}\label{definition_M}\tag{Fel}
    \MB u = \begin{dcases}
        Pu+Su & \text{in }\O, \\
        q_{\nu}\frac{\partial u}{\partial\nu}+P_{\partial}u+S_{\partial}^\Gamma u+S_{\partial}^\Oc u & \text{on }\Gamma.
        \end{dcases}
\end{equation}
In the paper, when the context is unambiguous or when the boundary operator $\B$ is trivial, $\M$ might
be identified with $\MB$.

Let $E\subset\BUC(\overline{\O},\R)$ be a unital commutative Banach algebra
\footnote{A unital commutative Banach algebra over the field $\R$ is a Banach space over $\R$ equipped with a 
multiplication operation such that the space is also a unital commutative associative algebra over $\R$ and
such that the norm is sub-multiplicative. Typical examples include $\R^{N'}$ or $\C^{N'}$ equipped with the 
Hadamard product and the norm $\|\cdot\|_{\infty}$, the space $\textup{Lip}(\overline{\Oc},\R)$ of Lipschitz
continuous functions equipped with the standard multiplication of functions (up to multiplication of the norm 
$\|\cdot\|_{\textup{Lip}}$ by a universal positive constant), the H\"{o}lder spaces 
$\c^\alpha(\overline{\O_c},\R)$ of H\"{o}lder continuous functions with exponent $\alpha\in(0,1)$ equipped 
with the same multiplication of functions (again up to a multiplication of the norm).

Being careful with the meaning given to products and ratios of functions, it is possible to extend our results
to more general Banach spaces. Since all applications we have in mind are in fact set in Banach algebras, 
we only consider this more convenient functional framework.}
continuously embedded in $\BUC(\overline{\O},\R)$ that satisfies the following \revision{standing} assumption.

\begin{ass}[Compatibility of $E$ and $\MB$]\label{ass:compatibility}
    The Banach algebra $E$ is compatible with the operator $\MB$ in the following sense:
    \begin{equation*}
        \overline{\dom_E(\MB)}=E,\quad\MB(\dom_E(\MB))\subset E.
    \end{equation*}
    where $\dom_E(\MB)=\dom(\MB)\cap E$.
\end{ass}

We emphasize right now, to avoid any confusion, that $E$ might be much smaller than $\BUC(\overline{\O},\R)$. 
In some applications it will be $\R^{N'}$ (\textit{cf.} Assumption \ref{ass:reduction_to_ODEs} below).

The canonical ordering $\leq$ associated with the nonnegative cone $\BUC(\overline{\O},\R^+)$ defines the nonnegative cone 
$E^+=\BUC(\overline{\O},\R^+)\cap E$ of $E$. 
We assume from now on that the cone $E^+$ has a nonempty interior and that it contains the constant function $1=1_E$. 
The duality product between $E$ and its topological dual $E'$ is denoted frow now on $\langle\cdot,\cdot\rangle_E$. 
Since $E\subset\BUC(\overline{\O},\R)$, any function in $\operatorname{int}(E^+)$
admits a positive infimum in $\O$ and the inverse of such a function is again in $\operatorname{int}(E^+)$.

Since $\O$ is bounded, $\BUC(\O,\R)\subset L^2(\O,\R)$. 
Therefore $E$ inherits the canonical scalar product $\left(\cdot|\cdot\right)_{L^2}$ of $L^2(\O,\R)$. 
Equipped with this scalar product, $E$ is an inner product space.
In general, the canonical norm $\|\cdot\|_E$ differs from, and is not equivalent to, the norm $\|\cdot\|_{L^2(\O)}$, 
and so $E$ is in general not a Hilbert space.

The $L^2$ adjoint of the densely defined operator $\MB$ is denoted $\MB^\star=(\M,\B)^\star$. 
An explicit formula for $\MB^\star$ could be deduced from the
explicit form of $\MB$ presented in \eqref{definition_M} but we will have no use for it.

\revision{
All along this article, $\MB$ satisfies in addition the following standing assumption.
\begin{ass}[Dense range]\label{ass:dense_range}
    There exists $C>0$ such that
    \begin{equation*}
        \overline{(-\M +C, \B)(\dom_E(\MB))}=E.
    \end{equation*}
\end{ass}}

\revision{By construction, $\MB$ satisfies the positive maximum principle, in the following sense: 
for any $x\in\O$ such that $u(x)=\max u\geq 0$, $\MB u(x)\leq 0$. 
According to the Hille--Yosida--Ray theorem \cite{Bony_Courrege_Priouret_1968}, the
weak maximum principle together with Assumptions \ref{ass:compatibility} and \ref{ass:dense_range} implies the following property.}
\begin{property}[Feller semigroup]\label{property:strong_continuity}
   $\MB$ is the infinitesimal generator of a \revision{Feller} semigroup $(\T_t )_{t\geq 0}$ in $E$.
\end{property}
 
From \cite{Arendt1987,Kato1984,Bony_Courrege_Priouret_1966,Bony_Courrege_Priouret_1968}, 
the form of $\MB$ presented in \eqref{definition_M} is roughly speaking generic \revision{among Feller semigroups, \textit{i.e.}
among strongly continuous semigroups whose generators satisfy the maximum principle}.
In any case, this form covers all applications we have in mind (see below). 

\revision{We are now in a position to introduce another standing assumption.}

\begin{ass}[Irreducibility \cite{Schaefer_1974}]\label{ass:irreducibility}
    For any $u_0\in \dom_E(\MB)\cap E^+\setminus\{0\}$ and $\psi\in (E')^{+}\setminus\{0\}$, there exists $t\in(0,+\infty)$ such that:
    \begin{equation*}
        \langle\psi,\T_t u_0\rangle_E>0.
    \end{equation*}
\end{ass}

From now on, for brevity, the set $\dom_E(\MB)\cap E^+$ is denoted $\dom_{E^+}(\MB)$.

\revision{Note that Assumption \ref{ass:irreducibility} restricts significantly the possible choices of operators $\MB$. 
For instance, if $S=0$
and $\B=\partial_\nu$, then $P$ has to be uniformly elliptic. More generally, the irreducibility implies the following property.} 

\begin{property}[Strongly positive resolvent operator]\label{property:strongly_positive_resolvent}
    There exists $C\in(0,+\infty)$ such that $(-\M +C,\B)^{-1}$ is continuous and strongly positive, 
    \textit{i.e.} it maps $E^+\setminus\{0\}$ into $\operatorname{int}(E^+)$.
\end{property}

Throughout the paper, when we write $\varphi=(-\M+C,\B)^{-1}f\in\operatorname{int}(E^+)$ with $f\in E^+\setminus\{0\}$,
we mean:
\begin{equation*}
    \begin{dcases}
        f\in E^+\setminus\{0\}, & \\
        \varphi\in\operatorname{int}(E^+), & \\
        (-\M+C)\varphi = f & \text{in }\O, \\
        \B\varphi = 0 & \text{on }\Gamma.
    \end{dcases}
\end{equation*}
Similar notations will be used for other operators built upon $\MB$. 

Property \ref{property:strongly_positive_resolvent} holds for all sufficiently large $C$ if and only if it holds 
for one value. The adjoint $\MB^\star$ is also a strongly positive resolvent operator.

Finally, by construction, $\MB$ satisfies the following property.

\begin{property}[A principal eigenfunction]\label{property:one_in_kernel}
$1_E\in \ker(\MB)$ and $1_E\in\dom(\MB^\star)$.
\end{property}

\subsubsection{The abstract semilinear nonlocal Cauchy problem}

Let $p\geq 1$ and $r,b,c\in\operatorname{int}(E^+)$. Let
\begin{equation}
    \begin{matrix}
        \r: & E\times\O & \to & \R \\
        & (v,x) & \mapsto & \displaystyle r(x)-b(x)\int_{\O}v(y)^p c(y)\upd\mu(y).
    \end{matrix}
\end{equation}

We are interested in the solution $u\in\mathcal{C}^1((0,+\infty),E)\cap\mathcal{C}([0,+\infty),\dom_{E^+}(\MB))$ of, for a given $u_0\in\dom_{E^+}(\MB)$:
\begin{equation}\label{general_system}\tag{Abs}
    \begin{dcases}
        \partial_t u(t,x) = \M  u(t,x) + u(t,x)\r(u(t),x)  & \text{for all }(t,x)\in(0,+\infty)\times\O, \\
        \B u(t,x) = 0 & \text{for all }(t,x)\in(0,+\infty)\times \Gamma, \\
        u(0,x) = u_0(x) & \text{for all }x\in\overline{\O}.
    \end{dcases}
\end{equation}
Above, the notation $u(t)$ denotes the function $x\mapsto u(t,x)$. This notation will be used repeatedly in the paper.
The context will make this notation unambiguous.

\subsubsection{\revision{Reduction to ODEs}}
The following assumption is not a standing assumption but will be mobilized repeatedly thereafter because of its special consequence, detailed below.

\begin{ass}[\revision{A possible partial instantiation}]\label{ass:reduction_to_ODEs}
    $E=\R^{N'}$, $|\Oc|=1$, $\B=0$.
\end{ass}

\begin{property}[Reduction to ordinary differential equations]
    Assume Assumption \ref{ass:reduction_to_ODEs} holds true.
    
    Then the scalar product $(\cdot|\cdot)_{L^2}$ restricted to $E\times E$ coincides with the canonical scalar product of $E$, 
    $\M$ is identified with $0$ if $N'=1$ or else to an essentially nonnegative irreducible square matrix $\mathbf{M}\in\mathbb{M}_{N'}(\R)$
    whose kernel contains the column vector $\mathbf{1}=(1,1,\dots,1)^{\textup{T}}$,
    $u_0$, $r$, $b$, $c$ are respectively identified with column vectors $\vect{u}_0$, $\vect{r}$, $\vect{b}$, $\vect{c}$ in $\R^{N'}$, and
    the solution $u$ of \eqref{general_system} is identified with a vector $\vect{u}\in\c^1([0,+\infty),\R^{N'})$ solution of the following 
    system of ordinary differential equations:
    \begin{equation}\label{discrete_ode_system}\tag{ODE}
        \begin{dcases}
            \dot{\vect{u}}=\mathbf{M}\vect{u}+\left(\diag(\vect{r})-\left(\sum_{i=1}^{N'}  c _i u^p_i\right)\diag(\vect{b})\right)\vect{u}, \\
            \vect{u}(0)=\vect{u}_0.
        \end{dcases}
    \end{equation}
\end{property}

\subsubsection{The underlying Hilbert space}

In the paper, we will repeatedly refer to the underlying Hilbert space, $H$, equipped with the inner product $(\cdot|\cdot)_{H}$:
\begin{equation}\label{def:H}
    H=\begin{cases}
        \C^{N'} & \text{if Assumption \ref{ass:reduction_to_ODEs} holds true}, \\
        L^2(\O,\C) & \text{otherwise.}
    \end{cases}
\end{equation}
This is the Hilbert space in which it is most convenient to conduct spectral analysis. 
The inner product of $H$ always coincides the inner product of $L^2(\O,\C)$. From now on, it will be denoted $\left(\cdot|\cdot\right)$ for brevity.

\subsection{Scope of the paper}
In this paper, we are interested in the positive equilibria of \eqref{general_system}, more precisely in their
existence, uniqueness and stability properties. 

When Assumption \ref{ass:reduction_to_ODEs} holds true and $N'=1$, then \eqref{general_system} reduces to the standard logistic 
ordinary differential equation. 
This case is of course completely understood: the positive equilibrium is unique and globally asymptotically stable. 
As is well known, this can be understood as a consequence of the comparison principle.

However, when $N'>1$ or Assumption \ref{ass:reduction_to_ODEs} fails, the equation is truly nonlocal and the comparison principle breaks.
Without comparison principle, much remains to be understood. The contribution of this paper is a significant step
forward. On one hand, we identify wide non-perturbative parameter regimes where the equilibrium remains unique and locally 
asymptotically stable, and we also identify significant parameter regimes for global asymptotic stability. 
On the other hand, we show by means of counterexamples that these parameter regimes cannot be extended 
indefinitely and that many interesting open questions subsist.

\subsection{Examples from the biomathematical literature}
As mentioned above, a large class of models for structured populations found in the existing 
literature can be reformulated as \eqref{general_system} by choosing appropriately the parameters. 

When Assumption \ref{ass:reduction_to_ODEs} is satisfied and in addition $p=1$ and $c=1$, then the nonlocal functional equation 
\eqref{general_system} reduces to the well-known Lotka-Volterra competition system with mutations, widely studied in the literature \cite{Bates2011,Burger1994,Buerger2000,Cai2015,Crow1970,Hofbauer1985}.
Without the restriction $p=1$ and with a symmetric $\mathbf{M}$, \eqref{discrete_ode_system} was studied in \cite{Coville_Fabre_}.

The same class of systems appears in age-structured population models or size-structured population models. 
In particular, recently, the system \eqref{discrete_ode_system} with $N'>2$ and $p=1$ appeared in \cite{Ress_Traulsen_Pichugin_2022} 
for a size-structured cell population model. 
There, the matrix $\mathbf{M}$ was far from being symmetric and it was more reminiscent of a Frobenius companion matrix (or Leslie matrix in ecology).

Similarly, when $N'=1$, $\O\simeq\Oc$, $E=\mathcal{C}^\alpha(\overline{\O})$ for some 
H\"{o}lder exponent $\alpha\in(0,1)$ and $\MB$ is the Laplace operator with Neumann boundary conditions 
$\Delta_N$, the equation \eqref{general_system} becomes the following reaction--diffusion problem:
\begin{equation*}
    \begin{dcases}
        \partial_t u(t,x) = \Delta u(t,x) + u(t,x)\r(u(t),x) & t>0,\,x\in\O, \\
        \r(v,x)= r(x)-b(x)\int_{\O}v(y)^p c(y)\upd y & v\in \mathcal{C}^\alpha(\overline{\O}),\ x\in\O, \\
        \partial_\nu u(t,x) = 0 & t>0,\,x \in \Gamma, \\
        u(0,x) = u_0(x) & x\in\overline{\O}.
   \end{dcases}
\end{equation*}
\revision{The domain of $\MB$ is $\{v\in E\ |\ v\in\c^2(\O)\cap\c^1(\overline{\O}),\ \partial_\nu v=0\text{ on }\partial\O\}$, it is dense in $E$,
and the range of $-\MB+C$ for some large $C$ is $E$ itself.}
This problem was studied in \cite{Jabin_Liu_2017} for $p=2$ and $b=c$. With $p=1$ and $c=1$, it reduces to a well-known problem
modeling a population subject to growth, mutation and competition and where the strength of the competition depends of 
each trait, \textit{cf.} for instance \cite{Arnold2012,Coville2013-prep,Desvillettes2008,Perthame2007,Prevost2004}. 

Nonlocal diffusion problems can also be considered by taking $N'=1$, $\O\simeq\Oc$, 
$E=\mathcal{W}^{1,\infty}(\O)\simeq\textup{Lip}(\overline{\O})$ and $\MB$ the generator of a jump process, \textit{e.g.}:
\begin{equation*}
    \begin{dcases}
        \partial_t u(t,x) = \int_{\O}(u(t,y)- u(t,x))J(x,y)\upd y + u(t,x)\r(u(t),x) & t>0,\ x\in\O, \\
        \r(v,x) = r(x)-b(x)\int_{\O}v(y)^p \upd y & v\in \textup{Lip}(\overline{\O}),\ x\in\O, \\
        u(0,x) = u_0(x) &x\in\overline{\O}.
   \end{dcases}
\end{equation*}
\revision{Here the domain is $E$ itself and the range of $-\MB+C$ is dense for sufficiently large $C$.}
This model is also well studied in quantitative genetics, \textit{cf.} \cite{Bonnefon2017,Buerger2000,Crow1970,Coville2013,Gil2017,Kingman1978,Calsina2007,Calsina2005}. 

Last, with $N'\geq 2$ and without Assumption \ref{ass:reduction_to_ODEs}, we can obtain for instance a system of the following form:
\begin{equation*}
    \begin{dcases}
        \partial_t\vect{u}(t,x) = \left[\mathbf{M}(x) + \diag(\vect{d}\Delta+\r(\vect{u}(t),x))\right]\vect{u}(t,x)& t>0,\ x\in\Oc, \\
        \r(\vect{v},x) = \vect{r}(x)-\left(\sum_{i=1}^{N'} \int_{\Oc} v_i(y)^p c_i(y)\upd y\right)\vect{b}(x) & \vect{v}\in(\c^{\alpha}(\overline{\Oc}))^{N'},\ x\in\Oc, \\
        \partial_\nu \vect{u}(t,x) = \vect{0} & t>0,\ x\in\partial\Oc, \\
        \vect{u}(0,x)=\vect{u}_0(x) & x\in\overline{\Oc}.
  \end{dcases}
\end{equation*}
Such models are related to so-called Fisher--KPP systems or mutation--diffusion--selection models 
\cite{Dockery_1998,Girardin_2017,Leman_Meleard_,Cantrell_Cosner_Yu_2018,Gueguezo_Doumate_Salako_2024,Mirrahimi_2016,Griette_Matano_2025}. 
The combination of both interphenotypic competition and nonlocal competition in space is not standard in the 
literature but could very reasonably be considered in a variety of situations.
It would also be natural to replace certain Laplacian diffusion operators $d_i\Delta$ by nonlocal diffusion operators $d_i(J\star u-u)$, 
or even by $0$, in order to model how different phenotypes have different dispersal modes.
As long as at least one phenotype keeps a non-degenerate local or nonlocal dispersal term, the whole operator $\MB$ remains irreducible.

\subsection{\revision{Understanding the limits of the abstract framework through the example of the Gurtin--MacCamy equation}}

\revision{When $N'=1$, $\O\simeq\Oc=(0,a^\dagger)$, $-\M$ is a first-order derivative, 
$\B=S_{\partial}^\Oc$ supported on $\{0\}$, $r=1$ and for the sake of this example the spatial variable $x$ is rather denoted $a$, 
we obtain formally:
\begin{equation*}
    \begin{dcases}
        \partial_t u(t,a) + \partial_a u(t,a) = u(t,a)\r(u(t),a) & t>0,\ a\in(0,a^\dagger) \\
        \r(v,a) = 1-b(a)\int_{0}^{a^\dagger}v(a)c(a)\upd a &  \\
        u(t,0) = \int_0^{a^\dagger}\frac{\beta(a)}{\|\beta\|_{L^1}}u(a)\upd a & t>0, \\
        u(0,a) = u_0(a) & x\in[0,a^\dagger].
    \end{dcases}
\end{equation*}
By multiplying by $\upe^{-a}$ and setting $v(t,a)=\upe^{-a}u(t,a)$, $f(a)=\upe^a c(a)$ and $\gamma(a)=\upe^a \beta(a)/\|\beta\|_{L^1}$,
we recover the classical logistic Gurtin--MacCamy equation for age-structured populations \cite{Webb_1985,Gurtin_MacCamy}:
\begin{equation*}
    \begin{dcases}
        \partial_t v(t,a) + \partial_a v(t,a) = -v(t,a)b(a)\int_{0}^{a^\dagger}v(t,a)f(a)\upd a & t>0,\ a\in(0,a^\dagger) \\
        v(t,0) = \int_0^{a^\dagger}\gamma(a)v(t,a)\upd a & t>0, \\
        v(0,a) = v_0(a) & a\in[0,a^\dagger].
    \end{dcases}
\end{equation*}
However, and most importantly, Assumptions \ref{ass:compatibility} and \ref{ass:dense_range} have not been verified. As a matter of fact, 
when trying to define $E$, we find (classically) that $\MB$ generates an irreducible Feller semigroup on continuous functions
if this set of continuous functions accounts for the nonlocal boundary condition, \textit{e.g.}
\begin{equation*}
    E=\left\{ u\in\BUC([0,a^\dagger])\ \middle|\ u(0)=\int_0^{a^\dagger}\frac{\beta(a)}{\|\beta\|_{L^1}}u(a)\upd a \right\}.
\end{equation*}
Unfortunately, such a choice of $E$ is fine when working on Banach spaces, but not when working on Banach algebras
($u\in \operatorname{int}(E^+)$ does not imply $1/u\in E$). 
Although certain results of the paper remain true when working on Banach spaces that are not Banach algebras, 
we believe the Banach algebra structure is not a mere technical assumption. It is at least very useful to make 
sense of important forthcoming stability properties.}

\revision{A possibly tempting way of circumventing this obstacle is to replace all along the paper the notion of $\c^0$-semigroup 
by the relaxed notion of integrated semigroups, that does not require densely defined generators. 
More precisely, rewrite the system as in \cite[Chapter 8]{Magal_Ruan_2018}:
\begin{equation*}
    \partial_t \begin{pmatrix}0\\u\end{pmatrix} = A\begin{pmatrix}0\\ u\end{pmatrix}+F\left(\begin{pmatrix}0\\ u\end{pmatrix}\right),
\end{equation*}
with
\begin{equation*}
    A\begin{pmatrix}0\\ u\end{pmatrix}=\begin{pmatrix}-u(0)\\-u'(0)\end{pmatrix},\quad 
    F\left(\begin{pmatrix}0\\ u\end{pmatrix}\right)=
    \begin{pmatrix}
        \int_0^{a^\dagger}\frac{\beta(a)}{\|\beta\|_{L^1}}u(a)\upd a \\
        u(a)\r(u,a)
    \end{pmatrix}.
\end{equation*}
The underlying Banach space being $\R\times L^1((0,a^\dagger))$, the generator $A$ is obviously nondensely defined. 
Yet, following \cite{Magal_Ruan_2018}, the well-posedness of the differential equation can be established with integrated semigroups.
Moreover, it is possible to restrict the underlying space to the Banach algebra $E=\R\times\BUC([0,a^\dagger])$.
In the present paper, for the sake of brevity and clarity, we limit ourselves to standard Feller semigroups. 
We leave the generalization to integrated semigroups to others. We conjecture that a large
part of what follows can indeed be applied to the logistic Gurtin--MacCamy equation and even to the diffusive
logistic Gurtin--MacCamy equation.}

\subsection{A short review on known stability results}
Below, instead of vainly attempting to give an exhaustive review of the wide literature where these models have 
been mathematically investigated, we focus on the most significant results, that inspire our work. This selected state
of the art will in particular motivate our special attention to the case $b\notin \R 1_E$.

\subsubsection{When $N'=2$ and Assumption \ref{ass:reduction_to_ODEs} holds}

In the celebrated article \cite{Dockery_1998}, among many other results, the global asymptotic stability of the unique
equilibrium was proved in the special case $p=b=c=1$. 

Twenty years later, the two-dimensional case with $p=1$ and without restrictions on $b$ or $c$ was entirely settled.
Indeed, \cite[Theorem 3.8]{Cantrell_Cosner_Yu_2018} showed, without conditions on $b$ or $c$,
that the positive equilibrium is globally asymptotically stable. 
The proof relies upon the Bendixson--Dulac and Poincaré--Bendixson theorems. 
These two theorems are specific to planar dynamical systems and do not apply with $N'\geq 3$.

When $b=1$ and $p\geq 1$, with no restriction on $c$, the global asymptotic stablity of the unique equilibrium 
was confirmed by the second author as part of a larger result \cite[Theorem 1.3]{Girardin_2017}.

\subsubsection{When $N'\in[3,+\infty)$ and Assumption \ref{ass:reduction_to_ODEs} holds}

The aforementioned result \cite[Theorem 1.3]{Girardin_2017} actually covers the full parameter range $N'\geq 2$, $p\geq 1$ and $b=1$.

When $b\notin \R 1_E$, the existence and uniqueness of the positive equilibrium were known for a long time, but its local 
or global stability remained elusive, especially in the case of a symmetric matrix $\vect{M}$.

\subsubsection{When $N'=1$ and Assumption \ref{ass:reduction_to_ODEs} does not hold}

When $-\MB$ satisfies a sufficiently strong spectral theorem (for instance, it is the Neumann Laplacian, whose resolvent 
is self-adjoint on $L^2(\O,\R)$ and compact) and $b=1$, the methods used by the second author in 
\cite[Theorem 1.3]{Girardin_2017} can be adapted. 
For the sake of completeness, we will detail this proof below. 

When $\MB$ is line-sum-symmetric 
\cite{Cantrell_Cosner_Lou_2012,Cantrell_Cosner_Lou_Ryan_2012,Eaves_1985,Girardin_Griette_2020}, 
$b=c$ and $p=1$, another method of proof was used previously for the nonlocal Fisher--KPP equation 
\cite{Berestycki_Nadin_Perthame_Ryzhik}. The nonlocal Fisher--KPP equation is closely related to \eqref{general_system},
with the scalar product $\left(c|u\right)$ replaced by a spatial convolution $c\star u$ and with the bounded domain 
$\Oc$ replaced by $\R^N$.

As observed in \cite{Jabin_Liu_2017}, a special gradient flow structure appears when $-\MB$ is the Neumann
Laplacian, $p=2$ and $b=c$. This gradient flow structure makes it possible to prove 
the global asymptotic stability of the positive equilibrium. We will present an adaptation of this proof, 
where the Neumann Laplacian is replaced by a more general self-adjoint operator with compact resolvent. 

When $p>2$, to the best of our knowledge, the global stability is an entirely open problem.


When $-\MB$ is a nonlocal diffusion operator, the first author has shown in a series
of works \cite{Bonnefon2017,Coville2013} that even the existence of the equilibrium can fail. 
We will recall in this paper how such a counterexample can be constructed.

\subsubsection{When $N'>1$ and Assumption \ref{ass:reduction_to_ODEs} does not hold}

To the best of our knowledge, this problem is entirely open. However a closely related problem with nonlocal
competition only in the continuous variable $x_{\textup{c}}$ or only in the discrete variable $x_{\textup{d}}$
has been studied in several papers (\textit{cf.} 
\cite{Dockery_1998,Girardin_2017,Leman_Meleard_,Cantrell_Cosner_Yu_2018,Gueguezo_Doumate_Salako_2024,Mirrahimi_2016,Griette_Matano_2025}
and references therein).

\subsection{Mild solution, well-posedness and pointwise nonnegativity}

The data of the problem being regular enough, and by Property \ref{property:strong_continuity}, for $u_0\in\dom_E(\MB)$, the Cauchy problem 
\eqref{general_system} possesses a global \emph{mild} solution. This is a weak solution 
$u\in\mathcal{C}([0,+\infty),E)$, explicitly given by the Duhamel principle using the semigroup 
$(\T_t)_{t\geq 0}$. Formally, a mild solution $u$ of \eqref{general_system} is a solution to the following problem:
\begin{equation*}
u(t)(x) = \T_t u_0(x) + \int_0^t \T_{\revision{t-t'}}u(t',x)\r(u(t'),x) \upd t',
\end{equation*}
\textit{cf.}, for instance, the construction in \cite{Cabre2013, Henry_1981, Pazy1983}. 
By Lipschitz continuity and a standard fixed point argument, there exists a mild solution $u$ defined in $(0,T)$ for some maximal time $T>0$. 

For any $T'\in(0,T)$, $u$ is bounded in $L^\infty((0,T'),E)$, and so is $a:(t,x)\mapsto\r(u(t),x)$.
It follows from the comparison principle applied to the semigroup generated by $(\M +a,\B)$ that, if
$u_0\in\dom_{E^+}(\MB)$, then $u(t)\in E^+$ for all $t\in(0,T')$. Subsequently, $a(t,x)\leq r(x)$ and
$\partial_t u\leq \M  u+ru$, and it follows again from the comparison principle that, 
for some $C>0$ independent of $T'$, $0\leq u(t,x)\leq C\upe^{CT'}1_E$ for all $(t,x)\in(0,T')\times\O$. 
Passing by continuity to the limit $T'\to T$, it follows that $\sup_{t\in(0,T)}\|u(t)\|_{E}<+\infty$.

Hence any solution of \eqref{general_system} with initial condition $u_0\in\dom_{E^+}(\MB)$ admits $T=+\infty$ as maximal time of existence
\cite[Theorem 3.3.4]{Henry_1981} and is nonnegative at all $(t,x)\in[0,+\infty)\times\overline{\O}$.

Other standard arguments \cite[Chapter 6, Theorem 1.5]{Pazy1983} show that the mild solution $u$ with initial condition
$u_0\in E^+$ is a \textit{strong} solution, \textit{i.e.} $u\in\mathcal{C}^1((0,+\infty),E)\cap\mathcal{C}([0,+\infty),\dom_{E^+}(\MB))$.

We will show later on that the global boundedness in space-time is a much trickier issue. Despite common knowledge on
equations of Fisher--KPP type, here it might actually fail, due to the interplay between local and nonlocal operators.



\subsection{Principal spectral theory with boundary conditions}

Next, let us recall some useful spectral quantities attached to $\M $. We consider:
\begin{equation*}
    \lambda_1(-\MB)=\sup\left\{\lambda\in\R \ \middle|\ \exists \varphi \in \dom_{E^+}(\MB)\setminus\{0\}\quad 
    \begin{dcases}
        -\M\varphi \geq \lambda\varphi & \text{in }\O \\
        \B\varphi = 0 & \text{on }\Gamma
    \end{dcases}
    \right\}.
\end{equation*}
The set whose supremum is considered is nonempty; indeed, $\varphi=1_E\in\ker(\MB)$ implies that $0$ belongs to this set. Hence $\lambda_1(-\MB)\geq 0$. 
To show that it is actually zero, we assume on the contrary the existence of 
$\lambda>0$, $\varphi\in \dom_{E^+}(\MB)\setminus\{0\}$ such that
$-\M\varphi\geq\lambda\varphi$ and $\B\varphi= 0$. Let $f=-\M\varphi-\lambda\varphi$ in $\O$ and extend
it by continuity on $\overline{\O}$. By Assumption \ref{ass:compatibility}, $f\in E$, and by assumption on 
$\varphi$, $f\in E^+$. Moreover, $f+(C+\lambda)\varphi\in E^+\setminus\{0\}$.
By applying the resolvent $(-\M +C,\B)^{-1}$ with $C>0$ given by Property \ref{property:strongly_positive_resolvent},
we deduce $\varphi\in\operatorname{int}(E^+)$. 
Subsequently $(f+(C+\lambda)\varphi,\B\varphi-g)\in\operatorname{int}(E^+)$.
Since $E$ is a Banach algebra, for any $\varepsilon>0$ sufficiently small, $v_\varepsilon = f+(C+\lambda)\varphi-\varepsilon 1_E\in \operatorname{int}(E^+)$.
Since $(-\M+C,\B)^{-1} 1_E = \frac{1}{C}1_E$,
\begin{equation*}
    (-\M+C,\B)^{-1}v_\varepsilon = \varphi-\frac{\varepsilon}{C}1_E,
\end{equation*}
and by strong positivity of the resolvent again, $\varphi-\frac{\varepsilon}{C}1_E\in\operatorname{int}(E^+)$.
Let 
\begin{equation*}
    \varepsilon^\star = \sup\left\{ \varepsilon>0\ \middle|\ v_\varepsilon\in \operatorname{int}(E^+)\right\}.
\end{equation*}
Then by continuity $v_{\varepsilon^\star}\in\partial E^+$. If $v_{\varepsilon^\star}=0$, then $\varphi-\frac{\varepsilon^\star}{C}1_E=0$,
and then $-\M\varphi = -\frac{\varepsilon^\star}{C}\M 1_E = 0$ contradicts $-\M\varphi\geq\lambda\varphi$ due to $\lambda>0$.
Thus $v_{\varepsilon^\star}\in\partial E^+\setminus\{0\}$. Since $v_{\varepsilon^\star}$ is still in the domain of
$(-\M+C,\B)^{-1}$, by strong positivity, $\varphi-\frac{\varepsilon^\star}{C}1_E\in\operatorname{int}(E^+)$. But now
$\left(f+\lambda\varphi+C\left(\varphi-\frac{\varepsilon^\star}{C}1_E\right),\B\varphi-g\right)\in\operatorname{int}(E^+)$ and this contradicts
the optimality of $\varepsilon^\star$, and subsequently the existence of $\lambda>0$.

Hence $\lambda_1(-\MB)=0$. In the literature, the quantity $\lambda_1(-\MB)$ is often referred to as \textit{the principal eigenvalue} of the operator $-\MB$. 

A similar comparison argument using the strong positivity of the resolvent shows that $\ker(\MB)=\R 1_E$, namely $0$ is a simple eigenvalue.

Similarly, by $L^2$ duality, we define \textit{the adjoint principal eigenvalue}:
\begin{equation*}
    \lambda_1(-\MB^{\star})=\sup\left\{\lambda\in\R\ \middle|\ \exists\psi\in \operatorname{int}(E^{+})\quad\revision{\inf_{\varphi\in \dom_{E^+}(\MB)}}\left(\psi|-\M\varphi-\lambda\varphi\right)\geq 0   \right\}.
\end{equation*}
By definition of $\lambda_1(-\MB)=0$, for any $\varepsilon>0$, there exists $\lambda\in(-\varepsilon,0)$ and 
$\varphi\in\dom_{E^+}(\MB)$ such that $-\M\varphi-\lambda\varphi\geq 0$ and $\B\varphi\geq 0$,
whence, by scalar product with $\psi=1_E$, $\lambda_1(-\MB^{\star})\geq 0$.
Assuming now by contradiction that $\lambda_1(-\MB^{\star})>0$, we set $\lambda=\frac12\lambda_1(-\MB^{\star})$ and, by 
testing against $\varphi=1_E$, we deduce
the existence of $\psi\in\operatorname{int}(E^+)$ such that $(\psi|\lambda 1_E)\leq 0$. This is an obvious contradiction. 
Hence $\lambda_1(-\MB^{\star})=\lambda_1(-\MB)=0$. Note that the existence of an adjoint principal eigenfunction is not established.

It is important to note that the characterization of $\lambda_1(-\MB)$ as an eigenvalue does not extend to 
$\lambda_1(-\M+a,\B)$ with arbitrary $a\in E$. For general $a\in E$, it can be verified that
\begin{equation*}
\lambda_1(-\M+a,\B)=\left\{\lambda\in\R \ \middle|\ \exists \varphi \in \dom_{E^+}(\MB)\setminus\{0\}\quad 
    \begin{dcases}
        (-\M +a)\varphi\geq\lambda\varphi & \text{in }\O \\
        \B\varphi = 0 & \text{on }\Gamma
    \end{dcases}
    \right\},
\end{equation*}
is still well defined and finite, nevertheless the existence of an eigenfunction 
$\varphi\in\operatorname{int}(E^+)$ such that $(-\M+a)\varphi=\lambda_1(-\M+a)\varphi$ might fail. We will exploit
such counterexamples in the article. In such cases, $\lambda_1$ is referred to as the \textit{principal point spectrum} instead.

We also emphasize, and this is in fact related to the preceding remark, that we assume no specific compactness property 
on $\MB$. The Krein--Rutman theorem is typically unavailable. 
This is necessary for the abstract framework to contain nonlocal diffusion.

\subsection{Main results}

Let us now state our main results.

The first result is concerned with uniform bounds.

\begin{theorem}[A priori bounds]\label{thm:uniform_bound_dynamical_system}
   Any solution $u$ of \eqref{general_system} is bounded in \\ $L^\infty((0,+\infty),L^1(\O,\R))$ by a constant that only depends on:
   \begin{equation*}
       \|u_0\|_{L^1},\ p,\ \mu(\O),\ \|r\|_{L^\infty},\ \|1/b\|_{L^\infty},\ \|\MB^\star 1\|_{L^\infty},\ \|1/ c \|_{L^\infty}
   \end{equation*}
   or, if $u$ is a stationary solution, on:
   \begin{equation*}
       p,\ \mu(\O),\ \|r\|_{L^\infty},\ \|1/b\|_{L^\infty},\ \|\MB^\star 1\|_{L^\infty},\ \|1/ c \|_{L^\infty}.
   \end{equation*}

   Furthermore, if $\sup_{v\in H}\left( v|\M v \right)\leq 0$ and $p\geq 2$, then any solution $u$ is bounded in $L^\infty((0,+\infty),L^2(\O,\R))$
   by a constant that only depends on 
   \begin{equation*}
       \|u_0\|_{L^2},\ p,\ \mu(\O),\ \|r\|_{L^\infty},\ \|1/b\|_{L^\infty},\ \|1/ c \|_{L^\infty}
   \end{equation*}
   or, if $u$ is a stationary solution, on
   \begin{equation*}
       p,\ \mu(\O),\ \|r\|_{L^\infty},\ \|1/b\|_{L^\infty},\ \|1/ c \|_{L^\infty}.
   \end{equation*}
\end{theorem}

Of course, the assumption $\sup_{v\in H}\left( v|\M v \right)\leq 0$ is satisfied in many applications, for instance if
$-\MB$ has a normal compact resolvent.

The next theorem is concerned with the instability of $0$, or, in other words, with the asymptotic positivity of the solution. In general,
it only holds in a weak sense.

\begin{theorem}[Instability of $0$]\label{thm:instability_zero}
    Assume that $\MB$ admits a principal adjoint eigenfunction, namely a function $\psi\in\operatorname{int}(E^+)$ such that
    $\MB^\star\psi=0$.
    
    Then any solution $u$ of \eqref{general_system} with $u_0\neq 0$ satisfies
    \begin{equation}\label{eq:pointwise_positivity_of_the_L1_norm}
        \int_{\O}u(t,x)\upd\mu(x)>0\quad\text{for all }t\geq 0
    \end{equation}
    \revision{and}
    \begin{equation}\label{eq:asymptotic_positivity_of_the_Lp_norm}
        \limsup_{t\to+\infty}\int_{\O}u(t,x)^p\upd\mu(x) >0.
    \end{equation}
\end{theorem}

The existence of a principal adjoint eigenfunction should not be understood as a restrictive assumption. In all applications
we have in mind it is satisfied. Moreover, it could be relaxed as the existence of a principal adjoint sub-eigenfunction, namely a function
$\psi\in\operatorname{int}(E^+)$ such that $-\MB^\star\psi\leq 0$.

Knowing that the solution is asymptotically positive in some sense leads naturally to the search for positive equilibria. 
This is the point of the next theorem.

\begin{theorem}[Existence and uniqueness of the positive equilibrium]\label{thm:existence_uniqueness_stationary_state}
    Assume that the linear operator $\left(-\frac{1}{b}\M -\frac{r}{b},\B\right)$ admits an eigenfunction in $E^+\setminus\{0\}$.

   Then the problem \eqref{general_system} admits a unique stationary state $u^\star\in E^+\setminus\{0\}$, 
   and in fact $u^\star\in \operatorname{int}(E^+)$. 
\end{theorem}

Of course the assumption of Theorem \ref{thm:existence_uniqueness_stationary_state} is always satisfied if Assumption \ref{ass:reduction_to_ODEs} 
holds true or if, for instance, $\M$ is the Neumann Laplacian.
However, we will illustrate the sharpness of the assumption by providing \revision{in Section \ref{sec:counterexample_nonexistence}} 
a counterexample where the following three properties hold concurrently:
\begin{enumerate}
    \item $\lambda_1\left(-\frac{1}{b}\M -\frac{r}{b},\B\right)$ does not admit eigenfunctions in \revision{$E^+$},
    \item \eqref{general_system} does not admit a stationary state in $E^+\setminus\{0\}$,
    \item \revision{yet there exist infinitely many positive singular Radon measures that solve, in distributional sense,
    the stationary equation}.
\end{enumerate}

When the stationary state \revision{in $E^+\setminus\{0\}$} exists and is unique, it is natural to investigate its stability. 
As will be illustrated \revision{in Section \ref{sec:instability}} by means of more counterexamples, in general even the local asymptotic stability is false. 
Hopf bifurcations, in particular, need to be ruled out by some additional structural condition. 
The next theorem will bring forth sufficient conditions for the local asymptotic stability.

We recall that $\MB$ has a special form \eqref{definition_M} that make it possible to define its resolvent as a linear operator from $H$ into $H$.
We recall that the operator $\MB$ is referred to as \textit{normal} if $\MB^\star \MB =\MB \MB^\star$. A linear operator is normal if and only
if its resolvent is normal. We also recall that a normal compact operator satisfies a standard spectral decomposition theorem. 
Moreover, if $\MB$ is normal, then $1_E$ is a principal adjoint eigenfunction: the chain of equalities
$0 = \MB^\star 0 = \MB^\star \MB 1_E = \MB\MB^\star 1_E$ implies that $\MB^\star 1_E\in\ker(\MB)$, and consequently
$\MB^\star 1_E\in\R 1_E$.

The next theorems will use the following notations, when the unique positive stationary state $u^\star$ is well-defined:
\begin{equation}
    \widetilde{\M} =\frac{1}{u^\star}\M (u^\star\operatorname{id}_E)-\left(\frac{1}{u^\star}\M  u^\star\right)\operatorname{id}_E,
\end{equation}
\begin{equation}
\widetilde{\B}=\frac{1}{u^\star}\B(u^\star\operatorname{id}),
\end{equation}
\begin{equation}
    \revision{\widetilde{b}=\left(\int_{\O}(u^\star)^p c \upd\mu\right) b}.
\end{equation}
\begin{equation}
    \widetilde{c}=\frac{(u^\star)^p c }{\int_{\O}(u^\star)^p c \upd\mu },
\end{equation}
\revision{These notations will be used in Section \ref{sec:reduction_to_1} to reduce the system \eqref{general_system} in such a
way that $u^\star$ becomes $1_E$ and $r$ and $b$ become proportional. This reduction simplifies significantly the proofs on a 
technical level.}

\revision{Certain results will require the following property of \textit{exponential asymptotic stability} of the strongly continuous semigroup $(\widetilde{\T}_t )_{t\geq 0}$ generated by $\widetilde{\M}_{\widetilde{\B}}=(\widetilde{\M},\widetilde{\B})$:
\begin{multline}\label{def:exponential_asymptotic_stability_semigroup}\tag{$\exp$}
    \exists\eta>0\quad\forall (t,u_0)\in[0,+\infty)\times\dom\left(\widetilde{\M}_{\widetilde{\B}}\right) \\
    \left\|\widetilde{\T}_t u_0-\frac{\left(\widetilde{c}\middle|u_0\right)}{\left(\widetilde{c}\middle|1_E\right)}1_E\right\|_E\leq \upe^{-\eta t}\left\|u_0-\frac{\left(\widetilde{c}\middle|u_0\right)}{\left(\widetilde{c}\middle|1_E\right)}1_E\right\|_E.
\end{multline}
When $\widetilde{c}$ is an adjoint principal eigenfunction of $\widetilde{\M}_{\widetilde{\B}}$, this is the standard
exponential stability for irreducible positive strongly continuous semigroups, \textit{e.g.} 
\cite[Property $(E3_1)$]{Fonte_Sanchez_Gabriel_Mischler_2023}, up to a normalization in $H$ of $1_E$ by rewriting the quantity
$\frac{\left(\widetilde{c}\middle|u_0\right)}{\left(\widetilde{c}\middle|1_E\right)}1_E$ as
\begin{equation*}
    \frac{\left(\widetilde{c}\middle|u_0\right)}{\left(\widetilde{c}\middle|1_E\right)}1_E=\frac{\left(\widetilde{c}\middle|u_0\right)}{\left(\widetilde{c}\middle|\frac{1_E}{\sqrt{\mu(\O)}}\right)}\frac{1_E}{\sqrt{\mu(\O)}}.
\end{equation*}
Although this property is not always true, it is true under a variety of assumptions, for instance if Assumption 
\ref{ass:reduction_to_ODEs} holds true, if $\widetilde{\M}_{\widetilde{\B}}$ is a self-adjoint operator with compact resolvent,
or under Doeblin--Harris-type conditions \cite{Fonte_Sanchez_Gabriel_Mischler_2023}.}

\revision{Other results will mobilize the following quantity:}
\begin{equation}\label{def:spectral_gap}
    \sigma_2 = \inf\left\{ \left(v\middle|-\frac{\widetilde{c}}{b}\widetilde{\M} v\right)\ \middle|\ v\in \dom_E(\MB),\ v\in (1_E)^\perp,\ \|v\|_{L^2}=1\right\}.
\end{equation}
In general $\sigma_2\in\overline{\R}$ (with the convention $\sigma_2=+\infty$ if $\dom_E(\MB)=\R 1_E$). 
\revision{When $1_E$ is an adjoint principal eigenfunction of $-\frac{\widetilde{c}}{b}\widetilde{\M}_{\widetilde{\B}}=\left(-\frac{\widetilde{c}}{b}\widetilde{\M},\widetilde{\B}\right)$, 
then it is a principal eigenfunction and an adjoint principal eigenfunction of its self-adjoint part,
namely $\frac12\left(-\frac{\widetilde{c}}{b}\widetilde{\M}_{\widetilde{\B}}+\left(-\frac{\widetilde{c}}{b}\widetilde{\M}_{\widetilde{\B}}\right)^\star\right)$. 
In such a case, $\sigma_2$ is the \textit{spectral gap} between the principal eigenvalue $0$ and the rest of the spectrum of
this self-adjoint operator. This spectral gap is positive and finite under a variety of assumptions, for instance
if $-\frac{\widetilde{c}}{b}\widetilde{\M}_{\widetilde{\B}}$ is normal and has compact resolvent, if it is a
symmetric nonlocal diffusion operators or if it is a doubly stochastic matrix (that does not need to be normal).}

\revision{Before stating the theorem, we briefly recall that the local asymptotic stability property states that
there exists an open neighborhood of the equilibrium point (here, $u^\star$) such that all trajectories that enter
this neighborhood converge to the equilibrium. By virtue of standard results on semigroup theory, \textit{e.g.}
\cite[Theorem 2.1.]{Desch_Schappacher_1986}, the local asymptotic stability for \eqref{general_system} is implied by 
the exponential asymptotic stability of the equilibrium for the linearized system:
\begin{equation*}\label{general_linearized_system}
    \begin{dcases}
        \partial_t h(t,x) = \L h(t,x) & \text{for all }(t,x)\in(0,+\infty)\times\O, \\
        \L(x) = \M(x)+\mathcal{R}(u^\star(x),x)+u^\star(x) (\partial_{u}\mathcal{R})(u^\star(x),x) & \text{for all }x\in\O, \\
        \B h(t,x) = 0 & \text{for all }(t,x)\in(0,+\infty)\times \Gamma, \\
        h(0,x) = h_0(x) & \text{for all }x\in\overline{\O}.
    \end{dcases}
\end{equation*}
This is sometimes referred to as the \textit{principle of linearized stability}.
}
\begin{theorem}[Local stability of the equilibrium]\label{thm:LAS}
    Assume the operator $(-\frac{1}{b}\M -\frac{r}{b},\B)$ admits an eigenfunction in $E^+\setminus\{0\}$.
    Let $u^\star\in\operatorname{int}(E^+)$ be the unique positive stationary state of \eqref{general_system} given by 
    Theorem \ref{thm:existence_uniqueness_stationary_state}. 

    Then $u^\star$ is locally asymptotically stable if at least one of the following conditions hold true:
    \begin{enumerate}[label=\textup{(LAS.\arabic*)}]
        \item \label{cond:LAS_projection} $\widetilde{c}\in\ker\left(\left(\widetilde{\M}_{\widetilde{\B}}\right)^\star \right)$,
        \revision{\eqref{def:exponential_asymptotic_stability_semigroup} holds true}, and then the convergence to $u^\star$ occurs in the topology of $E$;
        \item \label{cond:LAS_entropy} $1_E\in\ker\left(\left(\frac{\widetilde{c}}{\revision{\widetilde{b}}}\widetilde{\M}_{\widetilde{\B}}\right)^\star\right)$,
            $\sigma_2\in(0,+\infty)$, and then the convergence to $u^\star$ occurs in the topology of $H$;
        \item \label{cond:LAS_diagonalization} $-\widetilde{\M}_{\widetilde{\B}}$ is normal and has compact resolvent on $H$,             
            and at least one of the following three supplementary conditions holds true: 
            \begin{enumerate}
                \item $\revision{\widetilde{b}}\in\R 1_E$ \revision{and 
                \begin{equation}\label{spectral_gap_condition}\tag{$\sigma_{\textup{gap}}$}
                    \inf\left\{\operatorname{Re}(\lambda)\ \middle|\ \lambda\in\operatorname{sp}\left(-\MB\right)\setminus\{0\}\right\}>0;
                \end{equation}}
                \item $\widetilde{c}\in\R 1_E$ \revision{and \eqref{spectral_gap_condition} holds true};
                \item the following two conditions are true:
                    \begin{equation}\label{spectrum_cone_condition}\tag{$\sigma_{\textup{cone}}$}
                        \operatorname{sp}(-\widetilde{\M}_{\widetilde{\B}})\subset \left\{ \lambda\in\C\ \middle|\ \operatorname{Re}(\lambda)\geq |\operatorname{Im}(\lambda)|\right\},
                    \end{equation}
                    \begin{equation}\label{angle_condition}\tag{$\angle$}
                        \left(\revision{\widetilde{b}}|\widetilde{c}\right)+\frac{\|\revision{\widetilde{b}}\|_{L^1}\|\widetilde{c}\|_{L^1}}{\mu(\O)}> \sqrt{\left(\|\revision{\widetilde{b}}\|_{L^2}^2-\frac{\|\revision{\widetilde{b}}\|_{L^1}^2}{\mu(\O)}\right)\left(\|\widetilde{c}\|_{L^2}^2-\frac{\|\widetilde{c}\|_{L^1}^2}{\mu(\O)}\right)},
                    \end{equation}
                    and then the convergence to $u^\star$ occurs in the topology of $H$.
       \end{enumerate}
    \end{enumerate}
\end{theorem}

\revision{Remark that all conditions above are scale-invariant regarding $\widetilde{b}$ and $\widetilde{c}$, and in particular
$\widetilde{b}$ can be replaced everywhere by $b$ at no cost.}

\revision{Remark also that \eqref{spectral_gap_condition} is implied by the stronger
\eqref{spectrum_cone_condition}. If $-\widetilde{\M}_{\widetilde{\B}}$ is self-adjoint with
compact resolvent, \eqref{spectral_gap_condition} is automatic.}

The first two properties follow from relatively standard methods in stability analysis, that will be recalled below. 
The last one, on the contrary, was inspired by an original finite-dimensional lemma on symmetric matrices 
suggested to the first author by P. Autissier \cite{Autissier}. 
Below we will prove an infinite-dimensional generalization of it for normal operators. 
We will also show that it extends, in a quantified way, another more standard diagonalization method.
Furthermore, we will show with examples that each one of the three conditions in Theorem \ref{thm:LAS} contains 
possibilities not covered by the other two.

Through numerical tests with partially random coefficients, we investigated the sharpness of the geometrical condition 
\eqref{angle_condition} when $-\widetilde{\M}_{\widetilde{\B}}$ is not only normal with a spectrum controlled by \eqref{spectrum_cone_condition}, but 
actually self-adjoint. 
We never caught a case where $u^\star$ is unstable. Hence we suggest the following conjecture.

\begin{conj}
    Assume the operator $\left(-\frac{1}{b}\M -\frac{r}{b},\B\right)$ admits an eigenfunction in $E^+\setminus\{0\}$.
    Let $u^\star$ be the unique positive stationary state of \eqref{general_system} given by Theorem \ref{thm:existence_uniqueness_stationary_state}.

    Assume that $-\widetilde{\M}_{\widetilde{\B}}$ is self-adjoint and has compact resolvent on $H$.

    Then $u^\star$ is locally asymptotically stable. 
    The convergence to $u^\star$ occurs in the topology of $H$.
\end{conj}

When Assumption \ref{ass:reduction_to_ODEs} holds true, this conjecture reduces to an earlier one \cite[Conjecture 2.16]{Bierkens_Ran}. 

If $u^\star$ is constant, then $\M$ and $\widetilde{\M}$ coincide. 
But in general, the two operators differ, and the fact that one of them is normal does not imply that the other is also normal. 
Similarly, the fact that one of them is self-adjoint does not imply that the other is self-adjoint. 

It might come as a surprise that the conjecture is about the self-adjointness of $-\widetilde{\M}_{\widetilde{\B}}$ and not about that of $-\MB$. 
In fact, \revision{in Section \ref{sec:instability},} we will give an explicit counterexample under Assumption \ref{ass:reduction_to_ODEs} where $\M$ is a symmetric matrix but $u^\star$ is unstable. 
This counterexample settles negatively a (somewhat informal) long-standing conjecture on stability in mutation--selection models with symmetric mutations and rank-one competition. \revision{It also shows that the interest of the reduction of Section
\ref{sec:reduction_to_1} might go beyond mere technical simplifications, by making it easier to identify stability regimes.}

The possibility of Hopf bifurcations is reminiscent of results previously established by the second author for a system similar to
\eqref{discrete_ode_system} but with a circulant competition matrix \cite{Girardin_2018}. 

Our final result is the following theorem about global stability.

\begin{theorem}[Global asymptotic stability]\label{thm:GAS}
    Assume the operator $\left(-\frac{1}{b}\M -\frac{r}{b},\B\right)$ admits an eigenfunction in $E^+\setminus\{0\}$.
    Let $u^\star$ be the unique positive stationary state of \eqref{general_system} given by Theorem \ref{thm:existence_uniqueness_stationary_state}. 

    Then $u^\star$ is globally asymptotically stable for initial conditions $u_0\in \dom_{E^+}(\MB)\setminus\{0\}$
    if at least one of the following conditions hold true:
    \begin{enumerate}[label=\textup{(GAS.\arabic*)}]
        \item \label{cond:GAS_projection} $p=1$, $\revision{\widetilde{b}}\in \R 1_E$, 
        $\widetilde{c}\in\ker\left(\left(\widetilde{\M}_{\widetilde{\B}}\right)^\star\right)$,
        \revision{\eqref{def:exponential_asymptotic_stability_semigroup} holds true} and then the
            convergence to $u^\star$ occurs in the topology of $E$;
        \item \label{cond:GAS_entropy} $p=1$, $1_E\in\ker\left(\left(\frac{\widetilde{c}}{\revision{\widetilde{b}}}\widetilde{\M}_{\widetilde{\B}}\right)^\star\right)$,
            $\sigma_2\in(0,+\infty)$ and
            \begin{equation}\label{entropy_method_condition}\tag{$\mathcal{CS}$}
                \sup_{\substack{h\in E^+\setminus \R 1_E \\ \revision{\frac{\sigma_2}{\mu(\O)}}\left(\mu(\O)\|h\|_{L^2}^2-\|h\|_{L^1}^2\right)<(\widetilde{c}|h^2)+(\widetilde{c}|h)^2}} \frac{\sqrt{(\widetilde{c}|h^2)}-(\widetilde{c}|h)}{\sqrt{\mu(\O)\|h\|_{L^2}^2-\|h\|_{L^1}^2}} < \sqrt{\revision{\frac{\sigma_2}{\mu(\O)}}},
            \end{equation}
            and then the convergence to $u^\star$ occurs in the topology of $H$;
        \item \label{cond:GAS_diagonalization} $\revision{\widetilde{b}}\in\R 1_E$, $-\widetilde{\M}_{\widetilde{\B}}$ is normal and has compact resolvent on $H$, 
        \revision{\eqref{spectral_gap_condition} holds true} 
        and either Assumption \ref{ass:reduction_to_ODEs} holds true or $p\leq 2$, and then the convergence to $u^\star$ occurs in the topology of $E$;
        \item \label{cond:GAS_gradient_flow} $p=2$, $-\frac{\widetilde{c}}{\revision{\widetilde{b}}}\widetilde{\M}_{\widetilde{\B}}$ is self-adjoint and has compact resolvent on $H$,
            the parameters $u_0$, $r$, $b$, $c$, and those in $\MB$ have sufficient regularity to ensure that the solutions of \eqref{general_system} 
            are actually in $\c^1((0,+\infty),\dom_E(\MB))$, and then the convergence to $u^\star$ occurs in the topology of $H$.
    \end{enumerate}
\end{theorem}

\revision{All conditions above are scale-invariant regarding $\widetilde{b}$, and again $\widetilde{b}$ can be replaced 
everywhere by $b$ at no cost. However \eqref{entropy_method_condition} is not scale-invariant with respect to $\widetilde{c}$.}

Each one of the four conditions in Theorem \ref{thm:GAS} contains possibilities not covered by the other three.

We recall from the state of the art above that $u^\star$ is globally asymptotically stable if Assumption \ref{ass:reduction_to_ODEs} 
holds true and either $N'=1$ or $(N',p)=(2,1)$.

We point out that, on one hand, $p$ plays no role as far as local asymptotic stability is concerned (Theorem 
\ref{thm:LAS}), but on the other hand, it plays a quite significant role for global asymptotic stability, especially in 
infinite dimensions (Theorem \ref{thm:GAS}).

One of the most interesting subsisting problems is: does there exist a case where $u^\star$ exists, is unique, is locally asymptotically stable, 
but is not globally asymptotically stable? 
Numerical simulations, not shown here, support the conjecture that a succession of a supercritical Hopf bifurcation and a subcritical Hopf 
bifurcation can occur. This results in a bistable regime, where both the unique equilibrium and the limit cycle produced by the supercritical 
Hopf bifurcation are locally asymptotically stable, with limited basins of attraction. 
The rigorous proof of this bistability regime, stated in the following conjecture, will be the object of a future work.

\begin{conj}
    Assume that $N'=3$ and Assumption \ref{ass:reduction_to_ODEs} holds true. 

    Then there exists a choice of parameters such that $u^\star$ is locally asymptotically stable but not globally asymptotically stable.
    The $\omega$-limit set also contains a locally asymptotically stable limit cycle. 

    In addition, $\MB$ can be chosen self-adjoint.
\end{conj}

\subsection{Organization of the paper}

The rest of this paper is organized to prove the above theorems. \revision{In Section 2, we prove Theorem 
\ref{thm:uniform_bound_dynamical_system}. In Section 3, we prove Theorem \ref{thm:instability_zero}. In Section 4,
we prove Theorem \ref{thm:existence_uniqueness_stationary_state} on one hand and we construct a counterexample to existence
on the other hand. In Section 5, we reduce the stability problem to an easier one. In Section 6, we exhibit counterexamples 
to the local asymptotic stability on one hand and we prove \ref{thm:LAS} on the other hand. In Section 7, we prove \ref{thm:GAS}.}

\section{Global boundedness of the total population}

In this section, we derive \textit{a priori} estimates on the solutions of \eqref{general_system}.

\begin{proof}[Proof of Theorem \ref{thm:uniform_bound_dynamical_system}]
    The existence and uniqueness of the solution of the Cauchy problem \eqref{general_system}, being given an initial condition
   $u_0\in \dom_{E^+}(\MB)$, as well as an exponential bound on its growth in time, are standard, as explained in the introduction.

   First we prove the $L^\infty(L^1)$ estimate. Define $U:t\mapsto \|u(t)\|_{L^1(\O)}$ and note that 
   $U(t)=\int_{\O} u(t,y)\upd\mu(y)=\left(1|u(t,\cdot)\right)=\langle 1,u(t,\cdot)\rangle_E$ for any $t\geq 0$.
   
   By the H\"{o}lder inequality,
    \begin{equation*}
        U^{p}(t)\le \mu(\O)^{p-1} \left( 1| u^p(t)\right),
    \end{equation*} 
    and by uniform positivity of $c$,
    \begin{equation*}
        U^{p}(t)\leq \|1/ c \|_{L^\infty}\mu(\O)^{p-1} \left( c \middle| u^p(t)\right).
    \end{equation*} 
    Moreover, since $u(t,x)\in \c^1((0,+\infty),E)$, 
   \begin{equation*}
        \left( 1|\partial_t u(t,\cdot)\right)=\frac{\upd}{\upd t}\left( 1|u(t,\cdot)\right)=U'(t).
    \end{equation*} 
    
   Therefore, using $r,b,c\in E\subset\BUC(\overline{\O})$,
   \begin{equation*}
   \begin{split}
       U'(t) & = \left( 1| \M  u(t)\right) + \left(r|u(t) \right) - \left( b|u(t)\right) \left( c|u^p(t) \right), \\
        & \leq \left( \MB^\star 1|u(t)\right)+\|r\|_{L^\infty}U(t)-\inf_{x\in\O}b(x)U(t)\left(  c \middle| u^p(t,\cdot) \right), \\
        & \leq |\left( \MB^\star 1|u(t)\right)|+\|r\|_{L^\infty}U(t)-\frac{1}{\mu(\O)^{p-1}\|1/ c \|_{L^\infty}\|1/b\|_{L^\infty}}U^{p+1}(t), \\
        & \leq \left(|\MB^\star 1|||u(t)|\right)+\|r\|_{L^\infty}U(t)-\frac{1}{\mu(\O)^{p-1}\|1/ c \|_{L^\infty}\|1/b\|_{L^\infty}}U^{p+1}(t), \\
        & \leq \|\MB^\star 1\|_{L^\infty}\left( 1||u(t)|\right)+\|r\|_{L^\infty}U(t)-\frac{1}{\mu(\O)^{p-1}\|1/ c \|_{L^\infty}\|1/b\|_{L^\infty}}U^{p+1}(t), \\
        & \leq \left(\|\MB^\star 1\|_{L^\infty}+\|r\|_{L^\infty}\right)U(t)-\frac{1}{\mu(\O)^{p-1}\|1/ c \|_{L^\infty}\|1/b\|_{L^\infty}}U^{p+1}(t).
   \end{split}
   \end{equation*}

    Since $U\geq 0$, the boundedness of $U$ follows straightforwardly from the logistic-type ordinary differential equation. More precisely, for all $t\geq 0$, 
    \begin{equation*}
        \|u(t)\|_{L^1(\O)}\leq \max\left\{ \|u_0\|_{L^1}, \left(\mu(\O)^{p-1}\|1/ c \|_{L^\infty}\|1/b\|_{L^\infty}\left(\|\MB^\star 1\|_{L^\infty}+\|r\|_{L^\infty}\right)\right)^{\frac{1}{p}} \right\},
    \end{equation*} 
    or, if $u\in E^+$ is a stationary solution,
    \begin{equation*}
        \|u\|_{L^1} \leq \left(\mu(\O)^{p-1}\|1/ c \|_{L^\infty}\|1/b\|_{L^\infty}\left(\|\MB^\star 1\|_{L^\infty}+\|r\|_{L^\infty}\right)\right)^{\frac{1}{p}}.
    \end{equation*}
    This ends the proof of the $L^\infty(L^1)$ estimate.
    
    Next, if $\sup_{v\in H}\left(v|\M v\right)\leq 0$ and $p\geq 2$, the $L^\infty(L^2)$ estimate follows from the same arguments, simply replacing
    the test function $1_E$ by $2 u(t)$. 
\end{proof}

\begin{remark}
    If Assumption \ref{ass:reduction_to_ODEs} holds true, then it is possible to use the equivalence of norms in $E=\R^{N'}$
    (more precisely, $\|\cdot\|_\infty\leq \|\cdot\|_{1}$) to deduce a global pointwise bound. Otherwise, the norms
    are not equivalent, and as a matter of fact, we will prove that some stationary solutions are bounded only in the sense of measures.
    Still, in applications where $\M$ is local and uniformly elliptic in the continuous variable, for instance
    \begin{equation*}
        \M=
        \begin{cases}
            \Delta+\mathbf{M} & \text{in }\O, \\
            \partial_{\nu} & \text{on }\Gamma,
        \end{cases}
    \end{equation*}
    Schauder estimates and/or the Harnack inequality make it possible to deduce a bound in $\BUC([0,\infty)\times\overline{\O})$ from the $L^\infty(L^1)$ bound.
\end{remark}

\section{Instability of $0$}

In this section, we prove that, at least in a weak sense, the solution cannot stay close to $0$.

\begin{proof}[Proof of Theorem \ref{thm:instability_zero}]
    By comparisons similar to that of the proof of Theorem \ref{thm:uniform_bound_dynamical_system}, we obtain:
    \begin{equation*}
        \partial_t u \geq \M u + u\left(\inf_{\O}(r) -\sup_{\O}(b)\sup_{\O}(c)\int_{\O} u^p\upd\mu\right).
    \end{equation*}

    We begin with the proof of \eqref{eq:pointwise_positivity_of_the_L1_norm}.
    Let
    \begin{equation*}
        f:t\mapsto-\int_0^t\left(\inf_{\O}(r) -\sup_{\O}(b)\sup_{\O}(c)\int_{\O} u(t')^p\upd\mu\right)\upd t'.
    \end{equation*}
    This function is in general not bounded uniformly in time, but it is well defined pointwise, by virtue of semigroup arguments detailed in 
    the introduction. The function $v:(t,x)\mapsto \upe^{f(t)}u(t,x)$ satisfies:
    \begin{equation*}
        \begin{dcases}
            \partial_t v \geq \M v & \text{in }(0,+\infty)\times\O, \\
            \B v  = 0 & \text{on }(0,+\infty)\times \Gamma, \\
            v(0) = u_0 & \text{in }\overline{\O}.
        \end{dcases}
    \end{equation*}
    Thus, by virtue of the comparison principle, $u(t,x)\geq \upe^{-f(t)}\T_t u_0(x)$. It remains to prove the positivity of $w(t)=\T_t u_0$ for all $t\geq 0$.
    Taking the scalar product between the equation satisfied by $w$ and the principal adjoint eigenfunction $\psi$ of $\M$, we deduce that 
    $\left(\psi|\partial_t w\right)=\frac{\upd}{\upd t}\left(\psi|w\right)=0$. Consequently, the function $t\mapsto \left(\psi|w\right)$
    is constant, and identically equal to $\left(\psi|u_0\right)$. This constant is positive due to $\psi\in\operatorname{int}(E^+)$ and 
    $u_0\in E^+\setminus\{0\}$. This ends the proof of \eqref{eq:pointwise_positivity_of_the_L1_norm}.

    Next, to prove \eqref{eq:asymptotic_positivity_of_the_Lp_norm}, we argue by contradiction. Assume 
    $\int_{\O}u(t)^p\upd\mu\to 0$ as $t\to+\infty$. Then there exists $t_0\geq 0$ such that 
    \begin{equation*}
        \inf_{t\geq t_0}\left(\inf_{\O}(r) -\sup_{\O}(b)\sup_{\O}(c)\int_{\O} u^p\upd\mu\right)\geq\frac{\inf_{\O} r}{2}>0.
    \end{equation*}
    Then, by arguments similar to those used earlier in the proof, 
    \begin{equation*}
        \left(\psi|u(t)\right)\geq \upe^{\frac{\inf_{\O}r}{2}(t-t_0)}\left(\psi|u(t_0)\right)\quad\text{for all }t\geq t_0.
    \end{equation*}
    But now by virtue of \eqref{eq:pointwise_positivity_of_the_L1_norm} the right-hand side above goes to $+\infty$ as $t\to+\infty$. 
    Thus so does $\left(\psi|u(t)\right)$, and then so does $\int_{\O}u(t)^p\upd\mu$ by virtue of the H\"{o}lder inequality.
    This is a contradiction. This ends the proof of \eqref{eq:asymptotic_positivity_of_the_Lp_norm} and of Theorem \ref{thm:instability_zero}.
\end{proof}

\section{Existence and uniqueness of the stationary solution}

A stationary solution of \eqref{general_system} is a positive solution $u\in E^+\setminus\{0\}$ to the following system:
\begin{equation}\label{eq:sta}
    \begin{dcases}
        -\M  u(x)= u(x)\left(r(x)-b(x)\int_{\O}u(y)^p  c (y)\upd\mu(y)\right) & \text{for any }x\in\O, \\
        \B u(x) = 0 & \text{for any }x\in \Gamma.
    \end{dcases}
\end{equation}

\subsection{Existence and uniqueness}

\begin{proof}[Proof of Theorem \ref{thm:existence_uniqueness_stationary_state}]
Since $b\in\operatorname{int}(E^+)$ and $E$ is a Banach algebra, $1/b\in \operatorname{int}(E^+)$, and then we can divide \eqref{eq:sta} by $b$ and obtain the following equivalent system:
\begin{equation}\label{eq:sta_equivalent}
    \begin{dcases}
        -\frac{1}{b(x)}\M  u(x)= u(x)\left(\frac{r(x)}{b(x)}-\int_{\O}u(y)^p  c (y)\upd\mu(y)\right) & \text{for any }x\in\O, \\
        \B u(x) = 0 & \text{for any }x\in \Gamma.
    \end{dcases}
\end{equation}

It appears clearly that $u\in E^+\setminus\{0\}$ is a solution of this system if and only if 
it is an eigenfunction of $\left(-\frac{1}{b}\M -\frac{r}{b},\B\right):\dom_E(\MB)\to E$ for the eigenvalue $\lambda=-\int_{\O}u^p c$. By assumption of the
theorem, such an eigenfunction exists indeed.

Let $v\in E^+\setminus\{0\}$ be an eigenfuction of $\left(-\frac{1}{b}\M -\frac{r}{b},\B\right)$ associated to an eigenvalue 
$\lambda\in\R$. It follows that $rv+\lambda bv\in E$ is in the domain of $(-\M +C,\B)^{-1}$, 
where $C>0$ is given by Property \ref{property:strongly_positive_resolvent}, with
\begin{equation*}
    v=(-\M +C,\B)^{-1}(rv+Cv+\lambda bv) \quad\text{in }\O.
\end{equation*}
Up to increasing $C$, by strong positivity of the resolvent $(-\M +C,\B)^{-1}$, $v\in\operatorname{int}(E^+)$. 
Hence any eigenfunction in $E^+\setminus\{0\}$ is actually in $\operatorname{int}(E^+)$.

Let $v^1$, $v^2$ be two eigenfunctions of $\left(-\frac{1}{b}\M -\frac{r}{b},\B\right)$ in $\operatorname{int}(E^+)$, and let $\lambda^1$, $\lambda^2$
be the associated eigenvalues. Up to changing notations, assume $\lambda^1\leq \lambda^2$. Since $v^2\in\operatorname{int}(E^+)$, there
exists $\kappa>0$ such that $v^2-\kappa v^1\in\operatorname{int}(E^+)$. In particular $\frac{1}{\kappa}v^2\geq v^1\geq 0$. 
From $v^1\neq 0$, we deduce that the supremum $\kappa^\star$ of such $\kappa>0$ is necessarily finite.
By construction, $v^2-\kappa^\star v^1\in\partial E^+$. Moreover, $\B (v^2-\kappa^\star v^1)=0$ on $ \Gamma$ and, in $\O$,
\begin{equation*}
    \left(-\frac{1}{b}\M -\frac{r}{b}\right)(v^2-\kappa^\star v^1)=\lambda^2 v^2-\kappa^\star \lambda^1 v^1\geq\lambda^1 v^2-\kappa^\star \lambda^1 v^1=\lambda^1(v^2-\kappa^\star v^1).
\end{equation*}
Repeating the same strong positivity argument as before, we deduce that either $v^2-\kappa^\star v^1=0$ or $v^2-\kappa^\star v^1\in\operatorname{int}(E^+)$. Since the latter is impossible, we deduce $v^2=\kappa^\star v^1$. It follows subsequently from the calculation
that $\lambda^2=\lambda^1$. 

In other words, we have recovered a standard Krein--Rutman property: the eigenvalue associated with nonnegative
eigenfunctions is unique, simple and actually associated with a positive eigenfunction.

Now we verify that the principal eigenvalue $\lambda=\lambda_1\left(-\frac{1}{b}\M-\frac{r}{b},\B\right)$ is negative.
By contradiction, assume $\lambda\geq 0$. Let $v$ be an associated principal eigenfunction.
Using $-\M(C1_E-v) = rv+\lambda bv\geq 0$, we deduce by another comparison argument that $v$ is proportional to $1_E$.
But then by definition of $v$ it satisfies $rv+\lambda bv=0$, which contradicts $\inf_{\O}(r+\lambda b)>0$.
Thus $\lambda<0$.

Let $v\in\operatorname{int}(E^+)$ be the unique principal eigenfunction satisfying the normalization $\|v\|_E=1$. 
Any other principal eigenfunction has the form $Av$ with $A>0$. Since the equation $\lambda=-\int_{\O}(Av)^p c\upd\mu $
admits exactly one solution, which is of course
\begin{equation*}
    A=\left( \frac{-\lambda}{\int_{\O}v^p c\upd\mu } \right)^{1/p},
\end{equation*}
the uniqueness of the solution of \eqref{eq:sta_equivalent} follows.
\end{proof}

From the above proof it appears clearly that, conversely, the existence of a stationary solution
in $E^+\setminus\{0\}$ implies the existence of a principal eigenfunction of 
$-\frac{1}{b}\M -\frac{r}{b}$ in $E^+\setminus\{0\}$.

\subsection{An example of nonexistence}\label{sec:counterexample_nonexistence}
When no such eigenfunction exists, the situation is more elusive. Let us present a counterexample
where \revision{no solution in $E^+\setminus\{0\}$ exists. Actually, in this counterexample, infinitely many positive stationary 
solutions exist, but only in a weak sense: all these solutions are singular Radon measures.}

\begin{proposition}
    Let $\O=\O_1\cup\O_2\cup\O_3\subset\R^4$, with
\begin{equation*}
    \O_1 = B(0,1)\cap\left\{ x_1<0 \right\},
\end{equation*}
\begin{equation*}
    \O_2 = [0,1]\times B_{\R^3}(0,1),
\end{equation*}
\begin{equation*}
    \O_3 = B((1,0,0,0),1)\cap\left\{ \revision{x_1}>1 \right\}.
\end{equation*}
Let $E=\textup{Lip}(\overline{\O},\R)$. Let $D>0$ and let 
\begin{equation*}
    r:x\in\O\mapsto
    \begin{cases}
        2-|x|^2 & \text{if }x_1<0, \\
        2-x_2^2-x_3^2-x_4^2 & \text{if }x_1\in[0,1], \\
        2-|x-(1,0,0,0)|^2 & \text{if }x_1>1.
    \end{cases}
\end{equation*}

Then there exists infinitely many solutions to 
\begin{equation}\label{eq:counter_example_continuum_measures_weak_formulation}
    \forall\varphi\in\c_b(\O)\quad D\langle u,1\rangle_E\langle 1,\varphi\rangle_E-D|\O|\langle u,\varphi\rangle_E+\langle u,r\varphi\rangle_E-\langle u,1\rangle_E\langle u,\varphi\rangle_E = 0
\end{equation}
in the set of positive Radon measures on $\O$, but none of them are in $E^+\setminus\{0\}$.
\end{proposition}

\begin{remark}
The equation \eqref{eq:counter_example_continuum_measures_weak_formulation} is 
nothing more than a dual formulation of the following stationary nonlocal diffusion equation:
\begin{equation}\label{eq:counter_example_continuum_measures}
    0 = D\int_{\O}(u(y)-u(x))\upd y+u(x)\left(r(x)-\int_{\O}u(y)\upd y\right)\quad\text{for all }x\in\O,
\end{equation}
which is indeed a particular case of \eqref{general_system}.
\end{remark}

\begin{proof}
Let $u$ be a positive Radon measure on $\O$, 
let $M=\int_{\O}\upd u=\langle u,1\rangle_E$, and assume that $u$ satisfies \eqref{eq:counter_example_continuum_measures_weak_formulation}.

Since $1\leq r\leq 2$ in $\O$, we deduce by testing against $1\in\c_b(\O)$ that
\begin{equation*}
    DM|\O|-D|\O|M+M-M^2\leq 0\leq DM|\O|-D|\O|M+2M-M^2,
\end{equation*}
whence $1\leq M\leq 2$.

Let $A_D=M+D|\O|-2$ be the minimal value of $M+D|\O|-r$ on $\overline{\O}$. It satisfies $A_D<D|\O|$. 
It is attained exactly on $I=\left\{ (x_1,0,0,0)\ \middle|\ x_1\in[0,1] \right\}$.

First, assume $A_D<0$. Then by testing \eqref{eq:counter_example_continuum_measures_weak_formulation} against a 
nonnegative test function in $\c_b(\O)$ supported in, for instance, $\left\{ x\in\O\ \middle|\ \operatorname{dist}(x,I)\leq\varepsilon \right\}$ 
for a sufficiently small $\varepsilon>0$, we get a contradiction. Hence $A_D\geq 0$.

Next, assume $0<A_D<D|\O|$. Testing \eqref{eq:counter_example_continuum_measures_weak_formulation} against 
$(M+D|\O|-r)^{-1}\in\c_b(\O)$, we find $M=DM\int_{\O}\frac{1}{M+D|\O|-r}$, hence
$D\int_{\O}\frac{1}{M+D|\O|-r}=1$. After a few changes of variables, this equality reads:
\begin{equation*}
    \begin{split}
        1 & = D\int_{0}^1\frac{2\pi^2\rho^3+4\pi\rho^2}{A_D+\rho^2}\upd\rho \\
        & = 2\pi D\int_{0}^1\frac{\pi\rho(\rho^2+A_D)-\pi A_D\rho+2(\rho^2+A_D)-2 A_D}{A_D+\rho^2}\upd\rho \\
        & = 2\pi D\int_{0}^1 \left(\pi\rho-\frac{\pi A_D}{2}\frac{2\rho}{A_D+\rho^2}+2-2\sqrt{A_D}\frac{1/\sqrt{A_D}}{1+(\rho/\sqrt{A_D})^2} \right) \upd\rho \\
        & = 2\pi D\left( \frac{\pi}{2} -\frac{\pi A_D}{2}\ln(A_D+1)+\frac{\pi A_D}{2}\ln(A_D) +2- 2\sqrt{A_D}\arctan\left(\frac{1}{\sqrt{A_D}}\right)\right).
    \end{split}
\end{equation*}
However the convergence of the right-hand side to $0$ as $D\to 0$ is a contradiction.

Therefore, provided $D$ is small enough, $A_D=M+D|\O|-2=0$ exactly. This defines uniquely the total mass $M=2-D|\O|$
(which is positive indeed provided $D<\frac{2}{|\O|}$).
This also means that $x\in\O\mapsto M+D|\O|-r(x)$ is nonnegative and vanishes exactly on $I$.

Testing \eqref{eq:counter_example_continuum_measures_weak_formulation} against a test function of the form 
\begin{equation*}
    x\mapsto
    \begin{cases}
        \frac{\varphi(x)}{2-r(x)} & \text{if }x\in\omega, \\
        0 & \text{if }x\notin\omega,
    \end{cases}
\end{equation*}
with $\omega$ an arbitrary open measurable subset, $\varphi\in\c_b(\omega)$ and 
$\overline{\omega}\subset\operatorname{int}(\O\setminus I)$,
we deduce that the measure $u$ admits in $\operatorname{int}(\O\setminus I)$ the density 
\begin{equation*}
    f:x\in\O\setminus I\mapsto \frac{D(2-D|\O|)}{2-r(x)}.
\end{equation*}
This function $f$ is integrable in $\O\setminus I$ and
\begin{equation*}
    \|f\|_{L^1(\O\setminus I)} = D(2-D|\O|)\int_{0}^1\frac{2\pi^2\rho^3+4\pi\rho^2}{\rho^2}\upd\rho = D(2-D|\O|)(\pi^2+4\pi).
\end{equation*}
Therefore the measure $u$ necessarily has the form 
\begin{equation*}
     u=f\upd x+ \mu_I,
\end{equation*}
with $\mu_I$ a nonnegative measure supported in the interval $I$ such that 
\begin{equation*}
    \mu_I(\O) = M-\|f\|_{L^1(\O)} = (2-D|\O|)\left(1-D(\pi^2+4\pi)\right)
\end{equation*}
(which is positive indeed provided $D<\frac{2}{|\O|}$ and $D<\frac{1}{\pi^2+4\pi}$; since $|\O| = \frac12\pi^2+\frac{4}{3}\pi$,
this condition reduces to $D<\frac{1}{\pi^2+4\pi}$).

Conversely, \revision{assume $0<D<\frac{1}{\pi^2+4\pi}$ and pick} any measure $u$ of the above form.
\revision{Then, by the same calculations, $\langle u,1 \rangle_E=2-D|\Omega|$. 
Define the following Radon measure:
\begin{equation*}
    S=D\left(\langle u,1 \rangle_E-|\O|u\right)+ru-\langle u,1 \rangle_E u,
\end{equation*}
where the product $ru$ is indeed well-defined as a Radon measure. To prove that $u$ is a positive solution 
of \eqref{eq:counter_example_continuum_measures_weak_formulation} is to prove that $S=0$.
For any $\varphi\in\c_b(\O)$,
\begin{equation*}
    \begin{split}
        \langle S,\varphi\rangle_E & = D\langle u,1\rangle_E\langle 1,\varphi\rangle_E-D|\O|\langle u,\varphi\rangle_E+\langle u,r\varphi\rangle_E - \langle u,1 \rangle_E\langle u,\varphi\rangle_E \\
        & = D(2-D|\O|)\langle 1,\varphi\rangle_E-\langle u,(D|\O|-r+2-D|\O|)\varphi\rangle_E \\
        & = D(2-D|\O|)\langle 1,\varphi\rangle_E-\langle u,(2-r)\varphi\rangle_E \\
        & = D(2-D|\O|)\langle 1,\varphi\rangle_E-\langle f,(2-r)\varphi\rangle_E \\
        & = D(2-D|\O|)\langle 1,\varphi\rangle_E-\langle (2-r)f,\varphi\rangle_E \\
        & = D(2-D|\O|)\langle 1,\varphi\rangle_E-\langle D(2-D|\O|),\varphi\rangle_E \\
        & = 0.
    \end{split}
\end{equation*}
Thus $u$ is indeed a solution of \eqref{eq:counter_example_continuum_measures_weak_formulation}, whence it is a positive}
stationary solution, in the sense of measures, of \eqref{eq:counter_example_continuum_measures}.

\revision{Therefore, when $D$ is sufficiently small, the set of positive bounded stationary solutions is characterized
and it does not intersect $E$: 
all positive bounded stationary solutions are singular Radon measures and there exists no solution in $E^+$.}
\end{proof}

\section{\revision{A reduction of the stability problem}}\label{sec:reduction_to_1}

\revision{In this section, we prepare the proofs of Theorem \ref{thm:LAS} and Theorem \ref{thm:GAS}.
To do so, we assume the existence in $E^+$ of a principal eigenfunction of $\left(-\frac{1}{b}\M -\frac{r}{b},\B\right)$
and we denote $u^\star$ the unique positive stationary state given by Theorem \ref{thm:existence_uniqueness_stationary_state}.
Below, we show how the problem can be reduced to a particular case, which will turn out to be more amenable to the calculations
involved in the proofs of Theorem \ref{thm:LAS} and Theorem \ref{thm:GAS}.}

By construction, $u^\star\in\operatorname{int}(E^+)$, and since $E$ is a Banach algebra, $\frac{1}{u^\star}\in\operatorname{int}(E^+)$. 

Let $\widetilde{E}=\frac{1}{u^\star}E=\left\{ \frac{u}{u^\star}\ \middle|\ u\in E \right\}$. In fact $\widetilde{E}=E$. 

By definition of $u^\star$, for any $v\in E$,
\begin{equation*}
    -(\M  u^\star) v = ru^\star v-bu^\star v\int_{\O}(u^\star)^p c\upd\mu ,
\end{equation*}
whence:
\begin{equation*}
    rv = -\left(\frac{1}{u^\star}\M  u^\star\right)v+\left(b\int_{\O}(u^\star)^p c \upd\mu \right)v.
\end{equation*}

Now, we consider the evolution problem \eqref{general_system} and we change the unknown $u$ into $v=\frac{u}{u^\star}$. 

Straightforward algebra leads to 
\begin{equation*}
    v\r(u^\star v,x) = -\left(\frac{1}{u^\star}\M  u^\star\right)v+\left(b\int_{\O}(u^\star)^p c \upd\mu \right)v - bv\int_{\O}v^p(u^\star)^p c \upd\mu
\end{equation*}
and then to the following Cauchy problem for $v$:
\begin{equation*}
    \begin{dcases}
        \partial_t v = \frac{1}{u^\star}\M (u^\star v) +v\r(u^\star v,x)  & \text{in }(0,+\infty)\times\O, \\
        \frac{1}{u^\star}\B(u^\star v)=0 & \text{on }(0,+\infty)\times \Gamma, \\
        v(0)=\frac{u_0}{u^\star} & \text{in }\overline{\O}.
   \end{dcases}
\end{equation*}
Defining $v_0=\frac{u_0}{u^\star}$ and, as in the introduction,
\begin{equation*}
    \widetilde{\M }=\frac{1}{u^\star}\M (u^\star\operatorname{id}_E)-\left(\frac{1}{u^\star}\M u^\star\right)\operatorname{id}_E, 
\end{equation*}
\begin{equation*}
    \widetilde{\B}=\frac{1}{u^\star}\B(u^\star\operatorname{id}),
\end{equation*}
\begin{equation*}
    \widetilde{b}=\left(\int_{\O}(u^\star)^p c \upd\mu \right) b, 
\end{equation*}
\begin{equation*}
    \widetilde{ c }=\frac{(u^\star)^p c }{\int_{\O}(u^\star)^p c \upd\mu },
\end{equation*}
the new system reads:
\begin{equation*}
    \begin{dcases}
        \partial_t v = \widetilde{\M }v + \widetilde{b}v\left( 1 - \int_{\O}v^p \widetilde{ c } \upd\mu \right) & \text{in }(0,+\infty)\times\O, \\
        \widetilde{\B} v = 0 & \text{on }(0,+\infty\times \Gamma, \\
        v(0)=v_0 & \text{in }\overline{\O}.
   \end{dcases}
\end{equation*}

It can be verified that this reduced new problem has the same structure as \eqref{general_system}. In particular,
$(\widetilde{\M},\widetilde{\B})$ keeps the form \eqref{definition_M},
it is compatible with $E$ in the sense of Assumption \ref{ass:compatibility}, 
$\widetilde{\M}_{\widetilde{\B}}$ is an irreducible operator in the sense of Assumption \ref{ass:irreducibility},
$\widetilde{b},\widetilde{c}\in\operatorname{int}(E^+)$,
$v_0\in E^+$, $v\in\mathcal{C}^1([0,+\infty),E^+)$ and $1\in\ker(\widetilde{\M}_{\widetilde{\B}})$. 
To verify that $1\in\dom_E\left(\left(\widetilde{\M}_{\widetilde{\B}}\right)^\star\right)$, we have to verify that $\frac{1}{u^\star}\in\dom_E(\MB^\star)$. Since $\frac{1}{u^\star}\in E$,
we only have to verify that $\frac{1}{u^\star}\in\dom_{\BUC(\overline{\O})}(\MB^\star)$. This follows from the definition of $\M$ in 
\eqref{definition_M} and from the equalities $\nabla\frac{1}{u^\star}=-\frac{1}{u^\star}\nabla u^\star$ and
$\Delta\frac{1}{u^\star}=-\frac{1}{u^\star}\Delta u^\star+\frac{1}{(u^\star)^2}|\nabla u^\star|^2$ when they make sense.
Roughly speaking, $\frac{1}{u^\star}$ has exactly the same regularity as $u^\star$, and this regularity is enough to be in the domain of
$\MB^\star$. We deliberately omit details.
Moreover, if Assumption \ref{ass:reduction_to_ODEs} is satisfied for \eqref{general_system}, then it is satisfied for \eqref{reduced_system}.

In other words, we are now concerned with the stability of $1_E$ for:
\begin{equation}\label{reduced_system}\tag{Red}
    \begin{dcases}
        \partial_t v = \widetilde{\M} v + \widetilde{b}v\left(1 - \int_{\O}v^p \widetilde{c} \upd\mu \right) & \text{in }(0,+\infty)\times\O, \\
        \widetilde{\B} v = 0 & \text{on }(0,+\infty)\times \Gamma, \\
        v(0)=v_0 & \text{in }\overline{\O}.
   \end{dcases}
\end{equation}

The linearization at $v=1_E$ reads:
\begin{equation}\label{reduced_linearized_system}\tag{RedLin}
    \begin{dcases}
           \partial_t h = \widetilde{\M}  h - p\widetilde{b}\int_{\O}h \widetilde{c}\upd\mu & \text{in }(0,+\infty)\times\O, \\
           \widetilde{\B} h = 0 & \text{on }(0,+\infty)\times \Gamma, \\
           h(0)=v_0-1 & \text{in }\overline{\O}.
    \end{dcases}
\end{equation}

When Assumption \ref{ass:reduction_to_ODEs} holds, the rank-one operator 
$\revision{h}\mapsto p\widetilde{b}\int_{\O}h\widetilde{c}\upd\mu$ can be rewritten $pb\otimes c$, 
with $\otimes$ the usual notation for the outer product in $\R^{N'}$. Below, we use this outer product notation even if 
Assumption \ref{ass:reduction_to_ODEs} does not hold.

The interior operator $-\widetilde{\M} +p\widetilde{b}\otimes \widetilde{c}$ is the sum of a monotone operator $-\widetilde{\M}$ 
and a rank-one operator $p\widetilde{b}\otimes\widetilde{c}$. 
Its spectral properties are intricate even when it is finite-dimensional \cite{Bierkens_Ran,Anehila_Ran_2022} 
and, as claimed in the introduction, there exist counterexamples to the linear stability.

\revision{Since 
\begin{equation*}
    (-\widetilde{\M} +p\widetilde{b}\otimes \widetilde{c},\widetilde{\B})=\left(\frac{1}{u^\star}\left(-\M-\mathcal{R}(u^\star)-u^\star (\partial_{u}\mathcal{R})_{u^\star}\right)(u^\star\operatorname{id}_E),\frac{1}{u^\star}\B (u^\star\operatorname{id}_E)\right)
\end{equation*}
the spectrum of $(-\widetilde{\M} +p\widetilde{b}\otimes \widetilde{c},\widetilde{\B})$ exactly coincides 
with the spectrum of $(-\M-\mathcal{R}(u^\star)-u^\star (\partial_{u}\mathcal{R})_{u^\star},\B)$.}

\revision{In the rest of the paper, we work only on the reduced system \eqref{reduced_system}. For readability, we drop
the $\widetilde{\cdot}$ of notations $\widetilde{\M}$, $\widetilde{\B}$, $\widetilde{b}$, $\widetilde{c}$.
We emphasize that, in this new reduced form, the additional property $\int_{\O} c \upd\mu =1$ is satisfied -- it will
be used repeatedly in the calculations.}

\section{Local asymptotic stability}

In this section, we study the existence of a neighborhood $\mathcal{V}\subset \operatorname{int}(E^+)$ of $1_E$ such 
that the singleton $\{1_E\}$ is the global attractor of the solutions of \eqref{reduced_system} (dropping the $\widetilde{\cdot}$) 
with initial data $u_0\in\mathcal{V}$.

\revision{Recall that, by virtue of standard results on semigroup theory, \textit{e.g.}
\cite[Theorem 2.1.]{Desch_Schappacher_1986}, the local asymptotic stability for \eqref{reduced_system} is implied by 
the exponential asymptotic stability of the equilibrium for the linearized system \eqref{reduced_linearized_system}.
Therefore, in what follows, we only prove the aforementioned exponential asymptotic stability.}

We recall the notation $H$ for the underlying Hilbert space. More precisely, the pair $(\I,H)$ denotes either 
$([N'],\C^{N'})$ if Assumption \ref{ass:reduction_to_ODEs} holds or $(\N^\star,L^2(\O,\C))$ otherwise. The set $\I$
will typically be the set of indexes for eigenvalues.
The inner product $(\cdot|\cdot)_{L^2(\O,\C)}$ restricted to $H^2$ coincides with the canonical inner product of $H$, 
so that in all cases we will work with the same scalar product.

\subsection{Examples of linear instability}\label{sec:instability}

\subsubsection{First counterexample: in the reduced coordinates, with $b$ and $c$ close to orthogonality}
Inspired by \cite[Example 3.4]{Anehila_Ran_2022}, we consider the discrete special case \eqref{discrete_ode_system} 
with $N'=3$, $p=1$ and
\begin{equation*}
    \M =
    \begin{pmatrix}
        -1 & 1 & 0 \\
        0 & -1 & 1 \\
        1 & 0 & -1
    \end{pmatrix}
    ,\quad 
    b=1.0002
    \begin{pmatrix}
        10 \\ .001 \\ .001
    \end{pmatrix}
    ,\quad 
     c =\frac{1}{1.0002}
    \begin{pmatrix}
        .0001 \\ 1 \\ .0001 
   \end{pmatrix}.
\end{equation*}
Then $-\M 1=0$ and $\sum_{i=1}^3 c _i=1$ as required and 
\begin{equation*}
   \operatorname{sp}(-\M +b\otimes c )\simeq\{-0.039215 + 1.80168 i,-0.039215 - 1.80168 i,3.08043\}.
\end{equation*}
Consequently, $0$ is unstable for the linearized system \eqref{reduced_linearized_system}. By standard results
\revision{on systems of ordinary differential equations}, the stationary solution $v=1$ is then unstable for the 
semilinear system \eqref{reduced_system}.

Note that, by changing $b$ into $b/p$, we also obtain a counterexample for every possible value of $p\geq 1$.

This counterexample shows the importance of \eqref{angle_condition} in Theorem \ref{thm:LAS}. Indeed, 
$\M$ is normal, the condition \eqref{spectrum_cone_condition} is satisfied (because $\operatorname{sp}(-\M) = \left\{ 0, \frac12(3\pm i\sqrt{3}) \right\}$), but the condition \eqref{angle_condition} is not satisfied:
\begin{equation*}
    b\cdot c + \frac{\|b\|_{1}\|c \|_{1}}{3}- \sqrt{\left(\|b\|_{2}^2-\frac{\|b\|_{1}^2}{3}\right)\left(\|c \|_{2}^2-\frac{\|c \|_{1}^2}{3}\right)} \simeq -3.3287.
\end{equation*}

More generally, by considering perturbations of the form
\begin{equation*}
    \begin{pmatrix}
        \begin{pmatrix}
        -1 & 1 & 0 \\
        0 & -1 & 1 \\
        1 & 0 & -1
    \end{pmatrix} & 0 \\
    0 & \mathbf{I}_{N'-3}
    \end{pmatrix}
    +\varepsilon\mathbf{1}\otimes\mathbf{1},\quad
    \begin{pmatrix}
        b \\
        0
    \end{pmatrix}
    +\varepsilon\mathbf{1},\quad
    \frac{1}{1+N'\varepsilon}\left(
    \begin{pmatrix}
        c \\
        0
    \end{pmatrix}
    +\varepsilon\mathbf{1}\right),
\end{equation*}
we deduce that, in the discrete special case \eqref{discrete_ode_system}, there exist counterexamples to the linear stability of $0$ for 
\eqref{reduced_linearized_system} for every $N'\geq 4$. 
On the contrary, in the cases $N'\in\{1,2\}$, the spectrum of $-\M +b\otimes c$ is unconditionally in the open right half-plane.

\subsubsection{Second counterexample: in the original coordinates, with symmetric mutations}
Inspired by this first counterexample, we found a new counterexample, where in addition the operator $\M$ in 
\eqref{general_system} before the reduction to \eqref{reduced_system} is self-adjoint. Thus we present this counterexample with the notations of \eqref{general_system}:
$N'=3$, $p=1$,
\begin{equation*}
    \M = 0.01
    \begin{pmatrix}
        -1 & 1 & 0 \\
        1 & -2 & 1 \\
        0 & 1 & -1
    \end{pmatrix},
\end{equation*}
\begin{equation*}
    b=
    \begin{pmatrix}
        .001 \\ .001 \\ 10
    \end{pmatrix}
    ,\quad 
     c =
    \begin{pmatrix}
        1 \\ .000001 \\ .00000001 
   \end{pmatrix}
    ,\quad
    u^\star = 1000
    \begin{pmatrix}
        1 \\ 100 \\ 10000
    \end{pmatrix}
    .
\end{equation*}
Then $r$ can be reconstructed from $u^\star$ as follows:
\begin{equation*}
    r = -\diag\left(\frac{1}{u^\star}\right)\M u^\star+b( c \cdot u^\star) \simeq
    \begin{pmatrix}
        .0102 \\ .0201 \\ 10002
    \end{pmatrix}.
\end{equation*}
Moving now to the reduced system \eqref{reduced_linearized_system}, we compute:
\begin{equation*}
    \begin{split}
    \widetilde{\M} & = \diag\left(\frac{1}{u^\star}\right)\M\diag(u^\star) - \diag\left(\frac{1}{u^\star}\right)\diag\left(\M u^\star\right) \\
    & \simeq
    \begin{pmatrix}
        -1 & 1 & 0 \\
        .0001 & -1.0001 & 1 \\
        0 & .0001 & -.0001
    \end{pmatrix},
    \end{split}
\end{equation*}
and, noting the identity $\tilde{b}\otimes\tilde{c}=b\otimes(\diag(u^\star)c)$, we compute the spectrum:
\begin{equation*}
   \operatorname{sp}(-\widetilde{\M} +b\otimes(\diag(u^\star)c))\simeq\{-9.4331 + 18.6661 i, -9.4331 - 18.6661 i, 22.8666\}.
\end{equation*}

In this counterexample, numerical calculations show that $-\widetilde{\M}$ is not normal, \eqref{spectrum_cone_condition}
is satisfied, but \eqref{angle_condition} is not satisfied.

Again, by changing $b$ into $b/p$, we also obtain a counterexample for every possible value of $p\geq 1$, and by considering appropriately chosen perturbations,
we also obtain counterexamples for every $N'\geq 4$.

\subsection{Conditional local asymptotic stability by projection}

Here we prove Theorem \ref{thm:LAS}, condition \ref{cond:LAS_projection}.

\begin{proof}
If $c$ is an adjoint principal eigenfunction of $-\M$, then, by taking the inner product of \eqref{reduced_linearized_system} 
with $c$, the function $f=\left( c \middle|h\right)$ satisfies
\begin{equation*}
    \begin{dcases}
        f' = -p\left( c \middle|b\right)f & \text{for all }t>0, \\
        f(0) = \left(c\middle|h(0,\cdot)\right). &
    \end{dcases}
\end{equation*}
Thus, substituting in \eqref{reduced_linearized_system},
\begin{equation*}
    \begin{dcases}
        \partial_t h - \M h = -p\left(c\middle|h(0,\cdot)\right)\upe^{-p\left( c \middle|b\right)t}b & \text{for all }t>0, \\
        \B h = 0. & 
    \end{dcases}
\end{equation*}

\revision{By using the semigroup $\T$, by a standard change of variable, for all $t>0$  $h$ can be rewritten : 
\begin{equation}\label{eq:Duhamel_formula_LAS_projection}
    h(t,\cdot)=\T_t h(0,\cdot)-p\left(c\middle|h(0,\cdot)\right)\int_0^t\upe^{-p\left( c \middle|b\right)(t-t')}\T_{t'}b(\cdot)\,\upd t'.
\end{equation}
Now, we we can check that $h(t,\cdot)\to 0$ in $E$ as $t\to+\infty$ by observing that by adding and subtracting terms, we can rewrite $h(t,\cdot)$ as follows:
\begin{equation*}
    \begin{split}
        h(t,\cdot) = & \ \T_t h(0,\cdot)-\frac{(c|h(0,\cdot))}{(c|1_E)}1_E \\
        & + \frac{(c|h(0,\cdot))}{(c|1_E)}1_E \\
        & -p(c|h(0,\cdot))\upe^{-p(c|b)t}\int_0^t\upe^{p(c|b)t'}\left(\T_{t'}b-\frac{(c|b)}{(c|1_E)}1_E\right)\upd t' \\
        & -p(c|h(0,\cdot))\int_0^t\upe^{p(c|b)(t'-t)}\frac{(c|b)}{(c|1_E)}1_E\upd t'.
    \end{split}
\end{equation*}
By using \eqref{def:exponential_asymptotic_stability_semigroup}, the terms of the right-hand side  on the first and third lines converge exponentially to $0$ 
in $E$  as $t\to+\infty$.  
The term of the right-hand side on the last line equals 
$-p(c|h(0))\frac{1-\upe^{-p(c|b)t}}{p(c|b)}\frac{(c|b)}{(c|1_E)}1_E=-(1-\upe^{-p(c|b)t})\frac{(c|h(0))}{(c|1_E)}1_E$.
Summing this term with the one on the second line, it remains $\upe^{-p(c|b)t}\frac{(c|h(0))}{(c|1_E)}1_E$,
that converges exponentially to $0$ in $E$  as $t\to+\infty$.
Thus 
\begin{equation*}
    \|h(t,\cdot)\|_E\to 0\quad\text{exponentially as }t\to+\infty,
\end{equation*}
}which ends the proof of the \revision{exponential asymptotic} stability of $0$ for \eqref{reduced_linearized_system} in 
the topology of $E$. 
\end{proof}

\subsection{Conditional local asymptotic stability by construction of an entropy}

Here we prove Theorem \ref{thm:LAS}, condition \ref{cond:LAS_entropy}.

\begin{proof}
By taking the inner product of \eqref{reduced_linearized_system} with $\frac{c}{b}h\in E$, we obtain
\begin{equation*}
   \frac12\frac{\upd}{\upd t}\left\|\sqrt{\frac{c}{b}}h\right\|_{L^2}^2+\left(h\middle|-\frac{c}{b}\M h\right) + p\left( h|c\right)^2=0.
\end{equation*}

Since $1_E\in\ker\left(\left(\frac{c}{b}\MB\right)^\star\right)$ by assumption, the scalar product $\left(h\middle|-\frac{c}{b}\M h\right)$ satisfies, 
by definition of $\sigma_2$ (\textit{cf.} \eqref{def:spectral_gap}), 
\begin{equation*}
    \begin{split}
        \left(h\middle|-\frac{c}{b}\M h\right) & = \left(h-(h|1_E)1_E\middle|-\frac{c}{b}\M (h-(h|1_E)1_E)\right) \\
        & \geq \sigma_2\|h-(h|1_E)1_E\|_{L^2}^2.
    \end{split}
\end{equation*}
Therefore, by virtue of the Parseval or Pythagorean theorem,
\begin{equation*}
    \left(h\middle|-\frac{c}{b}\M h\right) + p\left( h|c\right)^2
    \geq \min\left(\sigma_2,\frac{p}{\|1/ c \|_{L^\infty}^2}\right)\|h\|_{L^2}^2.
\end{equation*}

Moreover, 
\begin{equation*}
    \|h\|_{L^2}^2\geq\left\|\frac{c}{b}\right\|_{L^\infty}^{-1}\left\|\sqrt{\frac{c}{b}}h(t)\right\|_{L^2}^2.
\end{equation*}

Thus there exists a positive constant $C>0$ such that 
\begin{equation*}
    -\frac12\frac{\upd}{\upd t}\left\|\sqrt{\frac{c}{b}}h\right\|_{L^2}^2 \geq C\left\|\sqrt{\frac{c}{b}}h\right\|_{L^2}^2.
\end{equation*}
Then Gronwall's lemma implies that $0$ is \revision{exponentially asymptotically} stable for \eqref{reduced_linearized_system} in 
the topology of $H$.
\end{proof}

\subsection{Conditional local asymptotic stability by diagonalization}

Before proving the last part of Theorem \ref{thm:LAS}, we focus on a particular case of condition 
\revision{\ref{cond:LAS_diagonalization}} that corresponds to a standard method in stability analysis. 

\begin{proof}
If:
\begin{enumerate}
   \item the resolvent of $-\MB$ is normal and compact on $H$;
   \item $(f_k)_{k\in\I}$ is an orthonormal basis of eigenfunctions of $-\MB$ associated with complex eigenvalues $(\lambda_k)_{k\in\I}$ with $\lambda_1=0$ and $f_1=1_E/\sqrt{\mu(\O)}$;
   \item $b=\|b\|_E 1_E$;
   \item \revision{\eqref{spectral_gap_condition} holds true;}
\end{enumerate}
then  $h$ rewrites $h=\sum_{k\in\I}h_k f_k$ and we find the following system of ordinary
differential equations:
\begin{equation*}
    \begin{dcases}
        h_1'+\lambda_1 h_1+p\|b\|_E \sum_{k\in\I}h_k\left( c \middle|f_k\right)_{L^2(\O,\R)} =0, \\
        h_k'+\lambda_k h_k=0\quad\text{for all }k\geq 2.
    \end{dcases}
\end{equation*}
By virtue of \revision{\eqref{spectral_gap_condition}}, \revision{$\inf_{k\in\I\setminus\{1\}}(\operatorname{Re}(\lambda_k))>\lambda_1=0$}.
It follows that each $h_k$, $k\geq 2$, vanishes exponentially fast. Subsequently so does $h_1$. Therefore $0$ is 
\revision{exponentially asymptotically} stable for \eqref{reduced_linearized_system}. 
\end{proof}

\subsection{Conditional local asymptotic stability by completed diagonalization}

Now we prove Theorem \ref{thm:LAS}, condition \ref{cond:LAS_diagonalization}. In the particular case exposed above, the underlying
idea is to relate the eigenvalues of $-\MB$ and those of $(-\M+pb\otimes c,\B)$. 
Here we push further this idea, with a more precise spectral analysis of $(-\M+pb\otimes c,\B)$.
The proof will make it clear that the case $b\in\R 1_E$ plays a very special role, and in fact so does
the case $c\in\R 1_E$.

First, inspired by \cite{Dobosevych_Hryniv_2021}, we prove the following lemma on the Krein formula for the operator $(-\M+pb\otimes c,\B)$ -- also known
as the Sherman-Morrison formula in the context of matrix analysis.

For a linear operator $\mathcal{A}$ on the Hilbert space $H$, we denote as usual $\rho(\mathcal{A})$ its resolvent set.
The Krein formula uses the so-called \textit{characteristic function} $\lambda\in\rho(-\MB)\mapsto p\left( c \middle|(-\M-\lambda,\B)^{-1}b\right)+1$. Its set
of zeros is denoted below $\mathcal{Z}$.

\begin{lemma}[Krein formula]\label{lem:Krein_formula}
    The set inclusion 
    \begin{equation*}
        \rho(-\MB)\setminus\mathcal{Z}\subset\rho(-\M+pb\otimes c, \B )
    \end{equation*}
    holds true. The operators
    \begin{equation*}
        R_0(\lambda)=(-\M -\lambda,\B)^{-1},\quad R_p(\lambda)=(-\M+pb\otimes c -\lambda,\B)^{-1}
    \end{equation*}
    are well-defined for any $\lambda\in\rho(-\MB)\setminus\mathcal{Z}$.
    
    Furthermore, for any $\lambda\in\rho(-\MB)\setminus\mathcal{Z}$, 
    \begin{equation}\label{eq:Krein_formula}
        R_p(\lambda) = R_0(\lambda)-\frac{p\left(R_0(\lambda)b\right)\otimes\left(R_0^\star(\lambda) c \right)}{p\left( c \middle|R_0(\lambda)b\right)+1}.
    \end{equation}
\end{lemma}

\begin{proof}
    Let $\lambda\in\rho(-\MB)\setminus\mathcal{Z}$. Let $g\in\dom_H(R_0(\lambda))$, $f\in\dom_H(\M)$, and assume that:
    \begin{equation}\label{eq:invertibility_Krein_formula}
        g = (-\M+pb\otimes c -\lambda)f.
    \end{equation}

    Rewriting the equality as $g=(-\M-\lambda)f+pb\left( c \middle|f\right)$ and applying $R_0(\lambda)$ on both sides, we obtain:
    \begin{equation}\label{eq:pre_Krein_formula}
        R_0(\lambda)g = f + p\left( c \middle|f\right)R_0(\lambda)b.
    \end{equation}
    
    Taking the scalar product of \eqref{eq:pre_Krein_formula} with $c$ results in
    \begin{equation*}
        \left( c \middle|R_0(\lambda)g\right) = \left( c \middle|f\right) + p\left( c \middle|f\right)\left( c \middle|R_0(\lambda)b\right),
    \end{equation*}
    which, owing to $\lambda\notin\mathcal{Z}$, leads to:
    \begin{equation*}
        \left( c \middle|f\right) = \frac{\left( c \middle|R_0(\lambda)g\right)}{p\left( c \middle|R_0(\lambda)b\right)+1}.
    \end{equation*}

    Substituting this expression in \eqref{eq:pre_Krein_formula}, we find:
    \begin{equation}\label{eq:Krein_formula_in_the_proof}
        f = R_0(\lambda)g - \frac{p\left( c \middle|R_0(\lambda)g\right)}{p\left( c \middle|R_0(\lambda)b\right)+1}R_0(\lambda)b.
    \end{equation}

    A direct verification shows that $f\in\dom_H(-\M+pb\otimes c)=\dom_H(\M)$ and is indeed a solution of 
    \eqref{eq:invertibility_Krein_formula}.

    Therefore the operator $-\M+pb\otimes c -\lambda$ from $\dom_H(\M)$ onto $\dom_H(R_0(\lambda))$ is surjective.
    The operator $-\M+pb\otimes c -\lambda$ is also injective since $g=0$ in \eqref{eq:Krein_formula_in_the_proof} gives $f=0$.
    Thus the operator $-\M+pb\otimes c -\lambda$ is invertible from $\dom_H(\M)$ onto $\dom_H(R_0(\lambda))$.
    By definition of the resolvent set, $\dom_H(R_0(\lambda))$ is dense in $H$. Consequently, again by definition of the resolvent set,
    $\lambda\in\rho(-\M+pb\otimes c, \B)$, and the resolvent operator is:
    \begin{equation*}
        R_p(\lambda) = R_0(\lambda)-\frac{p\left(R_0(\lambda)b\right)\otimes\left(R_0^\star(\lambda) c \right)}{p\left( c \middle|R_0(\lambda)b\right)+1}
    \end{equation*}
    as claimed.
\end{proof}

Now we turn to the proof of Theorem \ref{thm:LAS}, condition \ref{cond:LAS_diagonalization}.

\begin{proof}
Assume that $-\MB$ has a compact normal resolvent. 
By virtue of the spectral theorem for compact normal operators, there exists an orthonormal
basis of eigenfunctions $(f_k)_{k\in\I}$ of $-\MB$ associated with complex eigenvalues $(\lambda_k)_{k\in\I}$ such that 
$\lambda_1=0$, $f_1=1/\sqrt{\mu(\O)}$, and for all $k\in\I$, $\operatorname{Re}(\lambda_k)\leq\operatorname{Re}(\lambda_{k+1})$.
By virtue of the Krein--Rutman theorem, $\operatorname{Re}(\lambda_k)>0$ for each $k\geq 2$.

Note that, for any $\lambda\in\rho(-\MB)$ and any $k\in\I$, $(-\M-\lambda)f_k=(\lambda_k-\lambda)f_k$, so that $R_0(\lambda)f_k=\frac{1}{\lambda_k-\lambda}f_k$.
Therefore
\begin{equation*}
    \begin{split}
        \left( c \middle|R_0(\lambda)b\right) & = \left(\sum_{k\in\I}\left( c \middle|f_k\right)f_k|\sum_{k\in\I}\frac{(b|f_k)}{\lambda_k-\lambda}f_k\right) \\
        & = \sum_{k\in\I}\frac{\overline{\left( c \middle|f_k\right)}(b|f_k)}{\lambda_k-\lambda}
    \end{split}
\end{equation*}
and $\mathcal{Z}$ can be rewritten as:
\begin{equation*}
    \mathcal{Z} = \left\{ \lambda\in\rho(-\MB)\ \middle|\ p\sum_{k\in\I}\frac{\overline{\left( c \middle|f_k\right)}(b|f_k)}{\lambda_k-\lambda}+1 = 0 \right\}.
\end{equation*}

Then the Krein formula \eqref{eq:Krein_formula} of Lemma \ref{lem:Krein_formula} shows that, 
in $\rho(-\MB)\setminus\mathcal{Z}$, 
the resolvent of $(-\M+pb\otimes c,\B)$ is a rank-one perturbation of the compact resolvent of $-\MB$.
Thus the resolvent of $(-\M+pb\otimes c,\B)$ is itself a compact operator. Its spectrum contains only eigenvalues. 
In view of the set inclusion $\rho(-\MB)\setminus\mathcal{Z}\subset\rho(-\M+pb\otimes c,\B)$, 
any eigenvalue of $(-\M+pb\otimes c,\B)$ is either an eigenvalue of $-\MB$ or an element of $\mathcal{Z}$.

To prove that $0$ is \revision{exponentially asymptotically} stable for \eqref{reduced_linearized_system}, it suffices then to 
prove that any eigenvalue of $(-\M+pb\otimes c,\B)$ has a positive real part \revision{and that this positivity is uniform.}

\revision{The proof will be structured as follows: first, we will prove that $0$ is not an eigenvalue of $(-\M+pb\otimes c,\B)$;
second, we show that elements in $\text{sp}(-\M+pb\otimes c,\B)\cap\text{sp}(-\MB)$  
have a uniformly positive real part; third, we show that elements in $\text{sp}(-\M+pb\otimes c,\B)\cap \rho(-\MB)$  have a positive real part; fourth, we  prove that this positivity is  uniform.}

\begin{proof}[Step one : $0\not\in \operatorname{sp}(-\M+pb\otimes c,\B )$]
\revision{If not, there exists a nontrivial 
$\varphi\in \ker(-\M+pb\otimes c,\B )$. Now, by projecting $\varphi$ on each $f_k$ and by using 
$\lambda_1=0$ and $f_1=1/\sqrt{\mu(\O)}$, we deduce:}
\begin{equation*}
    \begin{cases}
        \frac{p}{\sqrt{\mu(\O)}}(\varphi| c )(b|1) = 0, & \\
        p(\varphi| c )(b|f_k)+\lambda_k(\varphi|f_k) = 0 & \text{for all }k\in\I\setminus\{1\}.
    \end{cases}
\end{equation*}
From the first equality, we deduce that $\varphi\in c ^\perp$ and the other equalities rewrite 
\begin{equation*}
    \lambda_k(\varphi|f_k) = 0 \quad \text{for all }k\in\I\setminus\{1\},
\end{equation*}
which, since $\lambda_k\neq 0$, enforce  $(\varphi|f_k)=0$ for each $k\in\I$, \textit{i.e.} $\varphi=0$. 
Therefore $\ker(-\M+pb\otimes c,\B )=\{0\}$, contradicting our assumption.  Thus $0$ is not an eigenvalue of $(-\M+pb\otimes c,\B)$. 
\end{proof}

\begin{proof}[Step two: elements of $\text{sp}(-\M+pb\otimes c,\B)\cap\operatorname{sp}(-\MB)$ have a uniformly positive real part]
\revision{This is obvious when $\I$ is finite. Otherwise, by using standard spectral theory of compact normal operators, 
we can check that $(|\lambda_k|)_{k\in\N} \to +\infty$,
and by virtue of the assumption \eqref{spectrum_cone_condition}, $(|\lambda_k|)_{k\in\N} \to +\infty$
implies $\operatorname{Re}(\lambda_k)\to+\infty$. Hence the set of nonzero eigenvalues can be written as the partition
\begin{equation*}
    \operatorname{sp}(-\MB)\setminus\{0\}=\{\lambda_k\ |\ k\in\I\setminus\{1\},\ |\operatorname{Im}(\lambda_k)|\leq\operatorname{Re}(\lambda_k)\leq 1\}
    \cup\{\lambda_k\ |\ k\in\I,\ \operatorname{Re}(\lambda_k)> 1\}
\end{equation*}
and the first subset is finite whereas the second one is countably infinite. The uniform positivity of $(\operatorname{Re}(\lambda_k))_{k\in\I\setminus\{1\}}$ follows.}
\end{proof}

\begin{proof}[Step three: $\operatorname{sp}(-\M+pb\otimes c,\B)\cap\rho(-\MB)\subset \{\lambda \in \C\ |\ \operatorname{Re}(\lambda)>0\}$]
\revision{
From previous considerations, we know that any element of $\operatorname{sp}(-\M+pb\otimes c,\B)\cap\rho(-\MB)$ 
is an element of $\mathcal{Z}$.}
By contradiction, assume there exists $\lambda\in ((-\infty,0]+i\R)\cap\mathcal{Z}$. 
By assumption, $\lambda\in\rho(-\MB)$. 

On one hand, 
\begin{equation*}
    \begin{split}
        \sum_{k\in\I}\frac{\overline{\left( c \middle|f_k\right)}(b|f_k)}{\lambda_k-\lambda} & = -\frac{\left( c \middle|1\right)(b|1)}{\mu(\O)\lambda}+\sum_{k\in\I\setminus\{1\}}\frac{\overline{\left( c \middle|f_k\right)}(b|f_k)}{\lambda_k-\lambda} \\
        & = -\frac{\| c \|_{L^1}\|b\|_{L^1}}{\mu(\O)\lambda}+\sum_{k\in\I\setminus\{1\}}\frac{\overline{\left( c \middle|f_k\right)}(b|f_k)}{\lambda_k-\lambda} \\
        & = -\frac{\| c \|_{L^1}\|b\|_{L^1}}{\mu(\O)\lambda}+\sum_{k\in\I}\frac{\overline{\left( c -\left( c \middle|f_1\right)f_1 \middle| f_k\right)}\left(b-\left(b \middle| f_1 \right)f_1 \middle| f_k\right)}{\lambda_k-\lambda}.
    \end{split}
\end{equation*}
Thus, by definition of $\mathcal{Z}$,
\begin{equation*}
    \begin{split}
        \frac{1}{p}-\frac{\| c \|_{L^1}\|b\|_{L^1}\overline{\lambda}}{\mu(\O)|\lambda|^2} & = -\sum_{k\in\I}\frac{\overline{\left( c -\left( c \middle| f_1 \right)f_1 \middle| f_k \right)}\left(b-\left(b \middle| f_1\right)f_1 \middle| f_k \right)}{\lambda_k-\lambda} \\
        & = -\left(  c -\left( c \middle| f_1 \right)f_1 \middle| R_0(\lambda)\left(b-\left(b\middle|f_1\right)f_1\right) \right).
    \end{split}
\end{equation*}
Multiplying by $-2\lambda$, we get:
\begin{equation*}
    \frac{-2\lambda}{p}+2\frac{\| c \|_{L^1}\|b\|_{L^1}}{\mu(\O)} = 2\lambda\left(  c -\left( c \middle| f_1 \right)f_1 \middle| R_0(\lambda)\left( b-\left(b\middle| f_1\right)f_1\right) \right).
\end{equation*}
On the other hand, by orthogonality,
\begin{equation*}
    ( c |b) = \frac{\| c \|_{L^1}\|b\|_{L^1}}{\mu(\O)} + \left(  c -( c |f_1)f_1 | b-(b|f_1)f_1 \right).
\end{equation*}
Summing the two equalities and using the identity $\operatorname{id}+2\lambda R_0(\lambda) = -2\M R_0(\lambda)-\operatorname{id}$, we find:
\begin{multline}\label{eq:Autissier_proof_eigenvalue_identity}
    \frac{-2\lambda}{p} + ( c |b) + \frac{\| c \|_{L^1}\|b\|_{L^1}}{\mu(\O)} = \Big(  c -( c |f_1)f_1 | \left(\left(-2\M R_0(\lambda)-\operatorname{id}\right)\left(b-(b|f_1)f_1\right)\right) \Big).
\end{multline}

Since the left-hand side has a positive real part (recall $\operatorname{Re}(-\lambda)\geq 0$), the right-hand side cannot be zero, 
whence neither $ c -( c |f_1)f_1$ nor $b-(b|f_1)f_1$ are zero. 

This immediately gives a contradiction if $c\in\R 1_E$ or if $b\in\R 1_E$, and this ends the proof in these two cases \footnote{
In fact, in these two cases, $\mathcal{Z}$ is exactly the singleton $\left\{p\|c\|_{L^1}\|b\|_{L^1}/\mu(\O)\right\}$}.
From now on, let us assume $c\notin\R 1_E$, $b\notin\R 1_E$, and recall that 
\eqref{spectrum_cone_condition} and \eqref{angle_condition} hold true.

Let us define the following quantities:
\begin{equation*}
    u=\frac{ c -( c |f_1)f_1}{\| c -( c |f_1)f_1\|_{L^2}},\quad v=\frac{b-(b|f_1)f_1}{\|b-(b|f_1)f_1\|_{L^2}}, 
    \quad \mathcal{A}_\lambda = -2\M R_0(\lambda)-\operatorname{id}.
\end{equation*}
By using the equalities $\| c -( c |f_1)f_1\|_{L^2}=\sqrt{\| c \|_{L^2}^2-( c |f_1)^2}$ and 
$\|b-(b|f_1)f_1\|_{L^2}=\sqrt{\|b\|_{L^2}^2-(b|f_1)^2}$, the equality \eqref{eq:Autissier_proof_eigenvalue_identity} rewrites : 
\begin{equation}\label{eq:Autissier_proof_eigenvalue_identity_reformulated}
    \frac{\frac{-2\lambda}{p} + ( c |b) + \frac{\| c \|_{L^1}\|b\|_{L^1}}{\mu(\O)}}{\sqrt{\| c \|_{L^2}^2-( c |f_1)^2}\sqrt{\|b\|_{L^2}^2-(b|f_1)^2}} = \left( u | \mathcal{A}_\lambda v \right).
\end{equation}

Note that  \eqref{angle_condition} implies that
\begin{equation*}
    1< \frac{( c |b) + \frac{\| c \|_{L^1}\|b\|_{L^1}}{\mu(\O)}}{\sqrt{\| c \|_{L^2}^2-( c |f_1)^2}\sqrt{\|b\|_{L^2}^2-(b|f_1)^2}}.
\end{equation*}
Hence, in order to reach a contradiction, and therefore discard the existence of $\lambda$, 
it only remains to show that the real part of $\left( u | \mathcal{A}_\lambda v \right)$ is lesser than 
or equal to $1$. 

Since $u$ and $v$ are real-valued and since \eqref{eq:Autissier_proof_eigenvalue_identity} already
gives the positivity of this quantity, we only have to show that 
$|\left(u \middle| \frac12\left(\mathcal{A}_\lambda+\overline{\mathcal{A}_\lambda}\right) v\right)|\leq 1$. 
Note that, by definition, $u,v\in 1^\perp$. Therefore, by virtue of the Cauchy--Schwarz inequality, it suffices to prove that,
for all $w\in \dom_E(\MB)$, real-valued, such that $w\in 1^\perp$ and $\|w\|_{L^2}=1$, we have $\|\frac12\left(\mathcal{A}_\lambda+\overline{\mathcal{A}_\lambda}\right) w\|_{L^2}\leq 1$.

Now we set:
\begin{equation*}
    \mathcal{B}_\lambda = \frac12\left(\mathcal{A}_\lambda+\overline{\mathcal{A}_\lambda}\right) = -\M R_0(\lambda)-\overline{\M R_0(\lambda)}-\operatorname{id}.
\end{equation*}

For any $k\in\I\setminus\{1\}$,
\begin{equation*}
    \begin{split}
        \mathcal{B}_\lambda f_k = -\M(-\M-\lambda)^{-1}f_k +\overline{-\M R_0(\lambda)\overline{f_k}}-f_k 
        & = \frac{\lambda_k}{\lambda_k-\lambda}f_k+\overline{\frac{\overline{\lambda_k}}{\overline{\lambda_k}-\lambda}\overline{f_k}}-f_k \\
        & = \left( \frac{\lambda_k}{\lambda_k-\lambda}+\frac{\lambda_k}{\lambda_k-\overline{\lambda}}-1\right)f_k \\
        & = \left( \lambda_k\frac{2\lambda_k-\lambda-\overline{\lambda}}{(\lambda_k-\lambda)(\lambda_k-\overline{\lambda})}-1 \right)f_k \\
        & = \left( 2\lambda_k\frac{\lambda_k-\operatorname{Re}(\lambda)}{\lambda_k^2-2\lambda_k\operatorname{Re}(\lambda)+|\lambda|^2}-1 \right)f_k \\
        & = \frac{\lambda_k^2-|\lambda|^2}{\lambda_k^2+2\lambda_k\operatorname{Re}(-\lambda)+|\lambda|^2} f_k.
    \end{split}
\end{equation*}
Hence the operator $\mathcal{B}_\lambda$ still admits each $f_k$ as an eigenfunction, and, for any $w\in \dom_E(\MB)\cap 1^\perp$ real-valued,
\begin{equation}\label{eq:Autissier_proof_L2_decomposition_of_Blambdaw}
    \mathcal{B}_\lambda w = \sum_{k\in\I\setminus\{1\}}(w|f_k)\mathcal{B}_\lambda f_k = \sum_{k\in\I\setminus\{1\}}(w|f_k)\frac{\lambda_k^2-|\lambda|^2}{\lambda_k^2+2\lambda_k\operatorname{Re}(-\lambda)+|\lambda|^2} f_k.
\end{equation}

Consequently, by virtue of the Parseval or Pythagorean theorem, it suffices to prove:
\begin{equation}\label{eq:Autissier_proof_modulus_eigenvalues_condition}
    \left|\frac{\lambda_k^2-|\lambda|^2}{\lambda_k^2+2\lambda_k\operatorname{Re}(-\lambda)+|\lambda|^2}\right|\leq 1\quad\text{for all }k\in\I\setminus\{1\}.
\end{equation}

\revision{Fix momentarily $k\in\I\setminus\{1\}$ and define 
$\zeta = \frac{\lambda_k^2-|\lambda|^2}{\lambda_k^2+2\lambda_k\operatorname{Re}(-\lambda)+|\lambda|^2}$,
$\alpha=\operatorname{Re}(\lambda_k)> 0$, $\beta=\operatorname{Im}(\lambda_k)$, then:
\begin{equation*}
    \zeta = \frac{\alpha^2-\beta^2-|\lambda|^2+i2\alpha\beta}{\alpha^2-\beta^2+2\alpha \operatorname{Re}(-\lambda)+|\lambda|^2+i2\beta(\alpha+\operatorname{Re}(-\lambda))}.
\end{equation*}
Since $\alpha\geq 0$ and $\operatorname{Re}(-\lambda)\geq 0$ are non negative, 
\begin{equation*}
\begin{split}
        |\zeta|^2 & = \frac{(\alpha^2-\beta^2-|\lambda|^2)^2+4\alpha^2\beta^2}{(\alpha^2-\beta^2+|\lambda|^2+2\alpha\operatorname{Re}(-\lambda))^2+4\beta^2(\alpha+\operatorname{Re}(-\lambda))^2} \\
        & \le \frac{(\alpha^2-\beta^2-|\lambda|^2)^2+4\alpha^2\beta^2}{(\alpha^2-\beta^2+|\lambda|^2+2\alpha\operatorname{Re}(-\lambda))^2+4\beta^2\alpha^2}.
        \end{split}
\end{equation*}
Since $\alpha\geq 0$, 
\begin{equation*}
    \alpha^2-\beta^2+|\lambda|^2+2\alpha\operatorname{Re}(-\lambda)\geq \alpha^2-\beta^2+|\lambda|^2=\alpha^2-\beta^2-|\lambda|^2+2|\lambda|^2.
\end{equation*}
By \eqref{spectrum_cone_condition}, $\alpha\geq\beta$, whence $\alpha^2-\beta^2-|\lambda|^2+2|\lambda|^2\geq 0$,
whence 
\begin{equation*}
\begin{split}
    (\alpha^2-\beta^2+|\lambda|^2+2\alpha\operatorname{Re}(-\lambda))^2 & \geq (\alpha^2-\beta^2-|\lambda|^2+2|\lambda|^2)^2 \\
   & = (\alpha^2-\beta^2-|\lambda|^2)^2+4|\lambda|^2(\alpha^2-\beta^2-|\lambda|^2 )+4|\lambda|^4 \\
   & = (\alpha^2-\beta^2-|\lambda|^2)^2+4|\lambda|^2(\alpha^2-\beta^2) \\
   & \geq (\alpha^2-\beta^2-|\lambda|^2)^2,
\end{split}
\end{equation*}
whence finally $|\zeta|^2\leq 1$.}
\end{proof}

\begin{proof}[Step four: elements of $\operatorname{sp}(-\M+pb\otimes c,\B)\cap \rho(-\MB)$ have a uniformly positive real part]
\revision{To conclude, it remains to prove that the positivity of the real parts of eigenvalues of 
$(-\M+pb\otimes c,\B)$ that are in $\rho(-\MB)$ is uniform. Such eigenvalues are elements of $\mathcal{Z}$.}

\revision{When $c\in\R 1_E$ or $b\in\R 1_E$, previous calculations show that $\mathcal{Z}=\left\{p\|c\|_{L^1}\|b\|_{L^1}/\mu(\O)\right\}$
and the proof is ended.}

\revision{Otherwise, if neither $c$ nor $b$ are in $\R 1_E$, 
\eqref{eq:Autissier_proof_eigenvalue_identity_reformulated} and \eqref{angle_condition}
together show that there exists $\varepsilon>0$ such that every such eigenvalue $\lambda$ with 
sufficiently small positive real part $\operatorname{Re}(\lambda)\in(0,\overline{\gamma}_0]$ (with a threshold
$\overline{\gamma}_0$ dependent on $\varepsilon$) satisfies:
\begin{equation}\label{eq:Autissier_uniform_positivity_contradiction_1}
    \operatorname{Re}\left(\left( u\middle| \mathcal{A}_\lambda v\right)\right)>1+\varepsilon.
\end{equation}
But on the other hand, previous calculations also show that
\begin{equation}\label{eq:Autissier_uniform_positivity_contradiction_2}
    \operatorname{Re}\left(\left( u\middle| \mathcal{A}_\lambda v\right)\right) =\left|\left(u\middle|\mathcal{B}_\lambda v\right)\right| \leq  \sup_{k\in\I\setminus\{1\}}\left|\frac{\lambda_k^2-|\lambda|^2}{\lambda_k^2+2\lambda_k\operatorname{Re}(-\lambda)+|\lambda|^2} \right|.
\end{equation}
The right-hand side is finite; this is obvious if $\I$ is finite and, otherwise, it follows from 
\eqref{eq:Autissier_proof_L2_decomposition_of_Blambdaw} and the Parseval theorem.
Also, the previous calculations show that, from $\lambda\in\rho(-\MB)$, the quantity
$\lambda_k^2+2\lambda_k\operatorname{Re}(-\lambda)+|\lambda|^2$ is indeed nonzero.}

\revision{By contradiction, assume that, for all $\gamma>0$, there exists 
$\lambda_\gamma\in\operatorname{sp}(-\M+pb\otimes c,\B)\cap \rho(-\MB)$
such that $\operatorname{Re}(\lambda_\gamma)\in(0,\gamma]$. }

\revision{With the notations 
$\zeta_{k,\gamma} = \frac{\lambda_k^2-|\lambda_\gamma|^2}{\lambda_k^2-2\lambda_k\operatorname{Re}(\lambda_\gamma)+|\lambda_\gamma|^2}$,
$\alpha_k=\operatorname{Re}(\lambda_k)> 0$, $\beta_k=\operatorname{Im}(\lambda_k)$, the previous calculations show that:
\begin{equation*}
    |\zeta_{k,\gamma}|^2 = \frac{(\alpha_k^2-\beta_k^2-|\lambda_\gamma|^2)^2+4\alpha_k^2\beta_k^2}{\left(\alpha_k^2-\beta_k^2+|\lambda_\gamma|^2-2\alpha_k\operatorname{Re}(\lambda_\gamma))^2+4\beta_k^2(\alpha_k-\operatorname{Re}(\lambda_\gamma)\right)^2}.
\end{equation*}
Let us prove that, provided $\gamma$ is sufficiently small, 
\begin{equation}\label{eq:Autissier_uniform_positivity_contradiction_3}
    \sup_{k\in\I\setminus\{1\}}|\zeta_{k,\gamma}|^2\leq 1+\varepsilon,
\end{equation}
where $\varepsilon$ is given by \eqref{eq:Autissier_uniform_positivity_contradiction_1} and does not depend on $\gamma$.
It is sufficient to prove that, provided $\gamma$ is sufficiently small, for any $k\in\I\setminus\{1\}$,
\begin{equation*}
    (\alpha_k^2-\beta_k^2-|\lambda_\gamma|^2)^2+4\alpha_k^2\beta_k^2
    \leq \left(1+\varepsilon\right)\left(\alpha_k^2-\beta_k^2+|\lambda_\gamma|^2-2\alpha_k\operatorname{Re}(\lambda_\gamma))^2+4\beta_k^2(\alpha_k-\operatorname{Re}(\lambda_\gamma)\right)^2.
\end{equation*}
Subtracting the left-hand side from the right-hand side and rearranging terms, we find that the previous
inequality is equivalent to the following one:
\begin{multline*}
    0\leq \varepsilon(\alpha_k^2+\beta_k^2)^2
    +(4+2\varepsilon)(\alpha_k^2-\beta_k^2)|\lambda_\gamma|^2
    +\varepsilon|\lambda_\gamma|^4
    -4(1+\varepsilon)\alpha_k|\lambda_\gamma|^2\operatorname{Re}(\lambda_\gamma) \\
    -4(1+\varepsilon)\alpha_k(\alpha_k^2+\beta_k^2)\operatorname{Re}(\lambda_\gamma)
    + 4\left(1+\varepsilon\right)(\alpha_k^2+\beta_k^2)\operatorname{Re}(\lambda_\gamma)^2.
\end{multline*}
}

\revision{Now, we distinguish two subcases: either $|\lambda_\gamma|\leq\gamma$ or $|\lambda_\gamma|>\gamma$.}

\revision{If $|\lambda_\gamma|\leq\gamma$, then, recalling $\alpha_k^2+\beta_k^2=|\lambda_k|^2$ and $0<\alpha_k\leq|\lambda_k|$,
it suffices to prove that
\begin{equation*}
    \begin{split}
        0 & \leq \varepsilon|\lambda_k|^4-4(1+\varepsilon)|\lambda_k|\gamma^3-4(1+\varepsilon)|\lambda_k|^3\gamma \\
        & = |\lambda_k|\left(\varepsilon|\lambda_k|^3-4(1+\varepsilon)\gamma^3-4(1+\varepsilon)|\lambda_k|^2\gamma\right).
    \end{split}
\end{equation*}
If $\gamma\leq 1$, then $-\gamma^3\geq-\gamma$ and subsequently it suffices to prove that
\begin{equation*}
    0\leq \varepsilon|\lambda_k|^3-4(1+\varepsilon)(1+|\lambda_k|^2)\gamma.
\end{equation*}
Defining
\begin{equation*}
    \overline{\gamma}_1=\min\left[1,\inf_{k\in\I\setminus\{1\}}\frac{\varepsilon|\lambda_k|^3}{4(1+\varepsilon)(1+|\lambda_k|^2)}\right]
\end{equation*}
and choosing any $\gamma\in(0,\overline{\gamma}_1]$, we deduce indeed \eqref{eq:Autissier_uniform_positivity_contradiction_3}.}

\revision{If, on the contrary, $|\lambda_\gamma|>\gamma$, then it suffices to prove that
\begin{equation*}
    0\leq \varepsilon(|\lambda_k|^4+|\lambda_\gamma|^4)-4(1+\varepsilon)\gamma|\lambda_k|(|\lambda_k|^2+|\lambda_\gamma|^2).
\end{equation*}
By equivalence of norms in $\R^2$, there exists a universal constant $C>0$, independent of $\gamma$ and $k$, such that it suffices to prove
\begin{equation*}
    \begin{split}
        0 & \leq \varepsilon(|\lambda_k|^4+|\lambda_\gamma|^4)-4(1+\varepsilon)\gamma C(|\lambda_k|^4+|\lambda_\gamma|^4)^{3/4} \\
        & =(|\lambda_k|^4+|\lambda_\gamma|^4)^{3/4}\left(\varepsilon(|\lambda_k|^4+|\lambda_\gamma|^4)^{1/4}-4(1+\varepsilon)\gamma C\right).
    \end{split}
\end{equation*}
Now, using $|\lambda_k|^4+|\lambda_\gamma|^4\geq|\lambda_k|^4$, defining
\begin{equation*}
    \overline{\gamma}_2=\inf_{k\in\I\setminus\{1\}}\frac{\varepsilon|\lambda_k|^4}{4(1+\varepsilon)C}
\end{equation*}
and choosing any $\gamma\in(0,\overline{\gamma}_2]$ to deduce again \eqref{eq:Autissier_uniform_positivity_contradiction_3}.}

\revision{Finally, choosing any $\gamma\in(0,\min_{i\in\{0,1,2\}}\overline{\gamma}_i]$ and bringing together \eqref{eq:Autissier_uniform_positivity_contradiction_1}, \eqref{eq:Autissier_uniform_positivity_contradiction_2} and 
\eqref{eq:Autissier_uniform_positivity_contradiction_3}, we deduce 
\begin{equation*}
    1+\varepsilon<\operatorname{Re}\left(\left( u\middle| \mathcal{A}_\lambda v\right)\right)\leq\sqrt{1+\varepsilon}.
\end{equation*}
This is a contradiction. 
Hence no eigenvalue of $(-\M+pb\otimes c,\B)$ that is also an element of $\mathcal{Z}$ can be in the set 
$(0,\min(\overline{\gamma}_1,\overline{\gamma}_2))+i\R$.}
\end{proof}

This ends the proof of Theorem \ref{thm:LAS}, condition \ref{cond:LAS_diagonalization}.
\end{proof}

\begin{remark}
The condition \eqref{angle_condition} is scale-invariant: it is equivalent to
\begin{equation*}
    \left(\frac{b}{\|b\|_{L^1}} \middle| \frac{c}{\|c\|_{L^1}}\right)+\frac{1}{\mu(\O)}> \sqrt{\left(\left\|\frac{b}{\|b\|_{L^1}}\right\|_{L^2}^2-\frac{1}{\mu(\O)}\right)\left(\left\|\frac{c}{\|c\|_{L^1}}\right\|_{L^2}^2-\frac{1}{\mu(\O)}\right)}.
\end{equation*}

It is in particular satisfied if $b$ is constant (this is the preceding special case where the linearized equation can be diagonalized) or if $c$ is constant. 

On the contrary, it is violated if, for instance, 
$\O = (0,2\pi)$, $p=1$, $b-\varepsilon$ is $x\mapsto\max(\sin x,0)$, $c-\varepsilon$ is $x\mapsto\max(\sin(\pi+x),0)$. 
In such a case, $(b|c)=2\pi\varepsilon^2 +4\varepsilon$, $\|b\|_{L^1}=\|c\|_{L^1}=2\left(\pi\varepsilon+1\right)$,
$\|b\|_{L^2}=\|c \|_{L^2}=\sqrt{2\pi\varepsilon^2+4\varepsilon+\frac{\pi}{2}}$. 
As $\varepsilon\to 0$, the left-hand side in \eqref{angle_condition} converges to $\frac{2}{\pi}$, whereas the right-hand side converges to
$\frac{\pi}{2}-\frac{2}{\pi}$. Indeed $\frac{2}{\pi}<\frac{\pi}{2}-\frac{2}{\pi}$. 

Hence the inequality \eqref{angle_condition} can be roughly understood as ``the angle between the directions of $b$ and $c$ is sufficiently small'' 
(hence the symbol $\angle$).

Interestingly, this geometrical alignment notion also appears in condition \ref{cond:LAS_entropy} of Theorem \ref{thm:LAS}, with the ratio between $c$ and $b$.
\end{remark}

\subsection{No condition in Theorem \ref{thm:LAS} is implied by another condition}

We actually prove this claim directly for the reduced system \eqref{reduced_system}, which can be conceived as a particular case of the general system
\eqref{general_system} obtained by adding the constraints $r=b$ and $\int_{\O} c =1$. 
In this case, $u^\star = 1$, $\MB=\widetilde{\M}_{\widetilde{\B}}$, and $c=\widetilde{c}$.

\begin{proof}[Proof that \ref{cond:LAS_projection} does not imply \ref{cond:LAS_entropy} or \ref{cond:LAS_diagonalization}]
Let $p=1$, $\O = (-1,1)$, $c =1/|\O|$, $\B=0$ and 
\begin{equation*}
    \M:u\mapsto\left(x\mapsto\int_{\O}J(x,y)u(y)\upd y - u(x)\right)
\end{equation*}
with $J(x,y)=J(y,x)$ and $\int_{\O}J(x,y)\upd y=1$. 
$\M$ is a self-adjoint nonlocal diffusion operator, whose resolvent is not compact.

\ref{cond:LAS_projection} is indeed true since $\frac{1}{|\O|}\M^\star 1=0$.

With an appropriate choice of $b$, $-\frac{c}{b}\M$ is not normal, whence \ref{cond:LAS_entropy} is false.

By lack of compactness, \ref{cond:LAS_diagonalization} are false. 
\end{proof}

\begin{proof}[Proof that \ref{cond:LAS_entropy} does not imply \ref{cond:LAS_projection} or \ref{cond:LAS_diagonalization}]
Let $p=1$. Choose $b$ and $c$, satisfying Neumann boundary conditions, such that \eqref{angle_condition} fails. 
Let $\MB=\frac{b}{c}\Delta_N$, where $\Delta_N$ is the Neumann Laplacian. Then $\MB^\star:u\mapsto\Delta_N\left(\frac{b}{c}u\right)$.

\ref{cond:LAS_entropy} is indeed true since $-\frac{c}{b}\MB=\Delta_N$.

By choice of $b$, it is not constant, so that $\frac{b}{c}c = b\notin\ker\Delta_N$, whence \ref{cond:LAS_projection} is false.

Since \eqref{angle_condition} is false and $b$ and $c$ are not constant, \ref{cond:LAS_diagonalization} is also false.
\end{proof}

\begin{proof}[Proof that \ref{cond:LAS_diagonalization} does not imply \ref{cond:LAS_projection} or \ref{cond:LAS_entropy}]
Let $p=1$, $\O=(0,2\pi)$, $b:x\mapsto 2+\cos x$, $c=1+\varepsilon b$ with $\varepsilon\in(0,1)$ so small that, by continuity, \eqref{angle_condition} holds true.
Let $\MB=\Delta_N$. 

\ref{cond:LAS_diagonalization} is indeed true by this particular choice.

\ref{cond:LAS_projection} is false because the non-constant function $c$ is not in the kernel of $-\MB^\star=-\Delta_N$.

\ref{cond:LAS_entropy} is false because  $-\left(\frac{1}{b}+\varepsilon\right)\Delta_N$ is not normal. Indeed, the test function $u=1$ yields:
\begin{equation*}
    \left(\frac{1}{b}+\varepsilon\right)\Delta\Delta\left[\left(\frac{1}{b}+\varepsilon\right)1\right]-\Delta\left[\left(\frac{1}{b}+\varepsilon\right)^2\Delta\right]1
    = \left(\frac{1}{b}+\varepsilon\right)\Delta^2\frac{1}{b} \neq 0.
\end{equation*}
\end{proof}

\section{Global asymptotic stability}

In this section, \revision{we prove conditions under which all solutions of \eqref{reduced_system} (dropping the $\widetilde{\cdot}$) 
coming from initial conditions $u_0\in\dom_{E^+}(\MB)\setminus\{0\}$ converge to $1_E$.}

We recall also the notation $(\I,H)$ for $([N'],\C^{N'})$ if Assumption \ref{ass:reduction_to_ODEs} holds or 
for $(\N^\star,L^2(\O,\C))$ otherwise. 

\subsection{Conditional global asymptotic stability by projection}

Here we prove Theorem \ref{thm:GAS}, condition \ref{cond:GAS_projection}.

\begin{proof}
    The proof follows exactly the same lines as for the local stability. 
    The function $f=\left( c |v\right)$ satisfies $f'=f-f^2$ and $f(t_0)>0$
    for some $t_0\geq 0$ given by Assumption \ref{ass:irreducibility}, whence $f\to 1$. 
    \revision{More precisely, 
    \begin{equation*}
        1-f(t)=\frac{(1-f(t_0))\upe^{-(t-t_0)}}{(1-f(t_0))\upe^{-(t-t_0)}+f(t_0)}\quad\text{for all }t\geq t_0.
    \end{equation*}
    Subsequently, $v$ satisfies in $(t_0,+\infty)\times\O$:
    \begin{equation*}
        \begin{split}
            \partial_t (v-1_E) & = \M (v-1_E) + v- (c|v) v \\
            & = \M(v-1_E)+(1-f)v \\
            & = \M (v-1_E) + (1-f)(v-1_E)+(1-f)1_E.
        \end{split}
    \end{equation*}
    Let $g:t\in[t_0,+\infty)\mapsto \upe^{\int_{t_0}^t (f(t')-1)\upd t'}$. There exists $t_1>t_0$ such that, for any $t>t_1$,
    \begin{equation*}
        \begin{split}
            g(t) & \leq g(t_1)\upe^{\int_{t_1}^t 2\frac{|1-f(t_0)|}{f(t_0)}\upe^{-(t'-t_0)}\upd t'} \\
            & =g(t_1)\upe^{2\frac{|1-f(t_0)|}{f(t_0)}\upe^{t_0}(\upe^{t_1}-\upe^{t})} \\
            & \leq g(t_1)\upe^{2\frac{|1-f(t_0)|}{f(t_0)}\upe^{t_0+t_1}}.
        \end{split}
    \end{equation*}
    The function $w:t\mapsto g(t)(v(t)-1_E)$ satisfies
    \begin{equation*}
        \begin{dcases}
            \partial_t w - \M w = g(1-f)1_E & \text{for }t>t_0, \\
            \B w = 0, & \\
            w(t_0)=v(t_0)-1_E, &
        \end{dcases}
    \end{equation*}
    and so we are back to the situation encountered in the proof of Theorem \ref{thm:LAS}, condition
    \ref{cond:LAS_projection}. Thanks to the same semigroup argument using 
    \eqref{def:exponential_asymptotic_stability_semigroup}, we obtain the (exponential) convergence in $E$ of $w$ to $0$.
    }
    Subsequently, $v$ converges \revision{(exponentially as well)} in $E$ to $1_E$.
\end{proof}

\begin{remark}
    Another proof of the exact same result would use the line-sum-symmetry of the operator $c\MB$, namely the property $(c\MB)^\star 1_E=(c\MB) 1_E$, 
    in the manner of \cite{Berestycki_Nadin_Perthame_Ryzhik,Girardin_Griette_2020,Cantrell_Cosner_Lou_2012,Cantrell_Cosner_Lou_Ryan_2012}.
    Since $c\MB$ is line-sum-symmetric, it satisfies the following important inequality:
    \begin{equation*}
        \left(\frac{1}{w} \middle| c\M w\right) \geq \left(1_E \middle| c\M 1_E\right) = 0\quad\text{for all }w\in \operatorname{int}(E^+),
    \end{equation*}
    with in addition by virtue of Assumption \ref{ass:irreducibility} equality if and only $w\in\R 1_E$. 
    Thus, by taking the scalar product between \eqref{reduced_system} and $c\left(1-\frac{1}{v(t)}\right)$ 
    (here it is required that $v(t)\in\operatorname{int}(E^+)$ after some time $t_0\geq 0$ sufficiently large, which is not true
    in general under Assumption \ref{ass:irreducibility} but is true under stronger irreducibility assumptions), we obtain:
    \begin{equation*}
        \frac{\upd}{\upd t}\left(c|v-\ln v\right) = -\left(\frac{1}{v} \middle| c\M v\right) + \left(c|v-1\right)\left(c|1-v\right)\leq 0.
    \end{equation*}
    Hence the function $t\mapsto \left(c|v(t)-\ln v(t)\right)$ is nonincreasing. It is also nonnegative, and therefore it converges to a nonnegative
    limit. With some work, it can be verified that this limit is $0$, whence $v(t)\to 1$.

    This argument has proved successful in \cite{Berestycki_Nadin_Perthame_Ryzhik,Girardin_Griette_2020} for different competition operators, 
    circulant instead of rank-one. With rank-one competition operators, unfortunately, it turns out to be just a less direct proof.
\end{remark}

\subsection{Conditional global asymptotic stability by construction of an entropy}

The entropy method requires first a change of unknown $v=1+h$ in the semilinear system, which leads to:
\begin{equation}\label{eq:entropy_change_of_unknown}
   \partial_t h = \M h + b + bh - b(1+h)\int_{\O}(1+h)^p c \upd\mu .
\end{equation}

For general $p\geq 1$, the term $(1+h)^p$ can be expanded as a binomial series, convergent for instance if $\|h\|_{L^\infty}<1$.
For general $p\in\N$, the expansion is given by the binomial formula. Below, we will focus on the proof of Theorem \ref{thm:GAS}, condition \ref{cond:GAS_entropy}. 
Following and adapting the calculations could give an idea of the obstacles encountered when considering $p>1$. We do not detail these obstacles explicitly.

\begin{proof}
The proof contains several steps.

\begin{proof}[Step one: additional $L^1$ and $L^2$ estimates on the solution]
Taking the scalar product between \eqref{reduced_system} and $c/b$ and using the orthogonality
between $-\frac{c}{b}\M (h(t)+1)$ and $1$, we obtain:
\begin{equation*}
    \frac{\upd}{\upd t}\left\|\frac{c}{b}v(t)\right\|_{L^1} = \|cv(t)\|_{L^1}(1-\|cv(t)\|_{L^1}).
\end{equation*}
Let $m:t\mapsto\left\|cv(t)\right\|_{L^1}/\|\frac{c}{b}v(t)\|_{L^1}$.
This function is clearly uniformly bounded and uniformly positive, with bounds that only depend on $c$ and $b$. 
We can rewrite:
\begin{equation*}
    \frac{\upd}{\upd t}\left\|\frac{c}{b}v(t)\right\|_{L^1} = m(t)\left\|\frac{c}{b}v(t)\right\|_{L^1}\left(1-m(t)\left\|\frac{c}{b}v(t)\right\|_{L^1}\right),
\end{equation*}
frow where it follows similarly, by comparison with logistic-type ordinary differential equations, that 
\begin{equation*}
    0<\liminf_{t\to+\infty}\left\|\frac{c}{b}v(t)\right\|_{L^1}\leq\limsup_{t\to+\infty}\left\|\frac{c}{b}v(t)\right\|_{L^1}<+\infty,
\end{equation*}
with bounds that only depend on $c$ and $b$. 

Therefore, there exists $\delta\in(0,1)$, dependent only on $c$, $b$ and $v_0$, such that, for all $t\geq 0$,
\begin{equation*}
    \|v(t)\|_{L^2}\geq \delta,\quad \delta\leq\|v(t)\|_{L^1}\leq\frac{1}{\delta}.
\end{equation*}
\end{proof}

\begin{proof}[Step 2: exhibition of an entropy]
Now, equipped with these new estimates, we go back to \eqref{eq:entropy_change_of_unknown}. The binomial formula yields:
\begin{equation*}
    \partial_t h = \M h - b\left(c|h\right) - b(c|h)h.
\end{equation*}
Thus
\begin{equation}\label{eq:reduced_nonlinear_system_entropy_method}
    \frac{c}{b}\partial_t h + \left(-\frac{c}{b}\M \right)h +c\left(c|h\right) + c(c|h)h = 0.
\end{equation}

Let 
\begin{equation*}
   \mathcal{E}:h\in\dom_E(\MB)\mapsto \left(h\middle| -\frac{c}{b}\M h\right) + \left(c|h\right)^2+(c|h)\left(c|h^2\right)
\end{equation*}
and
\begin{equation*}
    D=\left\{ h\in \dom_E(\MB)\ \middle|\ h+1\geq 0,\ \|h+1\|_{L^2}\geq \delta,\ \delta\leq\|h+1\|_{L^1}\leq\frac{1}{\delta} \right\}.
\end{equation*}

Any $h\in D$ can be uniquely decomposed as $h=-1+\tau z$, with $\tau\geq \delta/\sqrt{\mu(\O)}$ and $z\in E^+$ such that $\|z\|_{L^2}=\sqrt{\mu(\O)}$. We denote 
\begin{equation*}
    D'=\left\{(\tau,z)\in(0,+\infty)\times\{\|z\|_{L^2}=\sqrt{\mu(\O)}\}\ \middle|\ \tau z+1\in D\right\}.
\end{equation*}
Remark that, for any $(\tau,z)\in D'$,
\begin{equation*}
    \frac{\delta}{\tau}\leq \|z\|_{L^1}\leq\frac{1}{\delta\tau},\quad \|z\|_{L^2}=\sqrt{\mu(\O)},
\end{equation*}
and, by virtue of the Cauchy--Schwarz inequality,
\begin{equation*}
    (c|z) = \int_{\O}cz\upd\mu\leq\sqrt{\int_{\O}cz^2\upd\mu}\sqrt{\int_{\O}c\upd\mu} = (c|z^2)^{1/2}.
\end{equation*}

Substituting $h=\tau z-1$ into the definition of $\mathcal{E}$, and using the orthogonality between $1$ and $-\frac{c}{b}\M z$,
we obtain, for any $(\tau,z)\in D'$:
\begin{equation*}
\begin{split}
    \mathcal{E}(\tau z-1) & = \tau^2\left(z\middle| -\frac{c}{b}\M z\right)+\left(\left( c \middle|\tau z-1\right)+\left( c \middle|(\tau z-1)^2\right)\right)\left( c \middle|(\tau z-1)\right). \\
    & = \tau^2\left(z\middle| -\frac{c}{b}\M z\right)+\left(\tau\left( c \middle|z\right)+\tau^2\left( c \middle|z^2\right)-2\tau\left( c \middle|z\right)\right)\left(\tau\left( c \middle|z\right)-1\right) \\
    & = \tau^2\left(z\middle| -\frac{c}{b}\M z\right)+\tau\left(\tau\left( c \middle|z^2\right)-\left( c \middle|z\right)\right)\left(\tau\left( c \middle|z\right)-1\right).
\end{split}
\end{equation*}
\end{proof}

\begin{proof}[Step 3: a continuous lower bound]
Our goal is to bound $\mathcal{E}(\tau z-1)$ from below by a continuous function of $(\tau,z)$ that is nonnegative in $D'$
and whose zero set in $D'$ is exactly $\{(0,1),(1,1)\}$.

\revision{Just as in} the proof of Theorem \ref{thm:LAS}, \revision{we can use an auxiliary
unit vector $f_1=\frac{1}{\sqrt{\mu(\O)}}$ to calculate:
\begin{equation*}
    \begin{split}
        \left(z\middle| -\frac{c}{b}\M z\right) & = \left(z\middle| -\frac{c}{b}\M(z-(z|f_1)f_1)\right)\\
        & = \left(z-(z|f_1)f_1)\middle| -\frac{c}{b}\M(z-(z|f_1)f_1)\right) + (z|f_1)\left(f_1\middle| -\frac{c}{b}\M(z-(z|f_1)f_1)\right) \\
        & = \left(z-(z|f_1)f_1)\middle| -\frac{c}{b}\M(z-(z|f_1)f_1)\right) + (z|f_1)\left((-\frac{c}{b}\M)^\star f_1\middle| (z-(z|f_1)f_1)\right) \\
        & = \left(z-(z|f_1)f_1)\middle| -\frac{c}{b}\M(z-(z|f_1)f_1)\right) \\
        & \geq \sigma_2\|z-(z|f_1)f_1\|_{L^2}^2 \\
        & = \sigma_2\left\|z-\frac{\|z\|_{L^1}}{\sqrt{\mu(\O)}}f_1\right\|_{L^2}^2 \\
        & = \sigma_2\left(\|z\|_{L^2}^2-\left\|\frac{\|z\|_{L^1}}{\sqrt{\mu(\O)}}f_1\right\|_{L^2}^2\right) \\
        & = \sigma_2\|z\|_{L^2}^2-\sigma_2\frac{\|z\|_{L^1}^2}{\mu(\O)}\|f_1\|_{L^2}^2 \\
        & = \sigma_2\mu(\O)-\sigma_2\frac{\|z\|_{L^1}^2}{\mu(\O)}
    \end{split}
\end{equation*}
From this}, we deduce:
\begin{equation*}
    \frac{\mathcal{E}(\tau z-1)}{\tau} \geq (c|z) +\tau\left(\revision{\sigma_2\mu(\O)-\frac{\sigma_2}{\mu(\O)}\|z\|_{L^1}^2}-(c|z^2)-(c|z)^2\right) +\tau^2(c|z)(c|z^2).
\end{equation*}

Let
\begin{equation*}
    G:z\mapsto \revision{\sigma_2\mu(\O)-\frac{\sigma_2}{\mu(\O)}\|z\|_{L^1}^2}-(c|z^2)-(c|z)^2
\end{equation*}
and
\begin{equation*}
    \mathcal{G}:(\tau,z) \mapsto \left(c| z\right)+G(z)\tau+(c|z^2)(c|z)\tau^2.
\end{equation*}

In $D'\cap\{G(z)\geq 0\}$, 
\begin{equation*}
    \mathcal{G}(\tau,z) \geq \mathcal{G}(0,z) = (c|z) \geq \inf_{\O}(c)\frac{\delta}{\tau}.
\end{equation*}

To study the complementary subset, $D'\cap\{G(z)< 0\}$, we rewrite the second-order polynomial of the variable $\tau$:
\begin{equation}\label{eq:entropy_method_expression_polynomial}
    \mathcal{G}(\tau,z) = \frac{4(c|z^2)(c|z)^2-G(z)^2}{4(c|z^2)(c|z)}+\frac{\left(2(c|z^2)(c|z)\tau+G(z)\right)^2}{4(c|z^2)(c|z)}
\end{equation}

It will also be useful to define the following quantity:
\begin{equation}\label{def:sigma_2_prime}
    \sigma_2' = \revision{\mu(\O)}\left(\sup_{\substack{z\in E^+\setminus \R 1_E \\ \|z\|_{L^2}=\sqrt{\mu(\O)} \\ G(z)<0}} \frac{\sqrt{(c|z^2)}-(c|z)}{\sqrt{\mu(\O)^2-\|z\|_{L^1}^2}}\right)^2.
\end{equation}
Remark that the inequalities $\sqrt{(c|z^2)}-(c|z)\geq 0$ and
$\mu(\O)^2-\|z\|_{L^1}^2\geq 0$ are both Cauchy--Schwarz inequalities. The equality case of the first one
corresponds to proportionality between $\sqrt{c}$ and $\sqrt{c}z$, namely between $z$ and $1$. 
Similarly, the equality case of the second one corresponds also to proportionality between $z$ and $1$. 
Thus both quantities $\sqrt{(c|z^2)}-(c|z)$ and $\mu(\O)^2-\|z\|_{L^1}^2$ are in some sense measures
of the distance between $z$ and $1$ in the sphere $\{\|z\|_{L^2}=\sqrt{\mu(\O)}\}$.
Therefore, on one hand, $\sigma_2'\geq 0$ by construction, and on the other hand, $\sigma_2'<\sigma_2$ by the assumption \eqref{entropy_method_condition}. 

Fix momentarily $(\tau,z)\in D'\cap\{G(z)< 0\}$.

First, we focus on the first term on the right-hand side of \eqref{eq:entropy_method_expression_polynomial}.
In this ratio, the denominator is bounded from above by 
$4(\max c)^{2}\mu(\O)\|z\|_{L^1}\leq 4(\max c)^{2}\mu(\O)\frac{1}{\delta\tau}$.
Moreover,
\begin{equation*}
    4(c|z^2)(c|z)^2-G(z)^2 = \left(2\sqrt{(c|z^2)}(c|z)-G(z)\right)\left(2\sqrt{(c|z^2)}(c|z)+G(z)\right)
\end{equation*}
Since $G(z)<0$, the first term on the right-hand side is bounded from below by $2(\min c)^{3/2}\sqrt{\mu(\O)}\|z\|_{L^1}\geq 2(\min c)^{3/2}\sqrt{\mu(\O)}\frac{\delta}{\tau}$. 
We investigate now the second term on the right-hand side. By definition of $G$,
\begin{equation*}
    \begin{split}
        2\sqrt{(c|z^2)}(c|z)+G(z) & = \revision{\frac{\sigma_2}{\mu(\O)}}(\mu(\O)^2-\|z\|_{L^1}^2)-\left(\sqrt{(c|z^2)}-(c|z)\right)^2 \\
        & = \left(\sqrt{\revision{\frac{\sigma_2}{\mu(\O)}}}\sqrt{\mu(\O)^2-\|z\|_{L^1}^2}+\sqrt{(c|z^2)}-(c|z)\right) \\
        & \quad\times\left(\sqrt{\revision{\frac{\sigma_2}{\mu(\O)}}}\sqrt{\mu(\O)^2-\|z\|_{L^1}^2}-\left(\sqrt{(c|z^2)}-(c|z)\right)\right)
    \end{split}
\end{equation*}
The first term on the right-hand side above is bounded from below by $\sqrt{\revision{\frac{\sigma_2}{\mu(\O)}}}\sqrt{\mu(\O)^2-\|z\|_{L^1}^2}$.
As for the second term on the right-hand side, we recognize the quantity that appears in \eqref{def:sigma_2_prime}, whence it 
is bounded from below by $\left(\sqrt{\revision{\frac{\sigma_2}{\mu(\O)}}}-\sqrt{\revision{\frac{\sigma_2'}{\mu(\O)}}}\right)\sqrt{\mu(\O)^2-\|z\|_{L^1}^2}$.
Consequently, there exists a positive constant $C>0$, that does not depend on $(\tau,z)$, such that:
\begin{equation*}
    \frac{4(c|z^2)(c|z)^2-G(z)^2}{4(c|z^2)(c|z)} \geq C\left(\mu(\O)^2-\|z\|_{L^1}^2\right)\geq 0.
\end{equation*}

For the sake of future shortened notations, we define:
\begin{equation*}
    J(\tau z)=\sqrt{\mu(\O)}\frac{\mu(\O)\|\tau z\|_{L^2}^2-\|\tau z\|_{L^1}^2}{\|\tau z\|_{L^2}}=\tau\left(\mu(\O)^2-\|z\|_{L^1}^2\right).
\end{equation*}
We remark that, by virtue of the Parseval or Pythagorean theorem, $\mu(\O)\|\tau z\|_{L^2}^2-\|\tau z\|_{L^1}^2=\mu(\O)\|\tau z-(\tau z|f_1)f_1\|_{L^2}^2=\mu(\O)\|(\operatorname{id}-P_1)\tau z\|_{L^2}^2$, where $P_1$ is the orthogonal projection on $\R 1_E$.

We focus now on the second term on the right-hand side of \eqref{eq:entropy_method_expression_polynomial},
\begin{equation*}
    \begin{split}
        \frac{\left(2(c|z^2)(c|z)\tau+G(z)\right)^2}{4(c|z^2)(c|z)} & = \frac{\frac{1}{\tau}\left(2(c|\tau^2 z^2)(c|\tau z)+G(z)\tau^2\right)^2}{4(c|\tau^2 z^2)(c|\tau z)} \\
        & = \frac{\frac{1}{\tau^2}\left(\frac{2(c|\tau^2 z^2)(c|\tau z)+G(z)\tau^2}{\tau}\right)^2}{4(c|z^2)(c|z)}.
    \end{split}
\end{equation*}
The denominator is again bounded from above by $4(\max c)^{2}\mu(\O)\frac{1}{\delta\tau}$. To rewrite the numerator in a more convenient way, we define:
\begin{equation*}
    K(\tau z)= \sqrt{\mu(\O)}\frac{2(c|(\tau z)^2)(c|\tau z)-(c|(\tau z)^2)-(c|\tau z)^2}{\|\tau z\|_{L^2}},
\end{equation*}
\revision{so that
\begin{equation*}
    \begin{split}
        \frac{2(c|\tau^2 z^2)(c|\tau z)+G(z)\tau^2}{\tau} & = K(\tau z)+\frac{\tau^2\left(\sigma_2\mu(\O)-\frac{\sigma_2}{\mu(\O)}\|z\|_{L^1}^2\right)}{\tau} \\
        & = K(\tau z)+\frac{\sigma_2}{\mu(\O)}J(\tau z)
    \end{split}
\end{equation*}}
Hence, by definition of $G$, it follows that
\begin{equation*}
    \frac{\left(2(c|z^2)(c|z)\tau+G(z)\right)^2}{4(c|z^2)(c|z)} \geq \frac{1}{4(\max c)^{2}\mu(\O)\frac{1}{\delta}\tau}\left(K(\tau z)+\revision{\frac{\sigma_2}{\mu(\O)}}J(\tau z)\right)^2
\end{equation*}
When $z=1$, the quantity $K(\tau z)+\revision{\frac{\sigma_2}{\mu(\O)}}J(\tau z)$ reduces to $2\tau^2-2\tau=2\tau(\tau-1)$.

Recall that $\tau z=h+1$.
Then, up to decreasing the constant $C>0$ in a way that does not depend on $h$, the quantity $\mathcal{E}(h)$ can be 
bounded from below in $D$ as follows:
\begin{equation*}
    \begin{split}
        \mathcal{E}(h) & \geq \tau\mathcal{G}\left(\tau,z)\right) \\
        & \geq C\min\left[1,J(h+1)+\left(K(h+1)+\revision{\frac{\sigma_2}{\mu(\O)}}J(h+1)\right)^2\right].
    \end{split}
\end{equation*}
The function
\begin{equation*}
    \mathcal{F}:h\in D\mapsto \min\left[1,J(h+1)+\left(K(h+1)+\revision{\frac{\sigma_2}{\mu(\O)}}J(h+1)\right)^2\right]
\end{equation*}
is nonnegative and, since $-1\notin D$, its only zero is $h=0$.
\end{proof}

\begin{proof}[Step 4: uniform positivity of the entropy away from $h=0$]
Now, back to $h$ being a solution of \eqref{eq:reduced_nonlinear_system_entropy_method}, we deduce immediately from
\begin{equation*}
    -\frac{\upd}{\upd t}\left\|\sqrt{\frac{c}{b}}h(t)\right\|_{L^2}^2=\mathcal{E}(h(t))\quad\text{and}\quad h(t)\in D\quad\text{for all }t> 0
\end{equation*}
that $t\mapsto\left\|\sqrt{\frac{c}{b}}h(t)\right\|_{L^2}^2$ is nonincreasing and converges to some limit $a\geq 0$. 
In particular, 
\begin{equation*}
    \sqrt{\frac{a}{\sup_{\O}(c/b)}}\leq\|h(t)\|_{L^2}\leq\sqrt{\frac{1}{\inf_{\O}(c/b)}}\left\|\sqrt{\frac{c}{b}}h(0)\right\|_{L^2}\quad\text{for all }t\geq 0.
\end{equation*}
Let 
\begin{equation*}
    \underline{R}=\sqrt{\frac{a}{\sup_{\O}(c/b)}},\quad \overline{R}=\sqrt{\frac{1}{\inf_{\O}(c/b)}}\left\|\sqrt{\frac{c}{b}}h(0)\right\|_{L^2},
\end{equation*}
\begin{equation*}
    \widetilde{D}=D\cap \overline{B_{H}(-1_E,\overline{R})\setminus B_{H}(-1_E,\underline{R})}.
\end{equation*}
The set $\widetilde{D}$ is closed and bounded in $H$. When Assumption \ref{ass:reduction_to_ODEs} holds true and $H=\C^{N'}$, 
the uniform positivity of $\mathcal{F}$ in $\widetilde{D}$ is obvious by continuity. But when Assumption \ref{ass:reduction_to_ODEs} does not hold
and $H=L^2(\O,\C)$, the set $\widetilde{D}$ is not compact, and we resort instead to a careful study of the preimage 
$\mathcal{F}^{-1}([0,\varepsilon])\cap \widetilde{D}$ for small values of $\varepsilon>0$.

Let $\varepsilon\in(0,1)$ and $h\in\mathcal{F}^{-1}([0,\varepsilon])\cap\widetilde{D}$.
Necessarily,
\begin{equation*}
    \begin{cases}
        J(h+1)\leq\varepsilon, \\
        \left|K(h+1)+\revision{\frac{\sigma_2}{\mu(\O)}}J(h+1)\right|\leq\sqrt{\varepsilon},
    \end{cases}
\end{equation*}
from where it follows that
\begin{equation*}
    \begin{cases}
        J(h+1)\leq\varepsilon, \\
        |K(h+1)|\leq\sqrt{\varepsilon}+\revision{\frac{\sigma_2}{\mu(\O)}}\varepsilon.
    \end{cases}
\end{equation*}
We decompose $h+1=\alpha 1_E+w$, with $\alpha\in\R$ and $w\in (1_E)^\perp$. 
Since $h\in\widetilde{D}$, $\|h+1\|_{L^2}=\sqrt{\mu(\O)\alpha^2+\|w\|_{L^2}^2}\in[\delta,\overline{R}+1]$.
On one hand, by definition of $J$, 
\begin{equation*}
    \frac{\mu(\O)^{3/2}}{\overline{R}+1}\|w\|_{L^2}^2\leq\varepsilon.
\end{equation*}
On the other hand, $(c|(h+1)^2)=\alpha^2+2\alpha(c|w)+(c|w^2)$ and $(c|h+1)=\alpha+(c|w)$, so that:
\begin{equation*}
    \begin{split}
        K(h+1) & = \sqrt{\frac{\mu(\O)\alpha^2}{\mu(\O)\alpha^2+\|w\|_{L^2}^2}}K(\alpha 1_E)+\sqrt{\frac{\|w\|_{L^2}^2}{\mu(\O)\alpha^2+\|w\|_{L^2}^2}}K(w) \\
        & \quad +\frac{2\alpha\sqrt{\mu(\O)}\left((3\alpha-2)(c|w)+2(c|w)^2+(c|w^2)\right)}{\sqrt{\mu(\O)\alpha^2+\|w\|_{L^2}^2}}
    \end{split}
\end{equation*}
where $K(\alpha 1_E)=2\alpha(\alpha-1)$. Subsequently,
\begin{equation*}
    \begin{split}
        \frac{2\sqrt{\mu(\O)}\alpha^2|\alpha-1|}{\overline{R}+1} & \leq \sqrt{\varepsilon}+\revision{\frac{\sigma_2}{\mu(\O)}}\varepsilon+|K(w)|+2\left(|3\alpha-2||(c|w)|+2(c|w)^2+(c|w^2)\right) \\
        & \leq \sqrt{\varepsilon}+\revision{\frac{\sigma_2}{\mu(\O)}}\varepsilon+|K(w)| +2\left(3\frac{\overline{R}+1}{\sqrt{\mu(\O)}}+2\right)\max(c)\sqrt{\mu(\O)}\|w\|_{L^2} \\
        & \quad +4\max(c)^2\mu(\O)\|w\|_{L^2}^2+2\max(c)\|w\|_{L^2}^2.
    \end{split}
\end{equation*}
Remark that, since $h\in\widetilde{D}\subset D$, $\alpha=\frac{1}{\mu(\O)}(1|\alpha 1+w)=\frac{\|h+1\|_{L^1}}{\mu(\O)}\geq\frac{\delta}{\mu(\O)}$.
By Lipschitz continuity of $K$ in a neighborhood of $0$, we obtain the existence of a new positive constant $C'$, 
independent of the choice of $h$ and $\varepsilon$, such that:
\begin{equation*}
    |\alpha-1|\leq C'(\sqrt{\varepsilon}+\varepsilon).
\end{equation*}
It follows that there exists another positive constant, still denoted $C'>0$ for simplicity, such that 
\begin{equation*}
    \mu(\O)(\alpha-1)^2+\|w\|_{L^2}^2\leq C'\varepsilon.
\end{equation*}
Consequently,
\begin{equation*}
    \mathcal{F}^{-1}([0,\varepsilon])\cap\widetilde{D}\subset \overline{B\left(0,\sqrt{C'\varepsilon}\right)}.
\end{equation*}
In particular, and this is what we will use below, if $a>0$ and $\sqrt{C'\varepsilon}<\underline{R}$, then
$\mathcal{F}^{-1}([0,\varepsilon])\cap\widetilde{D}=\emptyset$.
\end{proof}

\begin{proof}[Step 5: conclusion of the proof]
Assume by contradiction that $a>0$. Then there exists $\varepsilon>0$ such that $\mathcal{F}^{-1}([0,\varepsilon])\cap\widetilde{D}=\emptyset$. Consequently, 
\begin{equation*}
    \underline{\mathcal{E}} = \inf_{t> 0}\mathcal{F}(h(t))\geq \inf_{w\in\widetilde{D}}\mathcal{F}(w))\geq\varepsilon> 0.
\end{equation*}
But then, by integrating over $(0,+\infty)$, we obtain: 
\begin{equation*}
    +\infty>\left\|\sqrt{\frac{c}{b}}h(0)\right\|_{L^2}^2-a\geq \int_{0}^{+\infty}\underline{\mathcal{E}}\upd t =+\infty,
\end{equation*}
an immediate contradiction. Thus $a=0$.

The equality $a=0$ precisely means that $t\mapsto\sqrt{\frac{c}{b}}h(t)$ converges in $H$ to $0$. Subsequently,
so does $t\mapsto h(t)$. In other words, $v(t)=h(t)+1\to 1$ in $H$. 
\end{proof}

This ends the proof of Theorem \ref{thm:GAS}, condition \ref{cond:GAS_entropy}.
\end{proof}

\begin{remark}
The condition \eqref{entropy_method_condition} can be either true or false, as shown by taking $c=1/\mu(\O)$ and 
$\sigma_2$ very large or very small. 
In particular, from $\sqrt{\mu(\O)\|h\|_{L^2}^2-\|h\|_{L^1}^2}\geq\sqrt{\mu(\O)}\|h\|_{L^2}-\|h\|_{L^1}$, it
follows that \eqref{entropy_method_condition} is true if there exists $\sigma_2'\in(0,\sigma_2)$ such that, for all $h\in E$,
\begin{equation*}
    \|1\|_{L^2(c\upd\mu)}\|h\|_{L^2(c\upd\mu)}-\|h\|_{L^1(c\upd\mu)}\leq\sqrt{\revision{\frac{\sigma_2'}{\mu(\O)}}}\left(\|1\|_{L^2(\upd\mu)}\|h\|_{L^2(\upd\mu)}-\|h\|_{L^1(\upd\mu)}\right).
\end{equation*}
The constants $\|1\|_{L^2(c\upd\mu)}$ and $\|1\|_{L^2(\upd\mu)}$ are the norms of the continuous embeddings 
of $L^2(\O,c\upd\mu)$ into $L^1(\O,c\upd\mu)$ and $L^2(\O,\upd\mu)$ into $L^1(\O,\upd\mu)$ respectively. 
Consequently, an appropriate relation between these two continuous embeddings and the spectral gap $\sigma_2$ is sufficient to imply the condition \eqref{entropy_method_condition}.
\end{remark}

\subsection{Conditional global asymptotic stability by diagonalization}

Here we prove Theorem \ref{thm:GAS}, condition \ref{cond:GAS_diagonalization}.

\begin{proof}
    By virtue of Theorem \ref{thm:instability_zero}, for all $t\geq 0$, $v(t)$ solution of \eqref{reduced_system} satisfies $\|v(t)\|_{L^1}=(1|v(t))>0$.

    Let $(f_k)_{k\in\I}$ be an orthonormal basis of eigenfunctions of $-\MB$ associated with complex eigenvalues $(\lambda_k)_{k\in\I}$ with $\lambda_1=0$. For each $k\in\I$, we denote $v_k=(v|f_k)f_k$. Note that\revision{
    \begin{equation*}
        \min_{x\in\overline{\O}}v_1(t,x)=(v(t)|f_1)\min_{x\in\overline{\O}}f_1(x)=\frac{\|v(t)\|_{L^1}}{\mu(\O)}\min_{\O}1_E=\frac{\|v(t)\|_{L^1}}{\mu(\O)}>0\quad \text{for all }t\geq 0.
    \end{equation*}}
    
    Following \cite{Girardin_2017}, we consider for all $t\geq 0$ and each $k\in\I$ the quantity
    \begin{equation*}
        z_k(t) = \left\|\frac{v_k(t)}{v_1(t)}\right\|_{L^2}^2 = \int_{\O}\left|\frac{v_k(t)}{v_1(t)}\right|^2\upd\mu,
    \end{equation*}
    whose time derivative writes:
    \begin{equation*}
        \begin{split}
            z_k' & = \int_{\O}\partial_t\left(\operatorname{Re}\left(\frac{v_k}{v_1}\right)^2+\operatorname{Im}\left(\frac{v_k}{v_1}\right)^2\right)\upd\mu \\
            & = \int_{\O}\partial_t\left(\left(\frac{\operatorname{Re}(v_k)}{v_1}\right)^2+\left(\frac{\operatorname{Im}(v_k)}{v_1}\right)^2\right)\upd\mu \\
            & = \int_{\O}2\frac{\operatorname{Re}(v_k)}{v_1}\partial_t\frac{\operatorname{Re}(v_k)}{v_1}+2\frac{\operatorname{Im}(v_k)}{v_1}\partial_t\frac{\operatorname{Im}(v_k)}{v_1}\upd\mu \\
            & = \int_{\O}2\frac{\operatorname{Re}(v_k)}{v_1}\left(\frac{\operatorname{Re}(\partial_t v_k)}{v_1}-\frac{\operatorname{Re}(v_k)\partial_t v_1}{v_1^2}\right)+2\frac{\operatorname{Im}(v_k)}{v_1}\left(\frac{\operatorname{Im}(\partial_t v_k)}{v_1}-\frac{\operatorname{Im}(v_k)\partial_t v_1}{v_1^2}\right)\upd\mu \\
            & = \int_{\O}2\operatorname{Re}\left(\frac{\overline{v_k}}{v_1}\left(\frac{\partial_t v_k}{v_1}-\frac{v_k\partial_t v_1}{v_1^2}\right)\right)\upd\mu
        \end{split}
    \end{equation*}
    
    The function $v_k$ satisfies in $[0,+\infty)$:
    \begin{equation*}
        \partial_t v_k = (-\lambda_k+1) v_k -v_k\int_{\O}v^pc\upd\mu,
    \end{equation*}
    so that
    \begin{equation*}
        \frac{\partial_t v_k}{v_1}-\frac{v_k\partial_t v_1}{v_1^2} = -\lambda_k\frac{v_k}{v_1}.
    \end{equation*}

    Thus, in $[0,+\infty)$,
    \begin{equation*}
        z_k' = -2\operatorname{Re}(\lambda_k) z_k
    \end{equation*}
    namely $z_k(t)=\exp(-2\operatorname{Re}(\lambda_k) t)z_k(0)$ for all $t\geq 0$.

    \revision{For all $t\geq 0$ and each $k\in\I$,
    \begin{equation*}
        \begin{split}
            \|v_k(t)\|_{L^2}^2 & \leq \|v_1(t)\|_{L^\infty}^2 z_k(t) \\
            & = (v(t)|f_1)^2\|f_1\|_{L^\infty}^2\upe^{-2\operatorname{Re}(\lambda_k) t}z_k(0) \\
            & = \frac{\|v(t)\|_{L^1}^2}{\mu(\O)}\frac{1}{\mu(\O)}\upe^{-2\operatorname{Re}(\lambda_k) t}z_k(0).
        \end{split}
    \end{equation*}
    Moreover, $z_1(0)=\mu(\O)$ by definition and, for each $k\in\I\setminus\{1\}$, the expression of
    $z_k(0)$ can be simplified thanks to the observation that $v_1(0)=\frac{\|v(0)\|_{L^1}}{\mu(\O)}1_E$ does not depend on $x$:
    \begin{equation*}
        z_k(0) = \frac{\mu(\O)^2\|v_k(0)\|_{L^2}^2}{\|v(0)\|_{L^1}^2}.
    \end{equation*}
    Consequently, for all $t\geq 0$,
    \begin{equation*}
        \begin{dcases}
            \|v_1(t)\|_{L^2}^2 \leq \frac{1}{\mu(\O)}\|v(t)\|_{L^1}^2, & \\
            \|v_k(t)\|_{L^2}^2 \leq \frac{\|v_k(0)\|_{L^2}^2}{\|v(0)\|_{L^1}^2}\upe^{-2\operatorname{Re}(\lambda_k) t}\|v(t)\|_{L^1}^2 & \text{for each }k\in\I\setminus\{1\}.
        \end{dcases}
    \end{equation*}}

    Now \revision{we recall the assumption \eqref{spectral_gap_condition}. We denote for future use}
    \begin{equation*}
        \sigma_{\textup{gap}} = \inf\left\{\operatorname{Re}(\lambda)\ \middle|\ \lambda\in\operatorname{sp}\left(-\MB\right)\setminus\{0\}\right\}>0.
    \end{equation*}
    
    Recall that by virtue of Theorem \ref{thm:uniform_bound_dynamical_system}, $\|v(t)\|_{L^1}$ is uniformly bounded
    with respect to $t$.
    By summing the frequency estimates over the full range $\I$ and applying the Pythagorean or Parseval theorem, we deduce:
    \begin{equation*}
        \begin{split}
            \|v(t)\|_{L^2}^2 & \leq \frac{1}{\mu(\O)}\|v(t)\|_{L^1}^2 + \frac{\upe^{-2\sigma_{\textup{gap}} t}\|v(t)\|_{L^1}^2}{\|v(0)\|_{L^1}^2}\sum_{k\in\I\setminus\{1\}}\|v_k(0)\|_{L^2}^2 \\
            & = \frac{1}{\mu(\O)}\|v(t)\|_{L^1}^2 + \frac{\upe^{-2\sigma_{\textup{gap}} t}\|v(t)\|_{L^1}^2}{\|v(0)\|_{L^1}^2}\left(\|v(0)\|_{L^2}^2-\frac{\|v(0)\|_{L^1}^2}{\mu(\O)}\right) \\
            & = \|v(t)\|_{L^1}^2\left(\frac{1}{\mu(\O)}+\upe^{-2\sigma_{\textup{gap}} t}\left(\frac{\|v(0)\|_{L^2}^2}{\|v(0)\|_{L^1}^2}-\frac{1}{\mu(\O)}\right)\right).
        \end{split}
    \end{equation*}

    Now we focus on the nonlocal term in \eqref{reduced_system}. Conveniently, we decompose $v=v_1+w$ with 
    $w\in (1_E)^\perp$. Recall $f_1=\frac{1}{\sqrt{\mu(\O)}}$. 
    \begin{equation*}
        \begin{split}
            \int_{\O}v(t)^pc\upd\mu & = \int_{\O}\left(v_1(t)+w(t)\right)^p c\upd\mu \\
            & = \int_{\O}\left(\frac{\|v(t)\|_{L^1}}{\mu(\O)}+w(t)\right)^p c\upd\mu \\
            & = \int_{\O}\frac{\|v(t)\|_{L^1}^p}{\mu(\O)^p}\left(1+\frac{w(t)\mu(\O)}{\|v(t)\|_{L^1}}\right)^p c\upd\mu \\
            & = \frac{\|v(t)\|_{L^1}^p}{\mu(\O)^p}\left\|\left(1+\frac{w(t)\mu(\O)}{\|v(t)\|_{L^1}}\right)^p c\right\|_{L^1}.
        \end{split}
    \end{equation*}
    It follows that 
    \begin{equation*}
        \begin{split}
            \left|\int_{\O}v(t)^pc\upd\mu-\frac{\|v(t)\|_{L^1}^p}{\mu(\O)^p}\right| 
            & \leq \frac{\|v(t)\|_{L^1}^p}{\mu(\O)^p}\left|\left\|\left(1+\frac{w(t)\mu(\O)}{\|v(t)\|_{L^1}}\right)^p c\right\|_{L^1}-1\right| \\
            & = \frac{\|v(t)\|_{L^1}^p}{\mu(\O)^p}\left|\int_{\O}\left|1+\frac{w(t)\mu(\O)}{\|v(t)\|_{L^1}}\right|^p c\upd\mu-\int_{\O}c\upd\mu\right| \\
            & \leq \frac{\|v(t)\|_{L^1}^p}{\mu(\O)^p}\left|\int_{\O}\left(\left(1+\left|\frac{w(t)\mu(\O)}{\|v(t)\|_{L^1}}\right|\right)^p-1\right) c\upd\mu\right| \\
            & \leq \frac{\|v(t)\|_{L^1}^p}{\mu(\O)^p}\int_{\O}\left|\left(1+\left|\frac{w(t)\mu(\O)}{\|v(t)\|_{L^1}}\right|\right)^p-1 \right|c\upd\mu.
        \end{split}
    \end{equation*}
    
    Also,
    \begin{equation*}
        \left\|\frac{w(t)\mu(\O)}{\|v(t)\|_{L^1}}\right\|_{L^2} \leq \mu(\O)\upe^{-\sigma_{\textup{gap}} t}\sqrt{\frac{\|v(0)\|_{L^2}^2}{\|v(0)\|_{L^1}^2}-\frac{1}{\mu(\O)}}.
    \end{equation*}
    Let $t_0\geq 0$ such that the right-hand side above is smaller than $1$ for all $t\geq t_0$.
    Then the finiteness of the term $\int_{\O}\left|\left(1+\left|\frac{w(t)\mu(\O)}{\|v(t)\|_{L^1}}\right|\right)^p-1 \right|c\upd\mu$ 
    is obvious in finite dimensions (when Assumption \ref{ass:reduction_to_ODEs} holds true and $H=\C^{N'}$). 
    In infinite dimensions (when Assumption \ref{ass:reduction_to_ODEs} does not hold and $H=L^2(\O,\C)$), we use the supplementary
    assumption $p\leq 2$:
    \begin{equation*}
        \begin{split}
            \int_{\O}\left|\left(1+\left|\frac{w(t)\mu(\O)}{\|v(t)\|_{L^1}}\right|\right)^p-1 \right|c\upd\mu
            & \leq \int_{\O}\left|\left(1+\left|\frac{w(t)\mu(\O)}{\|v(t)\|_{L^1}}\right|\right)^2-1 \right|c\upd\mu \\
            & = \int_{\O}\left(\left|\frac{w(t)\mu(\O)}{\|v(t)\|_{L^1}}\right|^2+2\left|\frac{w(t)\mu(\O)}{\|v(t)\|_{L^1}}\right|\right)c\upd\mu.
        \end{split}
    \end{equation*}
   
    In both cases, for all $t\geq t_0$,
    \begin{equation}\label{eq:Lp_estimate_GAS_by_diagonalization}
        \int_{\O}v(t)^pc\upd\mu = \frac{\|v(t)\|_{L^1}^p}{\mu(\O)^p}\left(1+O(\upe^{-2\sigma_{\textup{gap}} t})\right).
    \end{equation}

    Consequently, $v$ satisfies for all $t\geq t_0$:
    \begin{equation*}
        \partial_t v(t) = \M v(t) +v(t)\left(1-\frac{\|v(t)\|_{L^1}^p}{\mu(\O)^p}\right)+v(t)O(\upe^{-2\sigma_{\textup{gap}} t}).
    \end{equation*}
    Taking the scalar product with $1_E$ and using the orthogonality between $1_E$ and $\M v$, we obtain an approximate
    equation for the total mass of $v$:
    \begin{equation*}
        \frac{\upd}{\upd t}\|v(t)\|_{L^1}=\|v(t)\|_{L^1}\left(1-\frac{\|v(t)\|_{L^1}^p}{\mu(\O)^p}\right)+O(\upe^{-2\sigma_{\textup{gap}} t}).
    \end{equation*}
    It follows from standard ordinary differential equation theory that $\|v(t)\|_{L^1}\to \revision{\mu(\O)}$ 
    \revision{exponentially} as $t\to+\infty$. 
    \revision{In view of \eqref{eq:Lp_estimate_GAS_by_diagonalization}, $\int_{\O}v(t)^p c\upd\mu\to 1$ as $t\to+\infty$, 
    exponentially as well. Thus there exists $C>0$ such that, for each $k\in\I$, $v_k$ satisfies
    \begin{equation*}
        |\partial_t v_k+\lambda_k v_k(t)|=v_k(t) O(\upe^{-Ct})
    \end{equation*}
    Therefore it follows from standard ordinary differential equation theory again that $v_1$ converges to a function
    positively proportional to $f_1$ and, for each $k\in\I\setminus\{1\}$, $v_k$ converges to $0$. Moreover, each
    $\partial_t v_k$ converges to $0$.}
    Finally, \revision{using once more the equation satisfied by $v$, the limit $v_\infty$ of $v$ is necessarily $1_E$
    and the convergence occurs in the topology of $E$}.
\end{proof}

\subsection{Conditional global asymptotic stability by gradient flow structure}

In this subsection we prove Theorem \ref{thm:GAS}, condition \ref{cond:GAS_gradient_flow}. The proof is a direct adaptation of
the one in \cite{Jabin_Liu_2017}. The main difference is that, in \cite{Jabin_Liu_2017}, classical parabolic
estimates made it possible to shorten significantly the conclusion of the argument. Here, we give a different argument
that does not use parabolic estimates and that is therefore more adapted to our abstract framework.

\begin{proof}
    Taking the scalar product between \eqref{reduced_system} and $\frac{c}{b}\partial_t v$, we find:
    \begin{equation*}
        \left\|\sqrt{\frac{c}{b}}\partial_t v\right\|_{L^2}^2 = \left(\partial_t v\middle|-\frac{c}{b}\M v\right) + \left(\sqrt{c}v_t| \sqrt{c}v\right)
        -\left(\sqrt{c}v_t | \sqrt{c}v\right)\left(\sqrt{c}v|\sqrt{c}v\right).
    \end{equation*}

    Define 
    \begin{equation*}
        F:w\in \dom_E(\MB)\mapsto \frac12\left(w\middle|-\frac{c}{b}\M w\right) + \frac12\|\sqrt{c}w\|_{L^2}^2-\frac14\|\sqrt{c}w\|_{L^2}^4.
    \end{equation*}
    This mapping is of class $\c^1$ on its domain.
    Using the self-adjointness of $-\frac{c}{b}\MB$, its G\^{a}teaux derivative at $w\in \dom_E(\MB)$
    in the direction $h\in \dom_E(\MB)$ such that $w+h\in \dom_E(\MB)$ is
    \begin{equation*}
        DF(w)(h) = \left(h\middle|-\frac{c}{b}\M w\right) + \left(\sqrt{c}w|\sqrt{c}h\right)-\left(\sqrt{c}w|\sqrt{c}w\right)\left(\sqrt{c}w|\sqrt{c}h\right).
    \end{equation*}

    In general $t\mapsto \partial_t v(t)$ might not be valued in $\dom_E(\MB)$. However, by assumption
    on the coefficients of $\MB$ and on $r$, $b$, $c$, $u_0$, then indeed $t\mapsto\partial_t v(t)$ is valued in 
    $\dom_E(\MB)$.

    It follows from
    \begin{equation*}
        0\leq \left\|\sqrt{\frac{c}{b}}\partial_t v\right\|_{L^2}^2 = DF(v)(\partial_t v)=\frac{\upd}{\upd t}F(v)
    \end{equation*}
    that $t\mapsto F(v(t))$ is a nondecreasing function. 
    
    Since the terms $\frac12\|\sqrt{c}v(t)\|_{L^2}^2$ and 
    $\frac14\|\sqrt{c}v(t)\|_{L^2}^4$ are bounded by virtue of Theorem \ref{thm:uniform_bound_dynamical_system},
    it follows that one of the following two possibilities is true:
    \begin{enumerate}
        \item $\left(v(t)\middle|-\frac{c}{b}\M v(t)\right)\to +\infty$ as $t\to+\infty$;
        \item $F(v(t))$ converges to a finite limit.
    \end{enumerate}

    Remark that taking the scalar product between \eqref{reduced_system} and $\frac{c}{b}v$ leads to the identity
    \begin{equation*}
        F(v(t))=\frac12\frac{\upd}{\upd t}\left(\left\|\sqrt{\frac{c}{b}}v(t)\right\|_{L^2}^2\right)+\frac14\|\sqrt{c}v(t)\|_{L^2}^4\quad\text{for all }t>0.
    \end{equation*}
    Thus if $F(v(t))\to+\infty$, then so does the derivative of the function 
    $t\mapsto\left\|\sqrt{\frac{c}{b}}v(t)\right\|_{L^2}^2$. In particular it is bounded from below by, say, $1$,
    in a neighborhood of $+\infty$. Consequently, the function $t\mapsto\left\|\sqrt{\frac{c}{b}}v(t)\right\|_{L^2}^2$ 
    diverges to $+\infty$. But this contradicts the \textit{a priori} bound of Theorem 
    \ref{thm:uniform_bound_dynamical_system}. Therefore $F(v(t))$ converges to a finite limit $a\in\R$.

    Since $F(v)$ also satisfies the identity
    \begin{equation*}
        F(v(t))=F(v_0)+\int_0^t \left\|\sqrt{\frac{c}{b}}\partial_t v(t')\right\|_{L^2}^2\upd t'\quad\text{for all }t>0,
    \end{equation*}
    we deduce that $\int_0^t \left\|\sqrt{\frac{c}{b}}\partial_t v\right\|_{L^2}^2\upd t$ converges to a finite limit
    as $t\to+\infty$. For any $\varepsilon>0$, by virtue of $\int_0^{+\infty} \left\|\sqrt{\frac{c}{b}}\partial_t v\right\|_{L^2}^2\upd t<+\infty$,
    there exists $t_\varepsilon>0$ such that $\left\|\sqrt{\frac{c}{b}}\partial_t v(t_\varepsilon)\right\|_{L^2}=\varepsilon$.

    Since $-\frac{c}{b}\MB$ has compact resolvent, by composition, $-\MB$ has compact resolvent as well, and 
    then the semigroup $(\T_t )_{t\geq 0}$ is immediately compact. By Lipschitz continuity of $\r$, 
    it follows that $v(t_\varepsilon)$ converges as $\varepsilon\to 0$, up to extraction of a subsequence \cite[Theorem 3.3.6]{Henry_1981}. 
    Similarly, by assumption on the coefficients of $\MB$ and on $r$, $b$, $c$, $u_0$, up to extraction
    of another subsequence, $\partial_t v(t_\varepsilon)$ converges.
    In view of $\left\|\sqrt{\frac{c}{b}}\partial_t v(t_\varepsilon)\right\|_{L^2}=\varepsilon$, the 
    limit of $v(t_\varepsilon)$ is a stationary solution of \eqref{reduced_system}.

    This stationary solution is either $0_E$ or $1_E$. 

    By assumption on $\frac{c}{b}\MB$, for all $w\in\dom_E(\MB)$, 
    \begin{equation*}
        F(w)\geq \frac12\|\sqrt{c}w\|_{L^2}^2-\frac14\|\sqrt{c}w\|_{L^2}^4 = \frac12\|\sqrt{c}w\|_{L^2}^2\left(1-\frac12\|\sqrt{c}w\|_{L^2}^2\right).
    \end{equation*}
    Hence, in the set $\left\{\|\sqrt{c}w\|_{L^2}<\sqrt{2}\right\}\setminus\{0\}$, the values of $F$ are positive. 
    Assume by contradiction that the limiting stationary solution as $\varepsilon\to 0$ is $0_E$. Then, on one hand, 
    by continuity of $t\mapsto F(v(t))$ (which follows from Property \ref{property:strong_continuity}, 
    despite possible discontinuities of $F$ due to the unboundedness of $\M$), $a=F(0)=0$. 
    On the other hand, $\|\sqrt{c}v(t_\varepsilon)\|_{L^2}<\sqrt{2}$ for sufficiently small values of $\varepsilon$. 
    If $v(t_\varepsilon)=0_E$ for some $\varepsilon$, then,
    by uniqueness of the solution of \eqref{reduced_system}, $v(t)=0$ for all $t\geq 0$, which contradicts $v_0\neq 0$.
    Thus $F(v(t_\varepsilon))>0$ for all $\varepsilon$ sufficiently small. 
    But then by monotonicity $F(v(t))\geq F(v(t_\varepsilon))>0$ for all
    $t\geq t_\varepsilon$, which contradicts $\lim_{t\to+\infty}F(v(t))=0$.
    
    Therefore the limiting stationary solution as $\varepsilon\to 0$ is $1_E$.

    By virtue of Theorem \ref{thm:LAS}, which can be applied due to the assumptions on the resolvent of $-\frac{c}{b}\MB$,
    $1_E$ is locally asymptotically stable. Let $\mathcal{V}\subset\dom_E(\MB)$ be an open neighborhood of $1_E$ in the topology of $H$ 
    such that $1_E$ attracts all trajectories that enter $\mathcal{V}$. Then, for a sufficiently small $\varepsilon>0$, 
    $v(t_\varepsilon)\in\mathcal{V}$, and this ends the proof.
\end{proof}

\section{Acknowledgements}
L. G. acknowledges support from the ANR via the project Indyana under grant agreement ANR-21-CE40-0008.
J. C. and L. G. acknowledge support from the ANR via the project Reach under grant agreement ANR-23-CE40-0023-01. They also
acknowledge support from the CNRS via the IRN ReaDiNet. 

This project has been ongoing for several years and, during this period of time, several colleagues have been solicited for 
insights and ideas and to help identify dead ends or formulate conjectures. The authors thank warmly all these colleagues, 
who will recognize themselves.

\revision{Finally, the authors thank two anonymous reviewers, whose comments helped improved significantly the manuscript.}

\section{Competing interests}
The authors declare no competing interests.

\bibliographystyle{abbrv}
\bibliography{ref-msmod}

\end{document}